\documentclass[leqno,10pt,a4paper]{article}
\pdfoutput=1
\usepackage{a4wide}
\usepackage[utf8]{inputenc}
\usepackage{amsmath,amsfonts,amssymb,amsthm}
\usepackage{hyperref}
\usepackage{calc}
\usepackage{enumitem}
\usepackage{bbm}
\usepackage{graphicx}
\usepackage{verbatim}
\usepackage{wrapfig}
\usepackage[stable]{footmisc}
\hypersetup{colorlinks=true,linkcolor=black,citecolor=black}

\usepackage[blocks]{authblk} 
\usepackage[usenames,dvipsnames]{color}

\newcommand{\simplex}{\sigma}

\DeclareMathOperator{\aff}{aff}

\newcommand{\thickdiag}{\Delta}
\newcommand{\thin}{\delta}
\newcommand{\thindiag}[2]{\ensuremath{\thin_{#1}(#2)}}
\newcommand{\calZ}{\mathcal{Z}}

\newcommand{\delprod}[2]{\ensuremath{{#1}^{ #2}_{\thickdiag}}}
\newcommand{\join}[2]{\ensuremath{{#1}^{\ast #2}}}

\renewcommand{\t}{\ensuremath{\widetilde}}
\newcommand{\gauss}[1]{\widetilde{#1}}

\newcommand{\scap}{\raisebox{-0.5ex}{\scalebox{1.8}{$\cdot$}}\,}%{\bullet}{\frown}
\DeclareMathOperator{\sign}{sgn}
\newcommand{\isign}[3]{\ensuremath{\sign_{#1}(#2,\dots,#3)}}
\newcommand{\linking}{\ell}

\newcommand{\primaticsign}{\epsilon^{\textup{PRIS}}_{r,k}}

\newcommand{\obs}{\mathfrak{o}}
\newcommand{\vko}[2]{\ensuremath{\obs(\delprod{#1}{#2})}}

\newcommand{\cell}{\sigma}
\newcommand{\I}{\mathbbm{1}}

\renewcommand{\phi}{\varphi}
\newcommand{\cocyc}{\varphi}

\newcommand{\R}{\ensuremath{\mathbbm{R}}}
\newcommand{\Z}{\ensuremath{\mathbbm{Z}}}

\newcommand{\boundary}{\ensuremath{\partial}}

\newcommand{\interior}[1]{\ensuremath{\mathring{#1}}}

\renewcommand{\epsilon}{\ensuremath{\varepsilon}}
\newcommand{\Matrix}[1]{\left(\vcenter{\xymatrix@=0pt{ #1} }\right)}

\newcommand{\ie}{{i.e.},\ }

\DeclareMathOperator{\supp}{supp}

\newcommand{\singset}{S}

\newcommand{\iso}{\ensuremath{\cong}}

\newcommand{\skel}[2]{\ensuremath{\textup{skel}_{#1}(#2 )}}

\newcommand{\define}[1]{\textbf{\boldmath{#1}}}

\newcommand{\sym}{\ensuremath{\mathfrak{S}}}

\newcommand{\boldpi}{\boldsymbol{\pi}}

\newcommand{\boldx}{\boldsymbol{x}}

\newtheorem*{theorem*}{Theorem}
\newtheorem{theorem}{Theorem}%[section]
\newtheorem{proposition}[theorem]{Proposition}
\newtheorem{lemma}[theorem]{Lemma}

\newtheorem{corollary}[theorem]{Corollary}
\newtheorem{conjecture}[theorem]{Conjecture}

\theoremstyle{definition}
\newtheorem{definition}[theorem]{Definition}
\newtheorem{question}[theorem]{Question}
\newtheorem{problem}[theorem]{Problem}
\newtheorem{remark}[theorem]{Remark}
\newtheorem{observation}[theorem]{Observation}
\newtheorem{example}[theorem]{Example}
\newtheorem{remarks}[theorem]{Remarks}
\newtheorem*{strategy*}{Strategy}

\begin{document}

%\maketitle

\begin{center}

\renewcommand{\thefootnote}{\fnsymbol{footnote}}

{\Large Eliminating Higher-Multiplicity Intersections, I.}
\vskip 0.1in
{\Large A Whitney Trick for Tverberg-Type Problems}\footnotemark[1]

\vskip 0.15in

\begin{minipage}{0.25\textwidth}
\centering
{\large \sc Isaac Mabillard}%\footnotemark[2] 
\\
\texttt{\href{mailto:imabillard@ist.ac.at}{imabillard@ist.ac.at}}
\end{minipage}
\hskip 0.1in
and
\hskip 0.1in
\begin{minipage}{0.2\textwidth} % Compensé par la minipage suivante
\centering
{\large \sc Uli Wagner}%\footnotemark[2] 
\\ \texttt{\href{mailto:uli@ist.ac.at}{uli@ist.ac.at}}
\end{minipage}
\begin{minipage}{0.05\textwidth} %% Pour que les deux noms soient correctement alignés
\mbox{ }
\end{minipage}
\footnotetext[1]{Research supported by the Swiss National Science Foundation (Project SNSF-PP00P2-138948). 
An extended abstract of this paper appeared in Proc. 30th Annual Symposium on Computational Geometry (SoCG 2014) \cite{MabillardWagner:TverbergWhitney-2014}.}
%\footnotetext[2]{Research supported by the Swiss National Science Foundation (Project SNSF-PP00P2-138948).}

\vskip 0.15in

{\it IST Austria, Am Campus 1, 3400 Klosterneuburg, Austria \hskip 0.4in}

\vskip 0.1in

{\today}

\end{center}

\begin{abstract}
Motivated by \emph{topological Tverberg-type problems} in topological combinatorics and by classical results about embeddings (maps without double points), we study the question whether a finite %$m$-dimensional 
simplicial complex $K$ can be mapped into $\R^d$ without  triple, quadruple, or, more generally, \emph{$r$-fold points} (image points with at least $r$ distinct preimages), for a given multiplicity $r\geq 2$. In particular, we are interested in maps $f\colon K\to \R^d$ that have no \emph{$r$-Tverberg points}, i.e., no $r$-fold points %that are \emph{non-local}, i.e., 
with preimages in $r$ \emph{pairwise disjoint} simplices of $K$, and we seek necessary and sufficient conditions for the existence of such maps.

We present higher-multiplicity analogues of several classical results for embeddings, in particular of the completeness of the \emph{Van Kampen obstruction} for embeddability of $k$-dimensional complexes into $\R^{2k}$, $k\geq 3$. Specifically, we show that under suitable restrictions on the dimensions (viz., if $\dim K=(r-1)k$ and $d=rk$ for some $k\geq 3$), a well-known \emph{deleted product criterion} (\emph{DPC}) is not only necessary but also sufficient for the existence of maps without $r$-Tverberg points. Our main technical tool is a higher-multiplicity version of the classical \emph{Whitney trick}, by which pairs of isolated $r$-fold points of \emph{opposite sign} can be eliminated by local modifications of the map, assuming codimension $d-\dim K\geq 3$.

An important guiding idea for our work was that sufficiency of the DPC, together with an old result of \"Ozaydin on the existence of equivariant maps, might yield an approach to disproving the remaining open cases of the long-standing \emph{topological Tverberg conjecture}, i.e., to construct maps from the $N$-simplex $\simplex^N$ to $\R^d$ without $r$-Tverberg points when $r$ \emph{not a prime power} %, $d$ is a suitable dimension, 
and $N=(d+1)(r-1)$. Unfortunately, our proof of the sufficiency of the DPC requires codimension $d-\dim K\geq 3$, which is not satisfied for $K=\simplex^N$.

In a recent breakthrough, Frick found an extremely elegant way to overcome this ``codimension $3$ obstacle'' and to construct the first counterexamples to the topological Tverberg conjecture for all parameters $(d,r)$ with $d\geq 3r+1$ and $r$ not a prime power,
%i.e., maps $f\colon \simplex^{(3r+2)(r-1)}\to \R^{3r+1}$ without $r$-Tverberg points,
by a clever reduction (using the \emph{constraints method} of Blagojevi\'c--Frick--Ziegler) to a suitable lower-dimensional skeleton%$\smash{K=\skel{3r-3}{\simplex^{(3r+2)(r-1)}}}$
, for which the codimension $3$ restriction is satisfied and maps without $r$-Tverberg points exist by \"Ozaydin's result and sufficiency of the DPC.
%he shows that $f\colon \simplex^N\to \R^d$ without $r$-Tverberg points exists provided there is a map $g\colon K\to \R^{d-1}$ with the same property, where $K$ is the $m$-skeleton of $\simplex^N$, $m=3(r-1)$; for this $K$, the codimension restrictions are satisfied, and $g$ exists by \"Ozaydin's result and ours.

Here, we present a different construction (which does not use the constraint method) that
yields counterexamples for $d\geq 3r$, $r$ not a prime power. 
\end{abstract}

\clearpage

%%% ToC: LIST SECTIONS AND SUBSECTIONS, BUT NOT SUBSUBSECTIONS
\setcounter{tocdepth}{2}
\tableofcontents
%%%%%%%%%%%%%%%%%%%%%%%%%%%%%%%%%%%%%%

\section{Introduction}
\label{sec:introduction}

Let $K$ be a finite simplicial complex\footnote{Throughout this paper, we will (ab)use the same notation for a simplicial complex $K$ (a collection of simplices) and its \emph{underlying polyhedron}, relying on context to distinguish between the two when necessary.} and 
let $f\colon K\to \R^d$ be a continuous map. Given an integer parameter $r\geq 2$, we say that $y\in \R^d$ is an \define{$r$-fold point} or \define{$r$-intersection point} of $f$ if $|f^{-1}(y)|\geq r$, i.e., if there are $r$ pairwise distinct points $x_1,\ldots,x_r\in K$ such that $f(x_1)=\ldots=f(x_r)=y$.

Motivated by \emph{topological Tverberg-type problems} (see below), an important topic in topological combinatorics, we are particularly interested in the following special type of $r$-fold points. For a point $x$ in $K$, we define its \define{support} $\supp(x)$ 
as the smallest simplex\footnote{All simplices are considered closed, unless indicated otherwise.} of $K$ that contains $x$ in its relative interior. %We say that two points $x,x'$ in $K$ are \define{distant} or \define{close} depending on whether their supports are disjoint or not.
We say that $y\in \R^d$ is a \define{non-local $r$-fold point} or \define{$r$-Tverberg point} of a map $f\colon K\to \R^d$ if it has $r$ preimages with \emph{pairwise disjoint supports}, i.e., $y\in f(\sigma_1)\cap \ldots \cap f(\sigma_r)$ for pairwise disjoint simplices 
$\sigma_1,\ldots,\sigma_r$. Thus, when focussing on $r$-Tverberg points, we ignore \define{local $r$-fold points} that occur between images of simplices some of which share some vertices; we stress that being an $r$-Tverberg point depends on the actual simplicial complex $K$ (the triangulation), not just on the underlying polyhedron.

%\uli{TODO: Briefly discuss connections to immersions, either here or later in the introduction.}

The most basic case is that of (topological) \emph{embeddings}, i.e., maps without double points.\footnote{Since $K$ is compact and $\R^d$ is Hausdorff, a continuous map $f\colon K\rightarrow \R^d$ is an embedding iff it is injective.} Finding conditions for a simplicial complex $K$ to be embeddable into $\R^d$ --- a higher-dimensional generalization of graph planarity --- is a classical problem in geometric topology (see, e.g., \cite{RepovsSkopenkov:NewResultsEmbeddingsManifoldsPolyhedra-1999,Skopenkov:EmbeddingKnottingManifoldsEuclideanSpaces-2008} for surveys) and has recently also become the subject of systematic study from a viewpoint of algorithms and computational complexity (see, e.g., \cite{MatousekTancerWagner:HardnessEmbeddings-2011,MatousekSedgwickTancerWagner:EmbeddabilityS3Decidable-2013,Cadek:Algorithmic-solvability-of-the-lifting-extension-problem-2013}).

Generalizing classical results about embeddings, we are interested in necessary and sufficient conditions for the existence of maps without $r$-Tverberg points, and in techniques that allow us to eliminate $r$-fold points by local modifications of the map. In the present paper, we establish such results in the ``critical case'' $d=\frac{r-1}{r}\dim K$ (the smallest dimension $d$ for which a map in general position can have $r$-fold points), assuming codimension $d-\dim K\geq 3$; see Theorems~\ref{thm:vKcomplete} and~\ref{thm_whitney_trick_extended} below.

\subsection{Topological Tverberg-Type Problems}
The classical geometric \emph{Tverberg theorem}~\cite{Tverberg:A-generalization-of-Radons-theorem-1966}, a cornerstone of convex geometry, can be rephrased as saying that if $N=(d+1)(r-1)$ then any \emph{affine} map from the $N$-dimensional simplex $\simplex^N$ to $\R^d$ has an $r$-Tverberg point.
Bajmoczy and B\'ar\'any~\cite{Bajmoczy:On-a-common-generalization-of-Borsuks-and-Radons-1979} and Tverberg~\cite[Problem~84]{Gruber:Problems-in-geometric-convexity-1979} raised the question whether this remains true for arbitrary continuous maps:

\begin{conjecture}[\textbf{Topological Tverberg Conjecture}]
\label{conj:Tverberg}
Let $r\geq 2$, $d \geq 1$, and $N=(d+1)(r-1)$. Then every continuous map $f\colon \simplex^N\to \R^d$ has an $r$-Tverberg point.
\end{conjecture}
This was proved by Bajmoczy and B\'ar\'any~\cite{Bajmoczy:On-a-common-generalization-of-Borsuks-and-Radons-1979} for $r=2$, by B\'ar\'any, Shlosman, and Sz\H{u}cs~\cite{Barany:On-a-topological-generalization-of-a-theorem-of-Tverberg-1981} for all primes $r$, and by \"Ozaydin~\cite{Ozaydin:Equivariant-maps-for-the-symmetric-group-1987} for prime powers $r$,\footnote{Further proofs of the prime power case were given by Volovikov~\cite{Volovikov:On-a-topological-generalization-of-Tverbergs-theorem-1996}, \v{Z}ivaljevi\'c~\cite{Zivaljevic:UserGuideEquivariantTopologyCombinatorics2-98}, and Sarkaria~\cite{Sarkaria:Tverberg-partitions-and-Borsuk-Ulam-theorems-2000}.} 
 but the case of arbitrary $r$ has been a long-standing open problem, considered to be one of the most challenging in the area \cite[p.~154]{Matousek:BorsukUlam-2003}.

There are numerous close relatives and other variants of (topological) Tverberg-type problems and results, e.g., the \emph{Colored Tverberg Problem} \cite{Barany:On-the-number-of-halving-planes-1990,Barany:A-colored-version-of-Tverbergs-theorem-1992,Zivaljevic:The-colored-Tverbergs-problem-and-complexes-of-injective-functions-1992,Zivaljevic:UserGuideEquivariantTopologyCombinatorics2-98,Blagojevic:Optimal-bounds-for-the-colored-Tverberg-problem-2009} and generalized \emph{Van Kampen--Flores-type results} \cite{Sarkaria:A-generalized-van-Kampen-Flores-theorem-1991,Volovikov:On-the-van-Kampen-Flores-theorem-1996}.
\medskip

Here, we consider the following general problem:

\begin{problem}
\label{prob:Tverberg}
Given a finite simplicial complex $K$ and parameters $r$ and $d$, decide whether there exists a map $f\colon K\to \R^d$ without $r$-Tverberg points. 
\end{problem}
In particular, we are interested in methods for proving the existence of such maps, (i.e., for showing that $K$ does \emph{not} satisfy a topological Tverberg-type theorem with parameters $r$ and $d$).
\begin{remark}\label{rem:wlog-PL-gen-pos}
For Problem~\ref{prob:Tverberg}, it suffices to consider maps $f\colon K\to \R^d$ that are \emph{piecewise-linear}\footnote{Recall that $f$ is PL if there is some subdivision $K'$ of $K$ such that $f|_\sigma$ is affine for each simplex $\sigma$ of $K'$.} (\emph{PL}), since every continuous map $g\colon K\to \R^d$ can be approximated arbitrarily closely by a PL map, and if $g$ has no $r$-Tverberg points, then the same holds for any map sufficiently close to $g$. 

Moreover, if $\dim K<\frac{r-1}{r}d$ or, more generally, if the \emph{deleted product} of $K$ (see below) satisfies
$\dim \delprod{K}{r}<d(r-1)$, then a simple codimension count shows that a PL map $f\colon K\to \R^d$ in general position has no 
$r$-Tverberg points, so the problem is trivial. In the present paper, we focus on the first nontrivial case $\dim \delprod{K}{r}=d(r-1)$, for which a PL map $f\colon K\to \R^d$ in general position has a finite number of $r$-Tverberg points. 
\end{remark}

\paragraph{\boldmath{The Deleted Product Criterion.}}
There is a well-known \emph{necessary condition} for the existence of maps without Tverberg points, formulated in terms of the
(combinatorial)  \define{deleted $r$-fold product}\footnote{Some authors prefer to work with \emph{deleted joins} (which are again simplicial complexes) instead of deleted products as configuration spaces for Tverberg-type problems. However, it is known that deleted products provide necessary conditions that are at least as strong as those provided by deleted joins; see, e.g.,  \cite[Sec.~3.3]{Matschke:Equivariant-topology-methods-in-discrete-geometry-2011}. For further background on the broader \emph{configuration space/test map} framework, see, e.g., \cite[Ch.~6]{Matousek:BorsukUlam-2003} or \cite{Zivaljevic:UserGuideEquivariantTopologyCombinatorics-96,Zivaljevic:UserGuideEquivariantTopologyCombinatorics2-98}.} 
of a complex $K$, which is defined as
 $$
\delprod Kr := \{(x_1,\ldots,x_r)\in K^r \mid  \supp(x_i)\cap \supp(x_j)=\emptyset \textup{ for } 1\leq i < j \leq r\}.
$$ 
The deleted product $\delprod{K}{r}$ is a regular polyhedral cell complex (a subcomplex of the cartesian product $K^r$), whose cells are products $\sigma_1\times \dots\times \sigma_r$ of pairwise disjoint simplices of $K$.

\begin{lemma}[\textbf{Necessity of the Deleted Product Criterion}]
\label{lem:delprod-necessary} 
Let $K$ be a finite simplicial complex, and let $d\geq 1$ and $r\geq 2$ be integers. If there exists a map $f:K \rightarrow \R^d$ 
without $r$-Tverberg point then there exists an equivariant map\footnote{Here and in what follows, if $X$ and $Y$ are spaces on which a finite group $G$ acts (all group actions will be from the right) then we will use the notation $F\colon X\to_G Y$ for maps that are \define{equivariant}, i.e., that commute with the group actions, $F(x \cdot g)= F(x)\cdot g$ for all $x\in X$ and $g\in G$).} 
\begin{equation*}
\gauss{f} \colon \delprod Kr \rightarrow_{\sym_r} S^{d(r-1)-1},
\end{equation*}
where $S^{d(r-1)-1}=\big\{(y_1,\ldots,y_r)\in (\R^d)^r\mid \textstyle \sum_{i=1}^r y_i=0, \sum_{i=1}^r \|y_i\|_2^2=1\big\}$,
and the symmetric group $\sym_r$ acts on both spaces by permuting components.\footnote{We remark that the action of $\sym_r$ is free on $\delprod{K}{r}$ for all $r$, but not free on $S^{d(r-1)-1}$ unless $r$ is a prime.}
\end{lemma}
We briefly recall the standard proof, which uses several notions that we will need later.
\begin{proof}
Given $f \colon K \rightarrow \R^d$, one gets a map $f^r\colon \delprod Kr \to (\R^d)^{r}$ by %applying $f$ componentwise,
$f^r(x_1, \ldots, x_r):= (f(x_1), \ldots f(x_r))$. The map $f$ has no $r$-Tverberg point iff $f^r$ avoids the \define{thin diagonal}
\begin{equation}
\label{eq:thindiag}
\thindiag{r}{\R^d} := \{ (y, \ldots, y) \mid  y \in \R^d \} \subset (\R^d)^r.
\end{equation}
Moreover, $S^{d(r-1)-1}$ is the unit sphere in the orthogonal complement  $\thindiag{r}{\R^d}^\bot \cong \R^{d(r-1)}$, and
there is a straightforward homotopy equivalence%
\footnote{First orthogonally project $(\R^d)^r\setminus \thindiag{r}{\R^d}$ onto $\thindiag{r}{\R^d}^\bot \setminus \{0\}$, and then radially retract the latter to $S^{d(r-1)-1}$. Concretely, 
$\rho=\mu\circ \nu$, given by $\nu(y_1,\ldots,y_r)=(\bar{y}_1,\ldots,\bar{y}_r)$, where $\bar{y}_j=y_j-\sum_{i=1}^r y_i$, $1\leq j\leq r$, and $\mu(\bar{y}_1,\ldots,\bar{y}_r)=(\bar{y}_1,\ldots,\bar{y}_r)/(\sum_{i=1}^r \|\bar{y}_i\|_2^2)$.}
$\rho\colon (\R^d)^r \setminus \thindiag{r}{\R^d} \simeq S^{d(r-1)-1}.$ 
Both $f^r$ and $\rho$ are equivariant hence so is their composition
\begin{equation}
\label{eq:ftilde}
\gauss{f}:=\rho\circ f^r\colon \delprod{K}{r}\to_{\sym_r} S^{d(r-1)-1}.\qedhere
\end{equation}
\end{proof}

\paragraph{The $\boldsymbol{r}$-fold Van Kampen Obstruction.} Lemma~\ref{lem:delprod-necessary}  is an important tool for proving topological Tverberg-type results.
Moreover, in many interesting cases, the existence of an equivariant map 
\begin{equation}
\label{eq:equimap}
F\colon \delprod Kr \rightarrow_{\sym_r} S^{d(r-1)-1}
\end{equation}
can be decided using \emph{equivariant obstruction theory} (for which the standard reference is \cite[Sec.~II.3]{Dieck:Transformation-Groups-1987}). In particular, in the %first nontrivial 
case $\dim \delprod{K}{r} = d(r-1)$ there is a single $d(r-1)$-dimensional equivariant cohomology class $\vko{K}{r}$ defined on $\delprod{K}{r}$ that yields a complete criterion (see Sec.~\ref{sec_equivariant_obs_theory}): 
\begin{lemma}
\label{lem:gen_VKobstruction}
Suppose $\dim \delprod{K}{r} \leq d(r-1)$. Then there exists an equivariant map 
$F\colon \delprod{K}{r}\to_{\sym_r} S^{d(r-1)}$ if and only if $\vko{K}{r}=0$.
\end{lemma}
If $r=2$, $\dim K=m$, and $d=2m$ then $\vko{K}{2}$ is the classical \define{Van Kampen obstruction} to embeddability of $K$ into $\R^{2m}$
(\cite{vanKampen:KomplexeInEuklidischenRaeumen-1932,Shapiro:FirstObstruction-1957,Wu:TheoryImbeddingImmersionIsotopy-1965}; see also \cite{Melikhov} for a recent in-depth treatment and further references). Correspondingly, we call $\vko{K}{r}$ the \define{$r$-fold Van Kampen obstruction}.

However, there is a caveat: Vanishing of the $r$-fold Van Kampen obstruction implies the existence of an equivariant map $F$ as in \eqref{eq:equimap}, 
but it does not imply that $F$ is of the form $\gauss{f}$ as in \eqref{eq:ftilde}, i.e., induced by a map $f\colon K\to \R^d$ without $r$-Tverberg points; thus, if $\vko{K}{r}=0$ then it is unclear whether the deleted product criterion is incomplete and one needs more refined arguments to show that such a map $f$ does not exist, or whether $f$ does exist and a Tverberg-type theorem for $K$ is simply not true. A particularly pertinent example of this kind is a result of \"Ozaydin~\cite[Theorem~4.2]{Ozaydin:Equivariant-maps-for-the-symmetric-group-1987} (see Theorem~\ref{thm:ozaydin} below), which was a major inspiration for our work.

\paragraph{Sufficiency of the deleted product criterion.} This raises the question whether there exists a converse to Lemma~\ref{lem:delprod-necessary}, at least under some suitable additional hypotheses.

For the classical case $r=2$, this is known to be the case, under suitable restrictions on the dimensions. A fundamental result of this type
was first stated by Van Kampen~\cite{vanKampen:KomplexeInEuklidischenRaeumen-1932} (albeit with a lacuna in the proof \cite{van-Kampen:Berichtung-zu:Komplexe-in-euklidischen-Raumen.-1932}), and complete proofs were later provided by Shapiro~\cite{Shapiro:FirstObstruction-1957} and by Wu \cite{Wu:TheoryImbeddingImmersionIsotopy-1965}. It is convenient for us to separate the statement in two parts: a first one concerning maps without $2$-Tverberg points (also called \define{almost-embeddings}), and a second one concerning embeddings.
 
\begin{theorem}[\textbf{Van Kampen--Shapiro--Wu}]
\label{thm:vKcomplete}
Let $K$ be a simplicial complex, $m:=\dim K\geq 3$.
\begin{enumerate}[label=\textup{(VK\arabic*)}]
\item \label{VK1}
There exists an almost-embedding $f\colon K\to \R^{2m}$ if and only if there exists an equivariant map
%\footnote{In this case, 
%$\sym_2=\Z_2$ acts on the sphere by antipodality, and the map
%$\gauss{f}$ induced by an almost embedding $f$ as in \eqref{eq:ftilde} is particularly simple: 
%It is given by $\gauss{f}(x,y):=\frac{f(x)-f(y)}{\|f(x)-f(y)\|}$.}
 $\delprod{K}{2}\to_{\sym_2} S^{2m-1}$.
\item \label{VK2}
If there an almost-embedding $f\colon K\to \R^{2m}$ then there exists an embedding $g\colon K\hookrightarrow \R^{2m}$;
moreover, $g$ can be taken to be piecewise-linear.%\footnote{For general $m$ and $d$, the implications almost-embeddability $\Rightarrow$ embeddability $\Rightarrow$ PL embeddability of $m$-dimensional complexes into $\R^d$ are strict.}
\end{enumerate}
\end{theorem}

Our main result is a generalization of \ref{VK1} to $r$-Tverberg points.\footnote{Generalizing  \ref{VK2} to $r$-fold points that may be \define{local}, i.e., whose preimages are not pairwise disjoint, turns out to be more subtle; we plan to treat this in a follow-up paper.}
%\uli{ARBITRARY MULTIPE POINTS} in Section~\ref{subsec:r-intersections-general-position}, see in particular Theorem~\ref{thm_geom}.

\begin{theorem}[\textbf{Sufficiency of the Deleted Product Criterion for Tverberg Points%Generalized Van Kampen--Shapiro--Wu Theorem for Tverberg Points
}]
\label{thm:VK-Tverberg-complete} 
%\label{uber_theorem}  \label{thm_purely_comb} \label{thm:vKShWu-Tverberg}
Suppose $r\geq 2$,  $(r-1)d=rm$, and $d-m\geq 3$. If $K$ is a finite $m$-dimensional simplicial complex, then there exists a map $f:K \to \R^d$ without $r$-Tverberg point iff  there exists an equivariant map $F\colon \delprod{K}{r} \to_{\sym_{r}} S^{d(r-1)-1}$ (equivalently, iff $\vko{K}{r}=0$).
\end{theorem}

The proof of Theorem~\ref{thm:VK-Tverberg-complete} will be presented in Section~\ref{sec:DeletedProductCriterionTverberg} 
(see the beginning of that section for an overview).
The proof is structured along the lines of the classical proof of \ref{VK1} (see \cite{Freedman:van-Kampens-embedding-obstruction-is-incomplete-for-2-complexes-in-bf-R4-1994} for a very accessible account of the latter) and based on appropriate higher-multiplicity generalizations of the corresponding tools, in particular \emph{$r$-fold Van Kampen finger moves} (Section~\ref{sec_van_kampen_fingers_move}) and an \emph{$r$-fold Whitney trick} (Theorem~\ref{thm_whitney_trick_extended}). 
%These will be introduced and explained in more detail below.
%\uli{Alternatively, maybe already give an outline of the proof at this point: equivariant obstruction theory gives a generalized  \emph{Van Kampen obstruction} such that the equivariant map exists if and only if the obstruction is zero; this can be represented geometrically via an \emph{intersection number cocycle} $\varphi$ associated with an arbitrary PL map $f\colon K\to \R^d$ in general position; if the obstruction is zero, generalized \emph{$r$-fold Van Kampen finger moves} allow to modify the map by local moves and obtain a new map $f\colon K\to \R^d$ in general position such that $\varphi_g=0$ as a cocycle, which means that for every $r$-tuple of pairwise disjoint simplices of $K$, the intersection of their images consists of 
%pairs $r$-Tverberg points of opposite sign. Finally, use the higher-multiplicity Whitney trick to eliminate these pairs of opposite sign.}

\begin{remarks}
\begin{enumerate}
\item The assumption that the map $F$ is equivariant with respect to the action of the full symmetric group $\sym_r$ (and not just some subgroup $H\leq \sym_r$) will be important when applying the $r$-fold Van Kampen finger moves; see Section~\ref{sec_van_kampen_fingers_move} (Remark~\ref{rem:equivariance-sym_r-finger-moves}).
\item The \emph{codimension restriction} $d-m\geq 3$ is crucial for many steps of the proof of Theorem~\ref{thm:VK-Tverberg-complete}. In the classical case of embeddings, it is known that Theorem~\ref{thm:vKcomplete} fails for $m=2$ (see \cite{Freedman:van-Kampens-embedding-obstruction-is-incomplete-for-2-complexes-in-bf-R4-1994}) but holds for $m=1$ (embeddings of graphs in the plane), even under slightly weaker assumptions; the latter fact is equivalent to the \emph{Hanani--Tutte Theorem} \cite{Hanani:UnplattbareKurven-1934,Tutte:TowardATheoryOfCrossingNumbers-1970}. It would be interesting to know if either of these facts generalize to higher multiplicities; see Section~\ref{sec:open_problems} for a more detailed discussion of these and related open questions.
\item For embeddings, there is a far-reaching generalization of Theorem~\ref{thm:vKcomplete}: The \emph{Haefliger--Weber Theorem}~\cite{Haefliger:Plongements-de-varietes-dans-le-domaine-stable-1964,Weber:Plongements-de-polyhedres-dans-le-domaine-metastable-1967} (see also \cite{Skopenkov:EmbeddingKnottingManifoldsEuclideanSpaces-2008} for a modern survey and extensions) guarantees that in the so-called \emph{metastable range} $d\geq 3(m+1)/2$, an $m$-dimensional complex $K$ embeds (piecewise-linearly) into $\R^d$ if and only if there is an equivariant map $\delprod{K}{2}\to_{\sym_2} S^{d-1}$. In a subsequent paper, we plan to present a generalization of this to $r$-Tverberg points, which works in a corresponding \emph{$r$-metastable range} $rd \geq (r+1)m+3$.
\end{enumerate}
 \end{remarks}

Vanishing of the generalized Van Kampen obstruction amounts to the solvability of a certain system of inhomogeneous linear equations over the integers (see Section~\ref{sec_van_kampen_fingers_move}).
As a consequence, we have the following:
\begin{corollary}
\label{cor:VKO-computable}
There is an algorithm which, under the assumptions of Theorem~\ref{thm:VK-Tverberg-complete} , decides whether a given input complex $K$ admits a map into $\R^d$ without $r$-Tverberg points. If the parameters $r$ and $m$ are fixed, the algorithm runs in polynomial time in the size (number of simplices) of $K$.
\end{corollary}

\paragraph{\"Ozaydin's and Frick's work: counterexamples to the topological Tverberg conjecture.}
As mentioned above (see also the discussion in \cite%[{Motivation \& Future Work}]
{MabillardWagner:TverbergWhitney-2014}), an important motivation for our work was the following result by \"Ozaydin~\cite[Theorem~4.2]{Ozaydin:Equivariant-maps-for-the-symmetric-group-1987}. For every $n \geq 0$, let $E^n_{\sym_{r}}$ denote an $n$-dimensional, $(n-1)$-connected free $\sym_r$-cell complex. Such complexes exist for all $n\geq 0$: e.g., one can take the $(n+1)$-fold join $E^n_{\sym_{r}}=(\sym_{r})^{\ast (n+1)}$, where $\sym_r$ is considered as a $0$-dimensional complex and acts on itself by right multiplication. They have the universal property that every free $\sym_r$-cell complex $X$ of dimension $\dim X\leq n$ maps equivariantly  into $E^n_{\sym_{r}}$ (see \cite[Sec.~6.2]{Matousek:BorsukUlam-2003}).
\begin{theorem}[\textbf{\"Ozaydin}]
\label{thm:ozaydin} Let $d\geq 1$ and $r\geq 2$. There exists an equivariant map 
$$F \colon E^{d(r-1)}_{\sym_r} \to_{\sym_r} S^{d(r-1)-1}$$ 
if and only if $r$ is not a prime power.
\end{theorem}
Hence, by the universal property of $E^{d(r-1)}_{\sym_r}$, there exists an equivariant map
\begin{equation}
\label{eq:cor_ozaydin}
F\colon \delprod{K}{r} \to_{\sym_r} S^{d(r-1)-1}
\end{equation}
whenever $r$ is not a prime power and $K$ is a simplicial complex such that $\dim \delprod{K}{r} \leq d(r-1)$; in particular, this applies if $\dim K\leq \frac{r-1}{r}d$ or if $K=\simplex^N$.\footnote{On the other hand, B\'ar\'any et al.~\cite[Lemma~1]{Barany:On-a-topological-generalization-of-a-theorem-of-Tverberg-1981} showed that $\delprod{(\simplex^n)}{r}$ is $(n-r)$-connected for $1\leq r \leq n-1$, hence $\delprod{(\simplex^N)}{r}$ is of the type $E^{d(r-1)}_{\sym_r}$. Thus, for prime powers $r$, there is no equivariant map $F\colon \delprod{(\simplex^N)}{r}\to_{\sym_r} S^{d(r-1)-1}$, by Theorem~\ref{thm:ozaydin} (and hence that the topological Tverberg conjecture holds in this case).}

Inspired by this and by the analogy with the classical theorems on embeddability, one of the guiding ideas for our work was that combining \"Ozaydin's result and sufficiency of the deleted product criterion for $r$-Tverberg points might yield an approach to constructing counterexamples to the topological Tverberg conjecture if $r$ is not a prime power. 

Unfortunately, our proof of  Theorem~\ref{thm:VK-Tverberg-complete} requires codimension $d-\dim K\geq 3$, which is not satisfied for $K=\simplex^N$ (one can replace $\simplex^N$ by its $d$-skeleton $\smash{\skel{d}{\simplex^N}}$ without loss of generality, but the problem persists). 
%This appeared to be a serious obstacle to applying our methods to refute the conjecture.\footnote{We knew \cite{Berlin} that it would be possible to construct counterexamples to the conjecture if one could construct a map $g\colon \simplex^{(3r+1)(r-1)}\to \R^{rm}$ as in Proposition~\ref{prop:special-prismatic-map} below, with all $r$-Tverberg points occurring in pairs of opposite sign and with preimages only in simplices of dimension $3(r-1)$, but we did not know how to construct such a map at the time (cf. I.~Mabillard. \emph{Eliminating Tverberg points: An analogue of the Whitney trick}. Talk in the \emph{Discrete Geometry Seminar}, Freie Universit\"at Berlin, January 15, 2015).}

In a recent breakthrough, following the announcement of our work in the extended abstract~\cite{MabillardWagner:TverbergWhitney-2014}% and in a number of talks (e.g., \cite{CPH,Berlin})
, Frick~\cite{Frick:Counterexamples-to-the-topological-Tverberg-conjecture-2015} found a very elegant way to overcome this {\emph{codimension $3$ obstacle} and to construct the first counterexamples to the topological Tverberg conjecture. Specifically, Frick proves that for every $r\geq 6$ that is not a prime power, there exists a map $f\colon \simplex^M \to \R^{3r+1}$ without $r$-Tverberg points, where $M=(3r+2)(r-1)$; in particular, there exists a map $\sigma^{100} \rightarrow \R^{19}$ without $6$-Tverberg point. It is known that this implies that there are counterexamples for all $d\geq 3r+1$, see~\cite[Proposition~2.5]{Longueville:Notes-on-the-topological-Tverberg-theorem-2002}. 

Frick's argument exemplifies the \emph{constraint method} of Blagojevi\'{c}--Frick--Ziegler~\cite{Blagojevic:Tverberg-plus-constraints-2014} and builds the counterexample $f\colon \simplex^M\to \R^{3r+1}$ from a map $\smash{g\colon \skel{3(r-1)}{\simplex^M} \to \R^{3r}}$ without $r$-Tverberg points, where the existence of $g$ follows from \"Ozaydin's result (Theorem~\ref{thm:ozaydin}) and ours (Theorem~\ref{thm:VK-Tverberg-complete}).

Here, we present a different construction (which does not use the constraint method) that  yields counterexamples in dimension $d=3r$; this seems to be the natural limit for counterexamples constructed using the $r$-fold Whitney trick, due to the codimension 3 requirement for the latter.

\begin{theorem}
\label{thm:counterexamples}
Suppose $r\geq 6$ is not a prime power and let $N=(3r+1)(r-1)$. Then there exists a map $f\colon \simplex^{N}\to \R^{3r}$ without $r$-Tverberg points.
\end{theorem}
The smallest counterexample obtained in this way is a map $f\colon \simplex^{95}\to \R^{18}$ without any $6$-Tverberg point.

The proof of Theorem~\ref{thm:counterexamples} will be given in Section~\ref{sec:counterexamples}.
 It is based on three ingredients:
\"Ozaydin's result (Theorem~\ref{thm:ozaydin}), our higher-multiplicity Whitney trick  (Theorem~\ref{thm_whitney_trick_extended} below), and a particular kind of PL map $\simplex^N\to \R^{3r}$ that we will call \define{prismatic} (see Definition~\ref{def:prismatic}).

\begin{remark}
In principle, the proofs of Theorems~\ref{thm:VK-Tverberg-complete} and \ref{thm:counterexamples} are constructive and do not require explicit knowledge of \"Ozaydin's equivariant map \eqref{eq:cor_ozaydin}; the existence of this map enters only in terms of the equivalent condition that the relevant obstruction vanishes. In each case, we start with an arbitrary map  (respectively, with a prismatic map) that may have $r$-Tverberg points and then construct the desired map through a finite sequence of \emph{$r$-fold Finger moves}, followed by a finite number of applications of the $r$-fold Whitney trick. It is an interesting question how complicated the final PL map $f$ in Theorem~\ref{thm:counterexamples} needs to be; see the discussion in Section~\ref{sec:open_problems}~(3).
\end{remark}

The key property of prismatic maps is that we will be able to ensure that all their Tverberg points are of the same \emph{type} $\{(r-1)k\}^r$, in the following sense:

%%%%%%%%%%%%%%%%%%%%%%%%%%%
\begin{definition}[\textbf{Tverberg Partitions and Type}]
\label{def:type}
Let $r\geq 2$, $d\geq 1$, $N=(d+1)(r-1)$, and let $f\colon \simplex^N\to \R^d$ be a PL map in general position. Suppose $y\in f(\tau_1)\cap \dots \cap f(\tau_r)$ is an $r$-Tverberg point of $f$ and $\dim \tau_i=m_i$, $1\leq i\leq r$. The vertex sets of the simplices $\tau_i$ form a partition of the vertex set of $\sigma^N$, hence $\sum_{i=1}^r m_i =d(r-1)$ and (by general position) $m_i\leq d$ for $1\leq i\leq r$. Somewhat abusing terminology, we say that $\tau_1,\ldots,\tau_r$ form a \define{Tverberg partition} for $f$, and we call the multiset of dimensions $\{m_1,m_2,\ldots,m_r\}$ the \define{type} of this Tverberg partition and of the Tverberg point $y$.
\end{definition}
%%%%%%%%%%%%%%%%%%%%%%%%%%%

As a byproduct of the proof of Theorem~\ref{thm:counterexamples}, we obtain the following result (where $\{m\}^r$ denotes the multiset containing the element $m$ with multiplicity $r$):
%%%%%%%%
\begin{corollary}
\label{cor:type-m}
Suppose $r\geq 2$, $k\geq 1$, and $N=(rk+1)(r-1)$. Then there exists an affine map $f\colon \simplex^N\to \R^{rk}$ such that all $r$-Tverberg points of $f$ are of the same type $\{m\}^r$, where $m=(r-1)k$.
\end{corollary}
%%%%%%%
It is also well-known that for every $r$ and $d$, there are affine maps\footnote{
Specifically, such an affine map is given by the point configuration in $\R^d$ (the images of the vertices) consisting of $(d+1)$ small clusters 
of $(r-1)$ points centered at the vertices of a $(d+1)$-simplex, plus one point at the barycenter of the simplex.} 
 all of whose Tverberg points are of type $\{1\}\cup\{d\}^{r-1}$.This raises the question whether we can generally construct (affine) maps all of whose Tverberg points are of a specified type:
\begin{question}
Let $r\geq 2$ and $d\geq 1$. Suppose we are given integers $m_1,\ldots, m_r \in \{0,1,\dots,d\}$ such that $\sum_{i=1}^r m_i=d(r-1)$.
Does there exist an affine map $f\colon \sigma^N\to \R^d$ such that all $r$-Tverberg points of $f$ are of the same type $\{m_1,m_2,\dots,m_r\}$? 
\end{question}

\subsection{A Higher-Multiplicity Whitney Trick}
\label{sec:intro-Whitney}

Our main tool to deal with intersections of higher multiplicity is a \emph{Whitney trick for $r$-fold points}
(Theorem~\ref{thm_whitney_trick_extended} below).

The classical Whitney trick (more precisely, its piecewise-linear version, see, e.g., \cite[p.~179]{Weber} or \cite[Lemma~5.12]{Rourke:Introduction-to-piecewise-linear-topology-1982}) allows one to eliminate a pair of isolated double points of \emph{opposite sign} (see Section~\ref{sec:intersection_signs} for the definition of intersection signs) of a PL map by an ambient PL isotopy fixed outside a small ball, provided the codimension is at least $3$.

Here and in what follows, an \define{ambient PL isotopy} of $\R^d$ is a PL homeomorphism $H\colon \R^d\times [0,1]\to \R^d\times [0,1]$ that preserves the $[0,1]$-component and thus gives rise to a family of PL homeomorphisms $H_t\colon \R^d\to \R^d$, $0\leq t\leq 1$ (see Section~\ref{sec:PLbackground} for more background on isotopies).

\begin{theorem}[\textbf{Whitney Trick}]
\label{thm:Whitney}
Suppose that $M_1$ and $M_2$ are connected, orientable PL manifolds, possibly with boundary, of respective dimensions 
of respective dimensions $m_1$ and $m_2$, $m_1+m_2=d$, and that 
$$f\colon M_1\sqcup M_2\to \R^d$$ 
is a PL map in general position defined on their disjoint union. 

If $x,y\in f(M_1)\cap f(M_2)$ are two double points of opposite sign\footnote{We remark that the sign of a double point depends on the choice of orientations of the $M_i$ and of $\R^d$, but if the $M_i$ are connected then the condition of having opposite signs is independent of such a choice.}
and if $d-m_i\geq 3$, $i=1,2$, then there exists a PL ambient isotopy
$H$ of $\R^d$ such that %modifying $f(M_{2})$ by this isotopy eliminates these two double points and does not introduce any others, 
$$f(M_1)\cap H_1(f(M_2)) = \big(f(M_1)\cap f(M_2) \big) \setminus \{x,y\}.$$
Moreover, the isotopy can be chosen to be local, in the following sense: Given any closed polyhedron $L\subset \R^d$ of dimension $\ell\leq d-3$ and with $x,y\not\in L$, there exists a PL ball $B^d\subset \R^d$ %such that $B^d$ is 
disjoint from $L$ %, $B \cap f(\sigma_i)$ is an $m_i$-dimensional PL ball properly embedded in $B$ $i=1,2$, and 
such that $H$ is fixed outside of $\interior B^d$.% (in particular, on $\partial f(\sigma_i)$).
\end{theorem}
Figure~\ref{fig_exp_Whitney} illustrates this in a low-dimensional situation. The idea of the trick is to ``push'' $f(M_2)$ upwards until the two intersections points $x$ and $y$ disappear, while keeping the boundary of $f(M_2)$ fixed. In low codimensions, doing this might require passing over some obstacles and/or introducing new double points, but if $d-m_i\geq 3$, $i=1,2$ then these problems can be avoided.\footnote{The hypotheses for of the Whitney trick can be weakened, e.g., one of the $\sigma_i$ can 
be allowed to have dimension $m_i=d-2$, but then one needs to impose additional technical conditions like local flatness and simple connectivity of the complement $\R^d\setminus f(\sigma_i)$; see, e.g., \cite[Lemma~5.12]{Rourke:Introduction-to-piecewise-linear-topology-1982}.} 
\begin{figure}[ht]
\begin{center}
\includegraphics[scale=1]{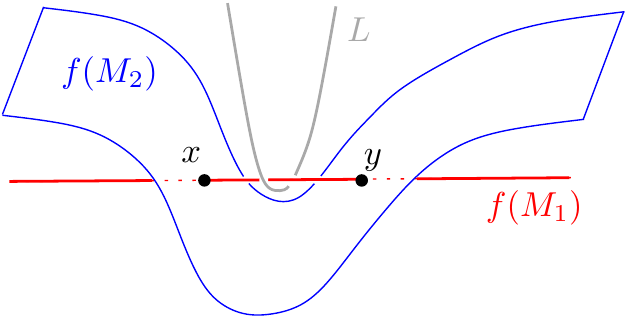}
\caption{$f(M_1)$ and $f(M_2)$ intersecting in two double points $x,y$ of opposite signs, and a potential obstactle $L$. 
%Pushing $f(\sigma_2)$ upwards to remove $x$ and $y$ will generate new double points between $f(\sigma_2)$ and $L$.
}
\label{fig_exp_Whitney}
\end{center}
\end{figure}

%We remark that the definition of the sign of a double of, more generally, $r$-fold point requires choosing orientations of $\R^d$ and of the $M_i$, but if the $M_i$ are \emph{connected} then the condition of having opposite signs is independent of the choice of orientations.

In the present paper, we prove the following analogue of Theorem~\ref{thm:Whitney} for $r$-fold points:
\begin{theorem}[\textbf{Higher-Multiplicity Whitney Trick}] 
\label{thm_whitney_trick_extended}

Let $r\geq 2$, and let $M_1, \dots, M_r$ be connected, orientable PL manifolds\footnote{We are mostly interested in the case that each $M_i\cong\simplex^{m_i}$ is a simplex, but the proof of the more general case comes at no extra cost.}, of respective dimensions $\dim M_i=m_i$, such that
\begin{equation}
\label{eq:dim-critical}
\sum_{i=1}^r m_i =d(r-1)
\end{equation}
and 
\begin{equation}
\label{eq:codimension-3}
d-m_i\geq 3, \qquad 1\leq i\leq r. 
\end{equation}
Let 
\[
f : M_1 \sqcup \cdots \sqcup M_r \rightarrow \R^d
\]
be a PL map in general position defined on their disjoint union, and suppose that 
\[
x,y\in f(M_1) \cap f(M_2) \cap \cdots \cap f(M_r)
\]
are two $r$-fold points of opposite intersection signs (see Section~\ref{sec:intersection_signs}).

Then there exist $r-1$ PL ambient isotopies $H^2,\ldots,H^r$ of $\R^d$
%\[
%H^2 , \dots , H^r : \R^d \times [0,1] \rightarrow \R^d \times [0,1]
%\]
such that%\footnote{Here, $H^i_1\colon \R^d \to \R^d$ denotes the ``final stage of the isotopy $H^i$, given by $H^i(x,1)=(H_1^i(x),t)$; 
%see Section~\ref{sec:prelim-isotopies} for more details on isotopies.}
\[
f(M_1) \cap H^2_1( f (M_2)) \cap \cdots \cap H^r_1( f (M_r)) = \big(f(M_1) \cap f(M_2) \cap \cdots \cap f(M_r) \big)\setminus \{x,y\}
\]
Moreover, these isotopies can be chosen to be \emph{local}, in the following sense: Given any closed polyhedron $L\subset \R^d$ of dimension $\ell\leq d-3$ and with $x,y\not\in L$, there exists a PL ball $B^d\subset \R^d$ %such that $B^d$ is 
disjoint from $L$ %, $B \cap f(M_i)$ is an $m_i$-dimensional PL ball properly embedded in $B$ $i=1,2$, and 
such that $H^i$ is fixed outside of $\interior B^d$, $2\leq i\leq r$.
\end{theorem}

As another application of these ideas, we also have the following generalization of the classical result of Whitney that $k$-dimensional manifolds embed into $\R^{2k}$:
\begin{proposition} \label{prop_r_intersection_manifold}
Let $r\geq 2$, $k\geq 3$, and let $M$ a PL manifold of dimension $m=(r-1)k$. 
Then there exists a PL map $M\to \R^{rk}$ without $r$-fold points.
\end{proposition}
The proofs
of Theorem~\ref{thm_whitney_trick_extended} and Proposition~\ref{prop_r_intersection_manifold} 
will be given in Section~\ref{sec_whitney_trick}.

\subsection{Future Work and Open Problems}
\label{sec:open_problems}

\begin{enumerate}[label=\textup{(\arabic*)}]
\item \textbf{Codimension $\boldsymbol{2}$.} Theorem~\ref{thm:vKcomplete} fails for $m=2$: Freedman, Krushkal and Teichner~\cite{Freedman:van-Kampens-embedding-obstruction-is-incomplete-for-2-complexes-in-bf-R4-1994} constructed examples of finite $2$-dimensional complexes whose Van Kampen obstruction vanishes but which are not embeddable into $\R^4$. (More generally, for every pair $(m,d)$ with $3\leq d<3(m+1)/2$, there are counterexamples \cite{Mardesic:varepsilon-Mappings-and-generalized-manifolds-1967,Segal:Quasi-embeddings-and-embeddings-of-polyhedra-in-mathbb-Rsp-m-1992,Segal:Embeddings-of-polyhedra-in-mathbb-Rm-and-the-deleted-product-obstruction-1998,GoncalvesSkopenkov:EmbeddingsHomologyEquivalentManifolds-2006} that show that the deleted product criterion is insufficient for embeddabbility of $m$-complexes into $\R^d$.)
We suspect that similar counterexamples to Theorem~\ref{thm:VK-Tverberg-complete} exist for $m=2$.
%\uli{TODO: Clarify embeddings vs. almost-embeddings.}
%Moreover, in the case of embeddability, the counterexamples were used \cite{MatousekTancerWagner:HardnessEmbeddings-2011} to prove that deciding embeddability of $m$-complexes into $\R^d$ is  NP-hard for $4\leq d < 3(m+1)/2$. Is there a similar hardness result for deciding the existence of maps without Tverberg points outside the $r$-metastable range?

\item \textbf{The Planar Case and Hanani--Tutte.}
On the other hand, Theorem~\ref{thm:vKcomplete} remains true for $m=1$ (embeddings of graphs in the plane), even under the slightly weaker assumption that the Van Kampen obstruction vanishes modulo $2$. This is essentially the \emph{Hanani--Tutte Theorem} \cite{Hanani:UnplattbareKurven-1934,Tutte:TowardATheoryOfCrossingNumbers-1970}, which guarantees that a graph is planar iff it can be drawn in the plane such that any pair of vertex-disjoint edges cross an even number of times. 
The classical proofs of that theorem rely on \emph{Kuratowski's Theorem}, but more recently \cite{Pelsmajer:Removing-even-crossings-2007,Pelsmajer:Removing-independently-even-crossings-2010}, more direct proofs have been found that do not use forbidden minors (an earlier attempt at a Whitney-trick for graphs in the plane \cite{Sarkaria:A-one-dimensional-Whitney-trick-and-Kuratowskis-graph-planarity-1991} contained an error; see \cite[p.~17]{Skopenkov:On-approximability-by-embeddings-of-cycles-in-the-plane-2003}). It would be very interesting to know whether there is an analogue of the Hanani--Tutte theorem for Tverberg-type problems in $\R^2$. In particular, in light of \"Ozaydin's result, this would be an approach to completely settling the non-prime power case of the topological Tverberg conjecture by constructing counterexamples for $d=2$. We plan to investigate this in a future paper.

\item \textbf{Complexity of Maps without Tverberg Points.} It is an interesting question how complicated the counterexamples to the topological Tverberg conjecture need to be. For $r=2$ and $m\geq 3$, Freedman and Krushkal have constructed examples of $m$-dimensional complexes $K$ with $n$ simplices such that $K$ admits a PL embedding into $\R^{2m}$ (equivalently, $\obs(\delprod{K}{2})=0$), but any subdivision $K'$ of $K$ that supports such a PL embedding requires at least $C^n$ simplices, where $C=C_m>1$ is a constant depending on $m$. Complementing this, they also showed that there is always a suitable subdivision with at most $O(e^{n^{4+\varepsilon}})$ simplices, for any $\varepsilon>0$. It would be interesting to know whether there are similar bounds for maps $K\to \R^d$ without $r$-Tverberg points, $\dim K=m=(r-1)k$, $d=mk$, $k\geq 3$.

%\uli{TODO: \textbf{Close $\boldsymbol{r}$-fold points.}} 
%%\item \textbf{Non-Tverberg Multiple Points.} The methods presented here can be extended to eliminate $r$-fold points that are not Tverberg points, i.e., with preimages in an $r$-tuple of simplices of $K$ that are not pairwise disjoint. We plan to address this in a subsequent paper.  This will involve, on the one hand, finding a good cellular configuration space\footnote{For $r\geq 3$,
%%%the combinatorial deleted product $\delprod{K}{r}$ is no longer homotopy equivalent to the topological deleted product
%%%$\{(x_1,\ldots,x_r)\in K^r\colon x_i\neq x_j, i\neq j\}$, which is, in a sense, the natural configuration space.} and a corresponding obstruction and, on the other hand, generalizing
%%%geometric operations like the \emph{van Kampen trick} to eliminate such non-Tverberg multiple points.
\end{enumerate}

\paragraph{Acknowledgements.}  We would like to thank the anonymous referees of the extended abstract \cite{MabillardWagner:TverbergWhitney-2014} for helpful comments and remarks. 

Moreover, we would like to thank Florian Frick, Gil Kalai, Arkadiy Skopenkov, and G\"unter Ziegler for detailed comments 
on a preliminary draft of this paper, and Pavle Blagojevi\'c for asking us to clarify the role of the full symmetric group 
(Remark~\ref{rem:equivariance-sym_r-finger-moves}).

Furthermore, U.W. would like to express his gratitude to Ji\v{r}\'{\i} Matou\v{s}ek, Eran Nevo, and Martin Tancer for years of fruitful collaboration on algorithmic and combinatorial aspects of the embeddability problem and many discussions on the classical Van Kampen obstruction; without this background, the work presented here would not have been undertaken.

\section{Preliminaries}
\label{sec:preliminaries}

%%%%%%%%%%%%%%%%%%%%%%%%%%
\subsection{Tools from Piecewise-Linear Topology}
\label{sec:PLbackground}
%%%%%%%%%%%%%%%%%%%%%%%%%%
In this subsection (which readers may want to skip or just skim through at first reading), we collect, for ease of reference, a number of basic notions and results from piecewise-linear (PL) topology that we will use repeatedly throughout the paper

For a very readable and compact introduction to the area, see the survey article \cite{Bryant:Piecewise-linear-topology-2002}. For more details see, e.g., the textbook \cite{Rourke:Introduction-to-piecewise-linear-topology-1982} or the lecture notes \cite{Zeeman:Seminar-on-combinatorial-topology-1966}.
We refer the reader to any of these sources for much of the basic terminology, such as \define{PL manifolds} and \define{regular neighborhoods}. %, and mostly focus on notions for which there might be some potential ambiguity.
A \define{polyhedron} will always mean the underlying polyhedron of some geometric simplicial complex in some $\R^d$.

%Throughout this section, all manifolds and maps will be piecewise linear, unless explicitly stated otherwise.

\subsubsection{Isotopies, Ambient Isotopies, and Unknotting}
\label{sec:prelim-isotopies}

One of the facts that make working in codimension at least $3$ easier is that \emph{isotopic} embeddings are also \emph{ambient isotopic}, see below. This fails in codimension $2$; for instance, any two PL knots (embeddings of $S^1$) in $S^3$ are isotopic, but not necessarily ambient isotopic.

%We follow the notation and terminology in \cite{Hudson:Concordance-isotopy-and-diffeotopy-1970}.

Let $X$ be a polyhedron, and let $Q$ be a PL manifold. 
%A \define{(PL)} \define{concordance} $F$ of $X$ in $Q$ is a PL embedding $F\colon  X \times [0,1] \hookrightarrow  Q \times [0,1]$ such that $F ( X \times 0) \subseteq Q \times 0$ and $F(X \times 1) \subseteq Q \times 1$.
%A \define{(PL)} \define{isotopy} of $X$ in $Q$ is a concordance that is \define{level-preserving}, i.e., such that $F(X \times t) \subseteq Q \times t$ for all $t\in [0,1]$. 
A \define{(PL)} \define{isotopy} of $X$ in $Q$ is a PL embedding $F\colon  X \times [0,1] \to  Q \times [0,1]$ 
that is \define{level-preserving}, i.e., such that $F(X \times t) \subseteq Q \times t$ for all $t\in [0,1]$. 
An isotopy determines embeddings $F_t\colon X\hookrightarrow Q$ by $F(x, t) = (F_t(x), t)$ for $x\in X$ and $t\in [0,1]$.
%In the more general case that $F$ is just a concordance, only the two embeddings $F_0$ and $F_1$ are defined.
%A concordance [isotopy]

An isotopy $F$ is \define{fixed} on a subspace $Y\subseteq X$ if $F(y,t)=(F_0(y),t)$ for all $t\in [0,1]$ and $y\in Y$.
%A concordance [isotopy] $F$ is \define{allowable} if $F^{-1}(Q \times 0) = X \times 0$, , $F^{-1}(Q \times 1) = X \times 1$, and 
%$F^{-1}(\partial Q \times [0,1])= X_0 \times I$ for some closed subpolyhedron $X_0\subseteq X$. 
An isotopy $F$ is \define{allowable} if $F^{-1}(\partial Q \times [0,1])= X_0 \times [0,1]$ for some closed 
subpolyhedron $X_0\subseteq X$. 

Two embeddings $f,g\colon X\hookrightarrow Q$ are \define{(allowably) %concordant [isotopic] 
isotopic (keeping $Y$ fixed)} if there is an (allowable) %concordance [isotopy] 
isotopy (fixed on $Y$) $F$ of $X$ in $Q$ such that $F_0=f$ and $F_1=g$.

An \define{ambient PL isotopy} of $H$ of $Q$ is a level-preserving PL homeomorphism $H\colon Q \times [0,1] \to Q\times [0,1]$ such that $H_0$ is the identity on $Q$. Two PL embeddings $f,g\colon X\hookrightarrow Q$ are \define{ambient isotopic (keeping $Y\subseteq Q$ fixed)} if there is an ambient isotopy $H$ of $Q$, fixed on $Y$, with $g=H_1\circ f$. 
An ambient isotopy $H$ of $Q$ \define{extends} an isotopy $F$ of $X$ in $Q$ if $F_t=H_t\circ F_0$ for all $t\in [0,1]$.

Let $M$ and $Q$ be PL manifolds, possibly with boundary. A PL embedding $f:M\rightarrow Q$ is \define{proper} if $f^{-1} (\boundary Q) = \boundary M$. An isotopy is proper if it is proper as an embedding.

\paragraph{From isotopy to ambient isotopy.}

\begin{theorem}[{\textbf{Hudson}%'s Isotopy Extension Theorem
~\cite[Thm 1]{Hudson:Extending-piecewise-linear-isotopies-1966}}] \label{thm_Hudson_Isotopy_Extension}
Let $M$ and $Q$ be PL manifolds, $M$ compact, and let $F:M \times [0,1] \rightarrow Q \times [0,1]$ be a proper isotopy of $M$ 
in $Q$, fixed on $\boundary M$. If $\dim Q-\dim M \geq 3$, then there is an ambient isotopy of $Q$, fixed on $\boundary Q$, that extends $F$.
\end{theorem}

We will also need the following result concerning embeddings of compact polyhedra:\footnote{In \cite{Hudson:Concordance-isotopy-and-diffeotopy-1970}, the result is stated in a stronger form: The conclusion remains true under the weaker hypothesis that $f$ and $g$ are \define{allowably concordant} keeping $Y$ fixed. (The notion of an allowable concordance $F$ between $f=F_0$ and $g=F_1$ fixing $Y$ is a generalization of an allowable isotopy fixing $Y$, where the requirement that $F$ preserve levels is relaxed to the conditions $F(X\times t) \subseteq Q\times t$ for $t=0,1$ and $F(X\times t) \subseteq Q\times (0,1)$ for $t\in (0,1)$, see \cite[Section~1]{Hudson:Concordance-isotopy-and-diffeotopy-1970}.)}

\begin{proposition}[{\textbf{Hudson}~\cite[Corollary~1.3]{Hudson:Concordance-isotopy-and-diffeotopy-1970}}]
\label{prop:Hudson-polyhedra}
Let $X$ be a compact polyhedron and let $Q$ be a PL manifold. Let $f, g\colon X \to Q$ be allowably isotopic embeddings keeping $Y\subseteq X$ fixed, with $X_0=f^{-1}(\partial Q) \subseteq Y$. If $\dim X \leq \dim Q - 3$, then $f$ and $g$ are ambient 
isotopic keeping $f(Y)\cup \partial Q$ fixed.
\end{proposition}

%Let $X$ be a compact PL space and $Q$ a PL manifold. An allowable $n$-isotopy of $X$ in Q is a PL embedding $F\colon X \times [0,1]^n \to Q \times [0,1]^n$ which commutes with the projection onto $[0,1]^n$, and with $F^{-1}(\partial Q \times [0,1]^n) = X_0 \times [0,1]^n$ for some PL subspace $X_0$ of $X$. An ambient $n$-isotopy of $Q$ is a PL homeomorphism $H\colon Q \times [0,1]^n \to Q \times [0,1]^n$ which commutes with projection onto $[0,1]^n$, and such that $H|_{Q \times 0}$ is the identity on $Q\times 0$ (where $0=(0,\ldots,0)\in I^n$).

%\begin{theorem}[\textbf{Hudson Isotopy Extension Theorem for Polyhedra}~{\cite[Theorem~4.1]{Hudson:Concordance-isotopy-and-diffeotopy-1970}}]
%Let $X$ be a compact polyhedron, and let $F\colon X \times [0,1]^n \to Q \times [0,1]^n$ be an allowable $n$-isotopy with 
%$F^{-1}(\partial Q\times [0,1]^n)=X_0 \times [0,1]^n$, and let $H\colon \partial Q\times [0,1]^n \to  \partial Q \times [0,1]^n$ be an ambient $n$-isotopy of $\partial Q$ such that $(H\circ F)|_{X_0\times [0,1]^n}=F_0|_{X_0} \times 1$. 
%
%If $\dim X \leq \dim Q - 3$, then there is an ambient $n$-isotopy $K$ of $Q$ such that $K|_{\partial Q\times [0,1]^n}=H$ and 
%and $K\circ F = F_0 \times 1$, where $F_0\colon X \to Q$ is defined by $F(x, 0) = (F_0(x), 0)$.)
%\end{theorem}

\paragraph{Unknotting of balls and spheres.}
A \define{(PL) $(q,m)$-manifold pair} $(Q,M)$ is a pair of PL manifolds $M$ and $Q$ of dimensions $m$ and $q$, respectively 
such that $M\subseteq Q$ properly. 

A pair $(B^q,B^m)$ of PL balls (respectively, a pair $(S^q,S^m)$ of PL spheres), $m\leq q$, is \define{unknotted} if it is PL homeomorphic to the \define{standard ball pair} $([-1,1]^q, [-1,1]^m \times 0)$ (respectively, to the \define{standard sphere pair} 
$(\partial [-1,1]^{q+1}, \partial([-1,1]^m \times 0))$.)

\begin{theorem}[{\textbf{Zeeman}~\cite[Ch.~IV, Theorem~9]{Zeeman:Seminar-on-combinatorial-topology-1966}}]
If $q-m\geq 3$ then every PL ball pair $(B^q,B^m)$ and every PL sphere pair $(S^q,S^m)$ are unknotted.
\end{theorem}

We will also need the following relative version:

\begin{corollary}[{\textbf{Zeeman}~\cite[Ch.~IV, Corollary~1, p.~16]{Zeeman:Seminar-on-combinatorial-topology-1966}}] \label{thm_Zeeman_proper_embeddings}
If $q-m \ge 3$, then any two proper embeddings $B^m \subseteq B^q$ that agree on $\boundary B^m$ are ambient isotopic, 
keeping $\boundary B^q$ fixed.
\end{corollary}

\paragraph{From homotopy to ambient isotopy.}

\begin{theorem}[{\textbf{Zeeman}~\cite[Ch X, p 198, Thm 10.1]{Zeeman:Seminar-on-combinatorial-topology-1966}}] 
\label{thm_Zeeman_Homotopy}
\label{thm:Zeeman-general-unknotting}
Let $M$ and $Q$ be compact manifolds of dimensions $q$ and $m$, respectively, and let $f,g : M \rightarrow Q$ be two proper embeddings. Suppose that $f$ is homotopic to $g$ relative to $\boundary M$. Then if $q-m \ge 3$, $M$ is $(2m-q+1)$-connected, and $Q$ is $(2m-q+2)$-connected, then $f$ and $g$ are ambient isotopic keeping $\boundary Q$ fixed.
\end{theorem}

\begin{theorem}[{\textbf{Irwin}~\cite[Ch. VIII, p. 4, Thm. 23]{Zeeman:Seminar-on-combinatorial-topology-1966}}] 
\label{thm_Irwin}
Assume $M$ is compact and let $f:M \rightarrow Q$ be a continuous map such that $f|_{\boundary M}$ is a piecewise-linear embedding of $\boundary M$ in $\boundary Q$. Then $f$ is homotopic to a proper embedding keeping $\boundary M$ fixed provided
\[
q-m \ge 3, \quad 
M \text{ is $(2m-q)$-connected}, \quad 
Q \text{ is $(2m-q +1)$-connected}.
\]
\end{theorem}

\subsubsection{General Position and Transversality}
\label{sec:GP}

There are many variants of general position. For the purposes of studying $r$-fold points and $r$-Tverberg points, the following definitions are convenient.

\paragraph{General position in $\boldsymbol{\R^d}$.}
A collection $\mathcal{A}$ of affine subspaces of $\R^d$ is \define{in general position} 
if for every $r\geq 2$ and pairwise distinct $A_1,\ldots, A_r \in \mathcal{A}$,
\begin{equation}
\label{eq:transversality-subspaces}
\dim\big({\textstyle \bigcap_{i=1}^r A_i} \big)=\max \big\{-1, \big({\textstyle \sum_{i=1}^r} \dim (A_i)\big) - d(r-1)\big\}.
\end{equation}

A set $S$ of points in $\R^d$ is \define{in general position} if, for every $r\geq 2$ and pairwise disjoint subsets $S_1,\ldots,S_r\subseteq S$, the affine hulls $\aff(S_i)$, $1\leq i\leq r$, are in general position.\footnote{Note that this is stronger than requiring that every 
subset of at most $d+1$ points in $S$ is affinely independent;  e.g. the vertices of a regular hexagon are not in general position 
in the stronger sense.} 

A collection $\mathcal{P}=\{P_1,\ldots,P_r\}$ of convex polyhedra in $\R^d$ is in general position if $\aff(F_1),\ldots\aff(F_r)$ are in general position for every choice of nonempty faces $F_i \subseteq P_i$, $1\leq i\leq r$.

If $K$ is a simplicial complex and $f \colon K \to \R^d$ is a simplexwise-linear map, then we say that $f$ is in general position if the images 
of the vertices of $K$ are pairwise distinct and in general position. A PL map $f\colon K\to \R^d$ is in general position if there is some subdivision $K'$ of $K$ such that $f$ is simplexwise-linear and in general position as a map $K'\to \R^d$. 

If $K$ is a finite simplicial complex and $f\colon K\to \R^d$ is a continuous map then, by a simple compactness and perturbation argument, 
for every $\varepsilon>0$, there exists a PL map $g\colon K\to \R^d$ in general position such that $\|f-g\|_\infty \leq \varepsilon$.

\paragraph{General position in PL manifolds.}
Defining general position without reference to a particular triangulation and, more generally, for maps into PL manifolds $M$ other than $\R^d$, is more involved. 
We follow the presentation \cite[Ch. VI]{Zeeman:Seminar-on-combinatorial-topology-1966}, which is very suitable for dealing with $r$-fold points.

Let $f\colon X \to Q$ be a PL map from a polyhedron to a PL manifold. For $r\geq 2$, let us say that a point $x\in X$ is \define{$r$-singular} if it is the preimage of an $r$-fold image point $y$ of $f$, i.e., if $|f^{-1}(f(x))|\geq r$. The \define{(closed) $r$-singular set} $S_r(f) \subseteq X$ is defined as the closure of the set of $r$-singular points of $f$. Each $S_r(f)$ is a subpolyhedron of $X$
(\cite[Ch. VI, Lemma~31, p.~19]{Zeeman:Seminar-on-combinatorial-topology-1966}). The set $S_2(f)$ is also sometimes simply called the \define{singular set} of $f$ and denoted $S(f)$. 

Suppose $\dim X=m$ and $\dim Q=q$. Then a PL map $f\colon X\to Q$ is said to be \define{in general position} if
$\dim S_r(f)\leq m-(r-1)(q-m)$ for every $r\geq 2$. If $X_0\subseteq X$ is a subpolyhedron then $f$ is said to be in
general position for the pair $(X,X_0)$ if $f$ and $f|_{X_0}$ are both in general position and, if $\dim X_0<\dim X$ then
$\dim(S_r(f)\cap X_0)<m-(r-1)(q-m)$ for every $r$.

\begin{theorem}[{\cite[Ch. VI, Theorem~18, p.~27]{Zeeman:Seminar-on-combinatorial-topology-1966}}]
Let $f\colon X \to \interior Q$ be a PL map, $\dim X<\dim Q$, and let $X_0\subseteq X$ be a subpolyhedron.
If $f|_{X_0}$ is in general position then for every $\varepsilon>0$ there exists a map $g\colon X\to Q$ that is in general position for the pair
$(X,X_0)$, and $f\simeq g$ are homotopic through an $\varepsilon$-small homotopy that keeps $X_0$ fixed.
\end{theorem}

We will also need the following version of being in general position with respect to a given polyhedron: 

\begin{theorem}[{\cite[Ch. VI, Theorem~15, p.~7]{Zeeman:Seminar-on-combinatorial-topology-1966}}]
Let $Q$ be a PL manifold of dimension $m$, and let $X_0\subseteq X$ and $Y\subseteq Q$ be polyhedra. Given an embedding
$f\colon X\to Q$ such that $f(X\setminus X_0)\subseteq \interior Q$, for every $\varepsilon >0$ there is an embedding $g\colon X\to Q$ such that $g|_{X\setminus X_0}$
is \define{in general position with respect to $Y$}, in the sense that
$$\dim(g(X\setminus X_0) \cap Y) \leq \dim(X\setminus X_0)+\dim Y -\dim Q,$$ 
and $f$ and $g$ are ambient isotopic through an $\varepsilon$-small ambient isotopy fixing $\partial Q$ and $f(X_0)$. 
\end{theorem}

\paragraph{Transversality.}

%PL submanifolds $M, P$ of $Q$ are transversal at the point $x \in \interior M \cap \interior P$ (respectively $\boundary M\cap \boundary P$) if there is a coordinate neighbourhood $h\colon \R^q\to Q$ (respectively, $h\colon \R^q_+\to Q$) of $x$ in $Q$
%such that $h^{-1}(M), h^{-1}(P)$ are two linear subspaces of $\R^q$ ($\R^q_+$) in general position.

Suppose that $M_1,\ldots, M_r$ are properly embedded PL submanifolds of a PL manifold $Q$, $\dim M_i=m_i$, $1\leq i\leq r$, and $\dim Q=q$. We say that the $M_i$ are \define{mutually transverse} (or that they \define{intersect transversely}) if they locally intersect
like $r$ affine subspaces in general position. 

More precisely, the $M_i$ intersect transversely at a point $y\in \interior Q$ [respectively, $y\in \partial Q$] if there is a neighborhood $N$ of $y$ in $Q$ and a PL homeomorphism $h\colon \interior N \cong \R^q$ [respectively, $h\colon N\cong \R^{q-1}\times \R_+$]
such that the images $h(M_i\cap \interior N)$, $1\leq i\leq r$, are affine subspaces in general position [respectively, intersections of such subspaces with the upper halfspace $\R^{q-1}\times \R_+$]. The $M_i$ are mutually transverse if they intersect transversely at every $y\in  \bigcap_{i=1}^r M_i$. (In particular, if $\bigcap_{i=1}^r M_i \neq \emptyset$, then this implies that $\sum_i m_i \geq d(r-1)$. )

In general, transversality for PL manifolds is much more subtle than the corresponding theory in the smooth case, see
e.g., the discussion in \cite{Armstrong:Transversality-for-piecewise-linear-manifolds-1967}.\footnote{A particularly striking fact is the failure of relative PL transversality: Hudson~\cite{Hudson:On-transversality-1969} showed that for every $m,n,q$ with $m+n?q=4k$, $m,n \geq 8k+2$, there are transverse PL spheres $S^m,S^n \subseteq S^q$ which can not be extended to transverse embeddings 
of balls $B^{m+1},B^{n+1}\subseteq B^{q+1}$.}

In the present paper, we will only use the following simple fact: If $M_1,\ldots,M_r$ are pairwise disjoint PL manifolds, $\dim M_i=m_i$, $\sum_i m_i=d(r-1)$, and if $f\colon M_1\sqcup \ldots \sqcup M_r \to \R^d$ is a PL map in general position, then the images $f(\sigma_i)$ are mutually transverse at every $r$-fold point (necessarily an $r$-Tverberg point) $y$  of $f$; indeed, for suitable subdivisions of the $M_i$ on which $f$ is simplexwise linear, there are simplices $\sigma_i'$ of the subdivisions, $1\leq i\leq r$, such that the images $f(\sigma_i')$ are linear $m_i$-simplices in general position whose relative interiors intersect exactly at $y$.
All operations that we will perform will preserve transversality of the intersections.

%%%%%%%%%%%%%%%%%%%%%%%%%%%%%
\subsection{Oriented Intersections and Intersection Signs}
\label{sec:intersection_signs}
%%%%%%%%%%%%%%%%%%%%%%%%%%%%%

In this subsection, we review the induced orientation on the intersection of oriented simplices in general position in $\R^d$ and
the resulting intersection product on piecewise-linear chains (this is a particular case of  Lefschetz intersection theory \cite{Lefschetz:Intersections-and-Transformations-of-Complexes-and-Manifolds-1926}). We first fix the notation and state the 
basic properties that we will need later (Lemmas~~\ref{lem:prop-inters-prod} and \ref{lem:intersection-bd}). The definition and the 
proofs of the two lemmas, which boil down to elementary linear algebra, are included here for the sake of completeness but are deferred until the end of this subsection, and the reader may wish to skip them at first reading.

Let $\sigma_1,\ldots,\sigma_r$ be oriented simplices or, more generally, convex polyhedra
in general position in $\R^d$, $\dim \sigma_i=m_i$, $1\leq i\leq r$ (see Figure~\ref{fig_three_intersecting_triangles}
for an illustration in the case $r=d=3$, $m_1=m_2=m_3=2$).
\begin{figure}[h]
\begin{center}
\includegraphics[scale=1]{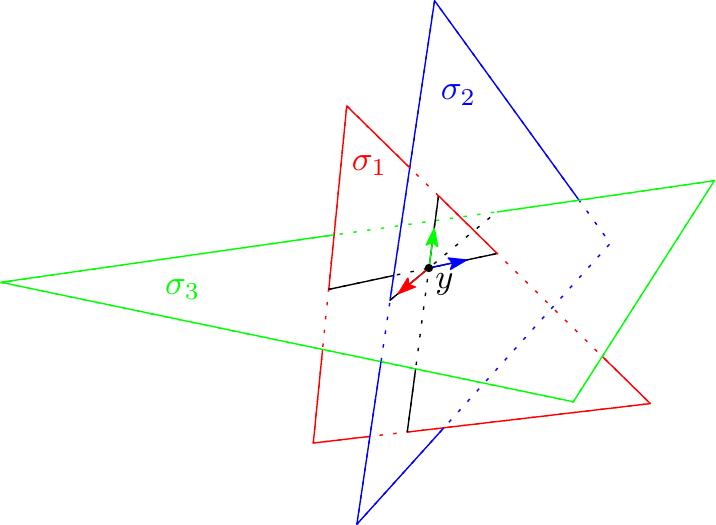}
\caption{Three triangles in general position intersecting at $y$.}
\label{fig_three_intersecting_triangles}
\end{center}
\end{figure}

Then the intersection $\bigcap_i \sigma_i$ is either empty or a convex polyhedron of dimension $\left(\sum_{i=1}^r m_i\right) -d(r-1)$. In the latter case, given orientations of the ambient space $\R^d$ and of each $\sigma_i$, we can define  (see Definition~\ref{def:induced_orientation_polytopes} below) an \define{induced orientation} on 
$$\sigma_1\cap\ldots\cap \sigma_r,$$ 
which depends on the order of the $\sigma_i$ and on the choices of the orientations. We will also speak of the 
\define{oriented intersection} of the $\sigma_i$ in $\R^d$, and occasionally write 
$(\sigma_1\cap \ldots \cap \sigma_r)_{\R^d}$ to stress dependence of the orientation on that of the ambient space. If the dimensions satisfy
\begin{equation}
\label{eq:intersection_dim_zero}
\sum_{i=1}^r m_i=d(r-1),
\end{equation}
then the intersection is either empty, or it consists of a single point $y$ that lies in the relative interior of each $\sigma_i$, and the induced orientation amounts to associating an \define{($r$-fold) intersection sign} in $\{-1,+1\}$ to $y$, denoted by 
$$\isign{y}{\sigma_1}{\sigma_r},$$
or by $\sign_y^{\R^d}(\sigma_1,\ldots,\sigma_r)$, if we want to stress the ambient space.

The following lemma summarizes several properties that we will need in this paper.
\begin{lemma} 
\label{lem:prop-inters-prod}
Suppose we have chosen an orientation of $\R^d$, and let $\sigma_1,\ldots,\sigma_r$ be oriented simplices in general position in 
$\R^d$, $\dim \sigma_i=m_i$, $1\leq i\leq r$.
\begin{enumerate}[label=\textup{(\alph*)}]
\item \textbf{Orientation reversal}:\label{orientation-reversal}  
Reversing the orientation of one  $\sigma_i$ (denoted by $-\sigma_i$) also reverses the orientation of the intersection,
$$ \sigma_1\cap \ldots \cap \sigma_{i-1}\cap (-\sigma_i)\cap \sigma_{i+1}\ldots \cap \sigma_r = -(\sigma_1\cap \ldots \cap \sigma_r).$$
If we reverse the orientation of $\R^d$ (denoted by $-\R^d$) then the orientation of the intersection changes by a factor of $(-1)^{r-1}$, 
$$(\sigma_1\cap \ldots \cap \sigma_r)_{-\R^d} = (-1)^{r-1}(\sigma_1\cap \ldots \cap \sigma_r)_{\R^d}.$$
\item \textbf{Skew commutativity}: \label{skew-commutativity}
For pairwise oriented intersections, 
$$\sigma_2\cap \sigma_1=(-1)^{(d-m_1)(d-m_2)} \sigma_1\cap \sigma_2.$$ 
Thus, in general, if $\pi\in \sym_r$ then 
$$\sigma_{\pi(1)}\cap \ldots \cap \sigma_{\pi(r)}=(-1)^{\sum_{(i,j)\in\textup{Inv}(\pi)}(d-m_i)(d-m_j)} \sigma_1\cap \ldots \cap \sigma_r,$$
where $\textup{Inv}(\pi):=\{(i,j)\in [r]^2\mid i<j, \pi(i)>\pi(j)\}$ is the set of \emph{inversions} of $\pi$.
\item \textbf{Restriction:} Consider the oriented pairwise intersections $\sigma_1\cap \sigma_2, \ldots, \sigma_1\cap \sigma_r$ as oriented convex subpolytopes of (the affine hull of) $\sigma_1$. If we compute the $(r-1)$-fold oriented intersection of these within $\sigma_1$, the result is the same as the $r$-fold oriented intersection of $\sigma_1,\ldots ,\sigma_r$ inside $\R^d$,
$$(\sigma_1\cap \ldots \cap \sigma_r)_{\R^d}=\big((\sigma_1\cap \sigma_2)_{\R^d} \cap \ldots \cap (\sigma_1\cap \sigma_r)_{\R^d}\big)_{\sigma_1}.$$
\item Suppose the dimensions satisfy \eqref{eq:intersection_dim_zero}, i.e., that  $\sigma_1\cap \ldots \cap \sigma_r$ consists of a single point $y$. Then the product $P:=\sigma_1\times \ldots \times \sigma_r$ is a convex polytope of dimension $d(r-1)$ 
that intersects the thin diagonal $\thindiag{r}{\R^d}$  transversely at the point $(y,\ldots,y)\in (\R^d)^r$. Moreover, 
the orientations of the $\sigma_i$ determine an orientation of $P$, and the orientation of $\R^d$ determines orientations of both $(\R^d)^r$ and of $\thindiag{r}{\R^d}$ (see Equation~\eqref{eq:orientations-product-diagonal} below), and with respect to these orientations,%
%and the $r$-fold intersection sign of the $\sigma_i$ at $y$ is equal, modulo a sign depending only on the dimensions, to the 
%pairwise intersection sign of $P$ and $\thindiag{r}{\R^d}$ at $(y,\ldots,y)$:
\footnote{For $r=2$, this is well-known, and can be found in \cite[§3]{Shapiro:FirstObstruction-1957}.}
\begin{equation}
\label{eq:r-sign-vs-product-diagonal}
\sign_{y}^{\R^d}(\sigma_1,\ldots,\sigma_r)=\varepsilon_{d,m_1,\ldots,m_r}\cdot \sign_{(y,\ldots,y)}^{(\R^d)^r}(\sigma_1\times \ldots \times \sigma_r,\thindiag{r}{\R^d}),
\end{equation}
where $\varepsilon_{d,m_1,\ldots,m_r}\in \{-1,+1\}$ is a sign that depends only on the dimensions. In the special case that $d=rk$ and 
all $m_i=(r-1)k$, $r\geq 2$ and $k\geq 1$, we abbreviate the notation for the sign to $\varepsilon_{r,k}$, and it is given by
\begin{equation}
\label{eq_epsilon}
\epsilon_{r,k} =
\begin{cases}
-1 & \mbox{if $k$ is odd and $r$ is $2$ mod $4$},\\
1  & \mbox{otherwise}.
\end{cases}
\end{equation}
\end{enumerate}
\end{lemma}

%%%%%%%%%%%%%%%%%%
\paragraph{Intersections of chains.}
%%%%%%%%%%%%%%%%%%
We will also need to consider oriented intersections and intersection signs for more general geometric objects,
in particular for PL submanifolds of $\R^d$ and for images of such manifolds under PL maps in general position.

A convenient framework is the following. An $m$-dimensional \define{PL chain} in $\R^d$ is a formal linear combination
$c=\sum_j a_j \sigma_j$, where the $a_j$ are integers (only finitely many nonzero) and each $\sigma_j$ is an $m$-dimensional
convex polyhedron, modulo the relation that $(-a)\sigma=a(-\sigma)$ for integers $a$ and convex polyhedra $\sigma$. 

Suppose now that $c_1,\ldots, c_r$ are PL chains in $\R^d$, $\dim c_i=m_i$ and $c_i=\sum_{i,j} a_{ij} \sigma_{ij}$,
$1\leq i\leq r$, and that the chains are in general position, i.e., $\sigma_{1j_1}, \ldots,\sigma_{r j_r}$ are in general position for any 
choice of $\sigma_{ij_i}$ in $c_i$. Then, by multilinearity, we can define the \define{oriented intersection} of the chains as the chain%\footnote{This implicitly requires Lemma~\ref{lem:prop-inters-prod}~2 in order to be well-defined.}
$$
c_1\cap \ldots \cap c_r:=\sum_{j_1,\ldots,j_r} \left(\prod_{i=1}^r a_{ij_i}\right) \sigma_{1j_1}\cap  \ldots \cap\sigma_{r j_r},
$$
with the understanding that $\sigma_{1j_1}\cap  \ldots \cap \sigma_{r j_r}=0$ if the intersection is empty.

As indicated above, we are mostly interested in the case where $c_i=f(\sigma_i)$ is the image\footnote{More precisely we mean the image chain, i.e., we slightly abuse notation here and use $f(\sigma_i)$ to denote the formal linear combination $\sum_\tau f(\tau)$, where $\tau$ ranges over all the $m_i$-simplices in a subdivision of $\sigma_i$ on which $f$ is simplexwise-linear, and each $\tau$ carries the orientation inherited from that of $\sigma_i$; a more precise but more cumbersome notation for this image chain would be $f_{\#}(\sigma_i)$.} of an $m_i$-simplex or, more generally, of an $m_i$-dimensional PL manifold $\sigma_i$ under a a PL map $f$ in general position (this includes the case that $\sigma_i$ is a submanifold of $\R^d$, we take $f$ to be the inclusion map).

Note that the dimension of $c_1\cap\ldots\cap c_r$ equals $\ell:=\sum_i m_i -d(r-1)$. In particular, if the dimensions satisfy \eqref{eq:intersection_dim_zero}, then $\ell=0$, and the intersection chain is a formal linear combination $\sum_y a_y y$ of points. 
In this case, we define the \define{algebraic intersection number} of the chains as the sum 
$$c_1\scap \ldots \scap c_r :=\sum_y a_y  \in \Z,$$
where the sum ranges over all $r$-fold intersection points $y$ in $c_1\cap\ldots \cap c_r$. 

In particular, if all (nonzero) coefficients in the chains $c_i$ are $\pm 1$ (for instance, this happens if each $c=f(\sigma_i)$ is the image of an oriented $m_i$-dimensional PL manifold, $m_i<d$) then for each point $y$ in the intersection, its coefficient $a_y$ is $\pm 1$ as well, and we call $a_y$ the \define{($r$-fold) intersection sign} of the chains at $y$, denoted
$$\isign{y}{c_1}{c_r} \in \{-1,+1\}.$$
Thus, in this case, $c_1\scap \ldots \scap c_r =\sum_y \isign{y}{c_1}{c_r}$.

Even more generally, the intersection product could be defined inside an ambient oriented PL manifold $M$ (possibly with boundary) instead of $\R^d$; however, we will only need this in the special case that $M=\sigma_1$ is itself a simplex in $\R^d$ (as in Lemma~\ref{lem:prop-inters-prod}~(c)), in which case we understand the intersection in $\sigma$ to mean the intersection in the oriented affine subspace spanned by 
$\sigma_1$.

By multilinearity, the properties in Lemma~\ref{lem:prop-inters-prod} carry over to chains in a straightforward way.\footnote{In Part~(d) the product of the chains is $c_1\times \ldots \times c_r:=\sum_{j_1,\ldots,j_r} \left(\prod_{i=1}^r a_{ij_i}\right) \sigma_{1j_1} \times \ldots \times \sigma_{r j_r}$.}

We will also need the following well-known fact about intersection numbers and boundaries:
\begin{lemma}
\label{lem:intersection-bd}
Suppose $c_1$ and $c_2$ are PL chains in general position in $\R^d$, $dim (c_i)=m_i$, $i=1,2$, and that $m_1+m_2=d+1$. 
Then 
$\partial c_1\scap c_2 = (-1)^{m_1} c_1 \scap \partial c_2.$ 
\end{lemma}

We now proceed to review the definition of oriented intersections and prove the two lemmas.
\paragraph{Orientations.} 
Specifying an orientation of an $m$-dimensional convex polyhedron $\sigma$ in $\R^d$, $m>0$, amounts to choosing an ordered basis\footnote{Here, we think of an ordered basis $B$ as a $(d\times m)$-matrix, whose columns are the basis vectors.} 
$B=[b_1|\ldots|b_m]\in \R^{d\times m}$ of the $m$-dimensional linear subspace $L(\sigma)$ parallel to $\sigma$. Given two such bases $B$ and $B'$, there is a unique invertible matrix $R\in \R^{m\times m}$ with $B'=BR$, and we say  that $B'$ and $B$ define the same or the opposite orientation of $\sigma$, denoted by $B'\sim B$ or $B'\sim \textup{op}(B)$, respectively, depending on whether $\det (R)$ is positive or negative. Equivalently, we can view orientations in terms of exterior algebra. Given a basis $B$, consider the decomposable nonzero vector $\beta=b_1\wedge \ldots \wedge b_m \in \bigwedge^m\R^d$. For two bases $B$ and $B'$, the corresponding exterior products satisfy 
$\beta'=\det(R) \cdot \beta$, and we will write $\beta'\sim \beta$ or $\beta'\sim -\beta$ depending on whether $\beta'$ and $\beta$ differ by a positive or negative factor. 

If $m=0$, i.e., if $\sigma$ is a point, then an orientation is given by a sign in $\{-1,+1\}$ assigned to that point,
and $\beta \in \bigwedge^0\R^d\cong \R$ is just a nonzero scalar. 

Note also that if $\tau \subseteq \sigma$ is a convex subpolyhedron of dimension $\ell$, then for any orientation $\alpha \in \bigwedge^\ell \R^d$ of $\tau$, we can choose\footnote{Write $\alpha=a_1\wedge\ldots\wedge a_\ell$ for some basis $A=[a_1|\ldots|a_\ell]\in \R^{d\times \ell}$ of $L(\tau)$, choose $C=[c_1|\ldots|c_{m-\ell}]$  such that $[A|C]$ is a basis of $L(\sigma)$, and set $\gamma=c_1\wedge\ldots \wedge c_{m-\ell}$.}  $\gamma\in \bigwedge^{m-\ell}\R^d$ such that $\alpha\wedge \gamma$ is an orientation of $\sigma$.

Moreover, the orientation of the boundary $\partial \sigma$ is given as follows: Let $\tau$ be a facet of $\sigma$, let $v=q-p \in \R^d$ be a vector connecting a point $p$ in the relative interior of $\sigma$ to a point $q\in \tau$ (we can think of $v$ as pointing ``outwards'' from $\sigma$ at $\tau$), and let $\alpha\in \bigwedge^{m-1}\R^d$ be any orientation of $\tau$. Then the orientation of $\tau$ in $\partial \sigma$ is given by $\pm \alpha$ depending on whether $v\wedge \alpha$ determines the chosen orientation of $\sigma$ or its opposite.

\begin{definition}
\label{def:induced_orientation_polytopes}
Let $r\geq 2$, and let $\sigma_1,\ldots,\sigma_r$ be convex polyhedra in general position in $\R^d$, $m_i:=\dim L_i$, $1\leq i\leq r$. 
Suppose we have also chosen an orientation of $\R^d$.

If $\sigma_1\cap \ldots \cap \sigma_r=\emptyset$, we consider the oriented intersection to be formally zero.

Else, $\sigma_1\cap \ldots \cap \sigma_r$ is a convex polyhedron of dimension $\ell:=\left(\sum_{i=1}^r m_i\right) - d(r-1)\geq 0$, by general position, and we proceed as follows:
\begin{enumerate}[label=\textup{(\roman*)}]
\item In the case $r=2$ of pairwise intersections, choose an arbitrary 
%ordered basis $A\in \R^{d\times \ell}$ for $L(\sigma_1\cap \sigma_2)$, and choose $B_i\in \R^{d\times(m_i-\ell)}$ such that $[A|B_i]\in \R^{d\times m_i}$ determines the chosen orientation of $\sigma_i$, $i=1,2$. (If $\ell=0$ or $\ell=m_i$ then $A$ or $B_i$, respectively, is just a sign in $\{-1,+1\}$.) 
orientation $\alpha\in \bigwedge^\ell \R^d$ of $\sigma_1\cap \sigma_2$,
and choose $\beta_i \in \bigwedge^{m_i-\ell} \R^d$ such that $\alpha\wedge \beta_i$ determines the chosen orientation of $\sigma_i$, $i=1,2$.
Then the  \define{induced orientation} on $\sigma_1\cap \sigma_2$ is given by $\alpha$ or $-\alpha$, respectively, depending on whether 
$\alpha\wedge \beta_1\wedge \beta_2 \in \bigwedge^d \R^{d}$ determines the chosen orientation of $\R^d$ or the opposite one.\footnote{It is routine to check that this does not depend on the choice of $\alpha$ or of the   $\beta_i$. Indeed, if we chose a different orientation $\alpha' \sim \varepsilon \alpha $ for $\sigma_1\cap \sigma_2$, $\varepsilon \in \{-1,+1\}$ then for any choice of corresponding ``complementary''  $\beta_i'$, we have $\beta_i'\sim \varepsilon \beta_i$ and hence $\alpha'\wedge \beta_1'\beta_2' \sim \varepsilon \alpha\beta_1\beta_2$.} The convex polyhedron $\sigma_1\cap \sigma_2$ with this induced orientation is called the \define{oriented intersection} of $\sigma_1$ and $\sigma_2$. 
%We also occasionally denote the oriented intersection by $(\sigma_1\cap \sigma_2)_{\R^d}$ if we want to stress the dependence on the oriented ambient space. 
%If $m_1+ m_2=d$, then $\sigma_1\cap \sigma_2=\{0\}$ and the induced orientation corresponds to a sign in $\{-1,+1\}$, called the  \define{(pairwise) intersection sign} and denoted by $\sign_0(\sigma_1,\sigma_2)$ or by $\sign_0^{\R^d}(\sigma_1,\sigma_2)$.
\item In general, the \define{oriented intersection} of $\sigma_1,\ldots,\sigma_r$ is defined inductively by
\begin{equation}
\label{eq:oriented_intersection_inductive}
\sigma_1\cap \ldots \cap \sigma_r:= (\sigma_1\cap \ldots \cap \sigma_{r-1}) \cap \sigma_r.
\end{equation}
(By Lemma~\ref{lem:associativity} below, we can ignore the parentheses and take the intersections in any order.) 
%In the special case that the dimensions satisfy \eqref{eq:intersection_dim_zero}, we 
%have $\sigma_1\cap\ldots\cap \sigma_r=\{0\}$, and the induced orientation is given by an \define{($r$-fold) intersection sign}, denoted by
%$$\isign{0}{\sigma_1}{\sigma_r},$$
%or by $\sign_0^{\R^d}(\sigma_1,\ldots,\sigma_r)$ if we want to emphasize the ambient space.
\end{enumerate}
\end{definition}

%\uli{Maybe it would be even more elegant to use exterior algebra and say that oriented linear subspaces are in bijection with 
%equivalence classes of nonzero \emph{decomposable} $m$-vectors $v_1\wedge \ldots \wedge v_m \in \bigwedge^m \R^d$,  
%up to multiplication of the latter by a positive scalar; in this case, the special case $m=0$ does not have to be treated separately, 
%the elements of $\bigwedge^0\R^d$ are just scalars; on the other hand, the matrix notation might be simpler for some readers.}

\begin{remark} 
\label{rem:unravel}
One can unravel the inductive definition \eqref{eq:oriented_intersection_inductive} as follows: Choose an orientation $\alpha \in \bigwedge^\ell \R^d$ for $\sigma_1\cap\ldots \cap \sigma_r$, and extend it by $\gamma_i\in \bigwedge^{d-m_i} \R^{d}$ to \emph{some} orientation $\alpha\wedge \gamma_i$ of $\bigcap_{j\neq i} \sigma_j$, $1\leq i\leq r$ (not necessarily the induced orientation). By general position, this determines signs $\varepsilon, \varepsilon_1, \ldots,\varepsilon_r \in \{-1,+1\}$ such that $\varepsilon \alpha\wedge \gamma_r\wedge \ldots \wedge \gamma_1 \in \bigwedge^d \R^d$ yields the chosen orientation of $\R^d$, and $\varepsilon_i \alpha\wedge \gamma_r \wedge \ldots \wedge \widehat{\gamma}_i\wedge \ldots \wedge \gamma_1\in \bigwedge^{m_i}\R^d$ yields the chosen orientation of $\sigma_i$, where the notation ``$\widehat{\gamma}_i$''  means that the factor $\gamma_i$ is omitted. Then the induced orientation of $\sigma_1\cap \ldots \cap \sigma_r$ is given by $\varepsilon^{r-1} \left(\prod_{i=1}^r\varepsilon_i\right) \alpha$.
\end{remark}
\begin{proof} For $r=2$, this follows immediately from Definition~\ref{def:induced_orientation_polytopes} ~(i). For $r\geq 3$, let $\alpha'=\alpha\wedge \gamma_r$. Then, by assumption, $\varepsilon_i \alpha'\wedge \gamma_{r-1}\wedge \ldots\wedge \widehat{\gamma}_i\wedge \ldots \wedge \gamma_1$ yields the chosen orientation of $\sigma_i$, $1\leq i<r$, and $\varepsilon \alpha'\wedge \gamma_{r-1}\wedge \ldots \wedge \gamma_1$ yields that of $\R^d$. Thus, by induction, $\sigma_1\cap \ldots \cap \sigma_{r-1}$ is oriented by $\varepsilon'\alpha'=\varepsilon'\alpha\wedge \gamma_r$, where $\varepsilon'=\varepsilon^{r-2}\left(\prod_{i=1}^{r-1}\varepsilon_i\right)$. Moreover, $\sigma_r$ is oriented by $\varepsilon_r \alpha\wedge \gamma_{r-1}\wedge\ldots\wedge \gamma_r$, so $(\sigma_1\cap\ldots\cap \sigma_{r-1})\cap \sigma_r$ is oriented by $\varepsilon \varepsilon' \varepsilon_r \alpha=\varepsilon^{r-1} \left(\prod_{i=1}^r\varepsilon_i\right) \alpha$.
\end{proof}

\begin{lemma}[\textbf{Associativity}]
\label{lem:associativity}
If $\sigma_1,\sigma_2,\sigma_3$ are oriented simplices in general position in $\R^d$ then we can take oriented pairwise intersections in any order and get the same induced orientation,
$$(\sigma_1\cap \sigma_2)\cap \sigma_3=\sigma_1\cap (\sigma_2\cap \sigma_3).$$
\end{lemma}

\begin{proof}[Proof of Lemma~\ref{lem:associativity} and of Lemma~\ref{lem:prop-inters-prod}~\textup{(a)--(c)}]
We may assume that $\sigma_1\cap \ldots \cap \sigma_r\neq \emptyset$, else all properties are trivially satisfied.
Moreover, Lemma~\ref{lem:prop-inters-prod}~\textup{(a)} and \textup{(b)} follow directly from the definition.

We proceed to prove Lemma~\ref{lem:associativity} and Lemma~\ref{lem:prop-inters-prod}~\textup{(c)} at the same time.
We use the notation from Remark~\ref{rem:unravel} (applied with $r=3$). By Lemma~\ref{lem:prop-inters-prod}~\textup{(a)}, 
both equations we want to establish are invariant under reversing the orientations of some $\sigma_i$ or of $\R^d$, so we 
may assume that the signs $\varepsilon$ and $\varepsilon_i$, $1\leq i\leq 3$, are all equal to $+1$. That is, we may assume
that $\R^d$ is oriented by $\alpha\wedge \gamma_3\wedge\gamma_2\wedge\gamma_1$, and that 
$\alpha\wedge \gamma_3\wedge \gamma_2$, $\alpha\wedge\gamma_3\wedge \gamma_1$, and
$\alpha\wedge\gamma_2\wedge \gamma_1$ determine the chosen orientations of $\sigma_1$, $\sigma_2$, and $\sigma_3$, respectively.

It follows directly from the definition that the induced orientation of $\sigma_1\cap \sigma_2$ is given by $\alpha\wedge\gamma_3$, and that of $(\sigma_1\cap \sigma_2)\cap \sigma_3$ is given by $\alpha$.

Moreover, $\alpha\wedge  \gamma_3\wedge  \gamma_1 \sim\alpha\wedge  \gamma_1\wedge(-1)^{(d-m_1)(d-m_3)} \gamma_3$, $\alpha\wedge \gamma_2\wedge \gamma_1 \sim \alpha\wedge \gamma_1\wedge(-1)^{(d-m_1)(d-m_2)} \gamma_2$, and $\alpha \wedge \gamma_3 \wedge \gamma_2 \wedge \gamma_1 \sim \alpha\wedge \gamma_1\wedge(-1)^{(d-m_1)(d-m_3)} \gamma_3\wedge(-1)^{(d-m_1)(d-m_2)} \gamma_2$. Thus, again applying the definition, 
the orientation of $\sigma_2\cap \sigma_3$ is given by $\alpha\wedge \gamma_1$, and hence that of $\sigma_1\cap (\sigma_2\cap \sigma_3)$ by $A$, which proves Lemma~\ref{lem:associativity}.

Similarly, the orientation of $\sigma_1\cap \sigma_3$ is given by $\alpha\wedge \gamma_2$ since $\alpha\wedge \gamma_3\wedge \gamma_2 \sim\alpha\wedge \gamma_2\wedge(-1)^{(d-m_2)(d-m_3)} \gamma_3$ and $\alpha\wedge \gamma_3\wedge \gamma_2\wedge \gamma_1 \sim \alpha\wedge \gamma_2\wedge(-1)^{(d-m_3)(d-m_2)} \gamma_3\wedge \gamma_1$. Therefore, the orientation of
$$\big((\sigma_1 \cap \sigma_2)_{\R^d} \cap (\sigma_1\cap \sigma_3)_{\R^d}\big)_{\sigma_1}$$ 
is given by $A$ as well, which proves Lemma~\ref{lem:prop-inters-prod}~\textup{(c)}.\qedhere
\end{proof}

\begin{proof}[Proof of Lemma~\ref{lem:prop-inters-prod}~\textup{(d)}]
Suppose the orientation of $\sigma_i$ is given by $B_i\in \R^{d\times m_i}$, $1\leq i\leq r$, and that of $\R^d$ by $B \in \R^{d\times d}$.
Then the orientations of $P:=\sigma_1\times\ldots\times \sigma_r$, of the thin diagonal $\thindiag{r}{\R^d}$, and of $(\R^d)^r$, respectively, are given by matrices $M_P\in \R^{dr \times d(r-1)}$, $M_\delta \in \R^{dr \times d}$, and $M\in \R^{dr \times dr}$, where
\begin{equation}
\label{eq:orientations-product-diagonal}
M_P=\left[
\begin{array}{cccc}
B_{1} & 0    &\cdots & 0  \\
0   & B_{2}     & \cdots     & 0  \\
\vdots  &  & \ddots &  \\
0   & 0  & \cdots & B_{_r} 
\end{array}
\right], 
\quad 
M_\thin=\left[
\begin{array}{c}
B\\
B\\
\vdots \\
B
\end{array}
\right], \quad \textrm{and}\quad
M=\left[
\begin{array}{cccc}
B & 0    &\cdots & 0  \\
0   & B     & \cdots     & 0  \\
\vdots  &  & \ddots &  \\
0   & 0  & \cdots & B 
\end{array}
\right].
\end{equation}
The pairwise intersection sign $\sign_{(y,\ldots,y)}(P,\thindiag{r}{\R^d})$ equals $\pm 1$ depending on whether the determinants of 
$[M_P|M_\thin]$ and of $M$ have the same or the opposite sign, i.e.,
$$\sign \det [M_P|M_\thin] = \sign_{(y,\ldots,y)}(P,\thindiag{r}{\R^d}) \cdot \sign\det M.$$
Note that reversing the orientation of one $\sigma_i$ reverses the orientation of $P$, and reversing the orientation of $\R^d$
reverses the orientation of $\thindiag{r}{\R^d}$ and changes the orientation of $(\R^d)^r$ by a factor of $(-1)^r$. Therefore, by 
Lemma~\ref{lem:prop-inters-prod}~\textup{(a)}, Equation~\eqref{eq:r-sign-vs-product-diagonal} is invariant under such orientation reversals.
Thus, we can proceed similarly to Remark~\ref{rem:unravel}, choose bases $C_i \in \R^{d\times (d-m_i)}$ of $L(\bigcap_{j\neq i}\sigma_j)$, $1\leq i\leq r$, and we may assume that  $B=[C_r|\ldots|C_1]$ and $B_i=[C_r|\ldots|\widehat{C}_i|\ldots|C_1]$. Hence, 
$$
\sign_y(\sigma_1,\ldots,\sigma_r)=+1.
$$
Moreover,
$$
%\left[
%\begin{array}{ccccc}
%B_{1} & 0    &\cdots & 0 &  B \\
%0   & B_{2}     & \cdots     & 0 &  B \\
%\vdots  &  & \ddots & & \vdots \\
%0   & 0  & \cdots & B_{_r} & B
%\end{array}
%\right]
[M_P|M_\thin] = 
\left[
\begin{array}{ccccc}
[C_r |\ldots|C_2] & 0    &\cdots & 0 &  [C_r|\ldots|C_1] \\
0   & [C_r|\ldots|{C}_3|C_1]      & \cdots     & 0 &  [C_r|\ldots|C_1] \\
\vdots  &  & \ddots & & \vdots \\
0   & 0  & \cdots & [C_{r-1}|\ldots|C_1] & [C_r|\ldots|C_1]
\end{array}
\right]
$$
By subtracting columns from one another (which does not change the orientation class), we can bring $[M_P|M_{\thin}]$ into the form
\begin{equation*}
%\label{eq:almost-there}
\left[
\begin{array}{ccccc}
[C_r |\ldots|C_2] & 0    &\cdots & 0 &  [\; 0\;| \ldots|\;0\;|C_1] \\
0   & [C_r|\ldots|{C}_3|C_1]      & \cdots     & 0 &  [\;0\;|\ldots|C_2|\;0\;] \\
\vdots  &  & \ddots & & \vdots \\
0   & 0  & \cdots & [C_{r-1}|\ldots|C_1] & [C_r|\;0\;|\ldots|\;0\;]
\end{array}
\right],
\end{equation*}
and this matrix can be transformed into  
$$
\left[
\begin{array}{ccccc}
[C_r |\ldots|C_1] & 0    &\cdots & 0\\
0   & [C_r|\ldots |C_1]      & \cdots     & 0 \\
\vdots  &  & \ddots & \vdots \\
0   & 0  & \cdots & [C_{r}|\ldots|C_1]
\end{array}
\right] =M;
$$
by a sequence of $t_{d,m_1,\ldots,m_r}:=\sum_{i=1}^r (r-i)d(d-m_i)\; + \sum_{1\leq i<j\leq r} (d-m_i)(d-m_j)$
column transpositions, which proves \eqref{eq:r-sign-vs-product-diagonal} if we set
\begin{equation}
\label{eq:number-transpositions}
\varepsilon_{d,m_1,\ldots,m_r}:=(-1)^{t_{d,m_1,\ldots,m_r}}.
\end{equation}
In the special case that $m_i=m=(r-1)k$ and $d=rk$, $k\geq 1$, the number of transpositions equals
$$
t_{r,k}:=d(d-m)\binom{r}{2}+(d-m)^2\binom{r}{2}=\frac{(r-1)r(r+1)k^2}{2},
$$
and it is easy to verify that setting $\varepsilon_{r,k}:=(-1)^{t_{r,k}}$ yields \eqref{eq_epsilon}.
\end{proof}

\begin{proof}[Proof of Lemma~\ref{lem:intersection-bd}]
By multilinearity, it suffices to prove the formula for simplices $\sigma_1,\sigma_2$ in general position in $\R^d$, $\dim(\sigma_i)=m_i$, $m_1+m_2=d+1$, $\sigma_1\cap \sigma_2\neq \emptyset$. By general position, $\sigma_1\cap \sigma_2$ is a line segment with endpoints $p \in \partial \sigma_1 \cap \sigma_2$ and $q\in \sigma_1 \cap \partial \sigma_2$, where $p$ lies in the relative interior of $\sigma_1$ and of some facet $\tau_2$ of $\sigma$, and $q$ lies in the relative interiors of $\sigma_2$ and some facet $\tau_1$ of $\sigma_1$, see Figure~\ref{fig_two_triangles}. We need to show that $\sign_q(\tau_1,\sigma_2) = (-1)^{m_1}\sign_p(\sigma_1,\tau_2)$.

\begin{figure}[h]
\begin{center}
\includegraphics[scale=1]{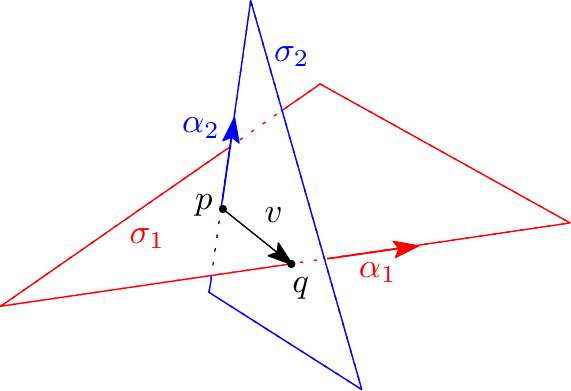}
\caption{Two triangles in general position in $\R^3$.}
\label{fig_two_triangles}
\end{center}
\end{figure}
Suppose the orientation of $\sigma_i$ is given by $\beta_i \in \bigwedge^{m_i}\R^d$ and the orientation of $\tau_i$ in $\partial\sigma_i$ is given by $\alpha_i \in \bigwedge^{m_i-1}\R^d$, $i=1,2$, and that $\R^d$ is oriented by $\beta \in \bigwedge^d\R^d$. Then, by definition,
the intersection signs 
$\sign_q(\tau_1,\sigma_2)$ and $\sign_p(\sigma_1,\tau_2)$ are determined by
\begin{equation}
\label{eq:eq-intersection-boundary}
\beta \sim \sign_q(\tau_1,\sigma_2)\cdot \alpha_1\wedge \beta_2 \sim \sign_p(\sigma_1,\tau_2)\cdot \beta_1\wedge \alpha_2
\end{equation}
Let $v:=q-p$. Then, by definition of the orientation of the boundary, $\beta_1\sim v\wedge \alpha_1$ and $\beta_2\sim (-v)\wedge \alpha_2$.
It follows that $\beta_1\wedge \alpha_2\sim  v\wedge \alpha_1\wedge \alpha_2 \sim (-1)^{m_1} \alpha_1\wedge (-v)\wedge \alpha_2 \sim (-1)^{m_1} \alpha_1\wedge \beta_2$.\end{proof}

%%%%%%%%%%%%%%%%%%%%%%%%%%%%%%%%%%%%%%%%
\section{The Higher-Multiplicity Whitney Trick} 
\label{sec_whitney_trick}
%%%%%%%%%%%%%%%%%%%%%%%%%%%%%%%%%%%%%%%%

In this section, we present the proof of Theorem~\ref{thm_whitney_trick_extended}. The proof is by induction on $r$. The base case $r=2$ is the PL version of the Whitney Trick (see, e.g., Weber~\cite{Weber}). 

Thus, inductively, we may assume that $r\geq 3$ and that the theorem holds for $r-1$. We proceed in three steps,
each of which is explained in detail in the corresponding subsection.

\begin{enumerate}
\item[\ref{subsec_reduction}]
We show how we can restrict ourselves to a standard local situation, in which $m_i$-dimensional balls $\sigma_i$ properly contained in a $d$-ball $B^d$, $1\leq i\leq r$, intersect in precisely two $r$-intersection points $x$ and $y$ of opposite signs.

\item[\ref{subsec_piping}]%[\ref{subsec_res_sigma_1}]
If we restrict ourselves to the sub-ball $\sigma_1\subseteq B^d$, then $x$ and $y$, seen as $(r-1)$-intersection points between $\sigma_1 \cap \sigma_2, \ldots, \sigma_1 \cap \sigma_r$ inside the $m_i$-ball $\sigma_1$, still have opposite signs. 
Moreover, we show that we can modify each $\sigma_1\cap \sigma_i$, $2\leq r\leq r$, by an ambient isotopy of $B^d$ (which corresponds to performing a pair of complementary ambient surgeries on $\sigma_i$) so that the pairwise intersections 
$\sigma_1\cap \sigma_i$ become connected.
\item[\ref{subsec_proof_whitney_extended}]
Inductively, we remove the $(r-1)$-intersection points between $\sigma_1 \cap \sigma_2, \ldots, \sigma_1 \cap \sigma_r \subseteq \sigma_1$ by ambient isotopies of $\sigma_1$ and then extend these to ambient isotopies of $B^d$, using that $\sigma_1$ is unknotted in $B^d$, so that $B^d \cong \sigma_1 * S^{d-n_1-1}$.
\end{enumerate}

\subsection{Reduction to a Standard Local Situation} 
\label{subsec_reduction}
The first step of the proof of Theorem~\ref{thm_whitney_trick_extended} is to reduce the problem to the following local situation:

%\uli{TODO: Add picture.}

\begin{definition} 
\label{def:standard-local-situation}
We say that $B\subset \R^d$ and $\sigma_1,\ldots, \sigma_r \subset B$ form a \define{standard local situation} around two
$r$-fold points $x,y$ if the following properties are satisfied:
\begin{enumerate}
\item $B \subset \R^d$ is a $d$-dimensional PL ball, with $x,y$ in the interior $\interior B$.
\item For $1\leq i\leq r$, $\sigma_i $ is an $m_i$-dimensional PL ball properly embedded (see Section~\ref{sec:PLbackground}) into $B$, with
\begin{equation}
\tag{\ref{eq:dim-critical}}
\sum_{i=1}^r m_i  = d(r-1).
\end{equation}
\item $\sigma_1,\ldots, \sigma_r$ are mutually transverse (see Section~\ref{sec:PLbackground}), $\sigma_1\cap \ldots \cap \sigma_r=\{x,y\}$, and for each index set $J\subseteq \{1,\ldots,r\}$ with $|J|\geq 2$, $\bigcap_{j\in J}\sigma_j$ is the disjoint union of two PL balls $B_{J,x} \ni x$ and $B_{J,y} \ni y$ (each properly embedded in $B$ and of dimension $d-\sum_{j\in J}(d-m_j)$, by transversality).
\end{enumerate} 
\end{definition}

\begin{lemma}[\textbf{Reduction to a standard local situation}]
\label{lem:localization}
Let $M_1,\ldots,M_r$ be connected PL manifolds\footnote{For the proof of the generalized Van Kampen--Shapiro--Wu theorem, we would only need the two cases that $M_i\cong \simplex^{m_i}$ is a PL ball, or that $M_i\cong \partial\simplex^{m_i}\times [0,1]$ is a PL cylinder.} (possibly with boundary) of respective dimensions $\dim M_i=m_i$, $1\leq i \leq r$, such that $\sum_{i=1}^r m_i  = d(r-1)$ and
\begin{equation}
\tag{\ref{eq:codimension-3}}
d-m_i\geq 3, \qquad 1\leq i\leq r. 
\end{equation}
Suppose that $f\colon M_1 \sqcup \ldots \sqcup M_r \to \R^d$ is a PL map in general position defined on the disjoint union of the $M_i$, and let $$x,y \in f(M_1)\cap \ldots \cap f(M_r)$$ 
be two $r$-fold points of $f$.

Then there exists a $d$-dimensional PL ball $B \subset \R^d$ such that $B$ and 
$\sigma_i:=f(M_i)\cap B$, $1\leq i \leq r$, form a standard local situation around $x$ and $y$.

Moreover if $L \subseteq \R^d$ is any compact polyhedron of dimension at most $d-3$ and disjoint from $x$ and $y$ then we can choose $B$ to be disjoint from $L$.

Furthermore, if $B'$ is a $d$-dimensional $PL$ ball %in general position with respect to the image of $f$ and 
such that $x,y\in \interior{B}'$ and $x$ and $y$ lie in the same connected component of 
$f(M_i)\cap \interior B'$, $1\leq i\leq r$, 
then we can choose $B$ to be contained in $\interior B'$.
\end{lemma}
\begin{proof}
For each $i$, let us use the abbreviation $\singset_{M_i}$ for the closed singular set of $f|_{M_i}$ (see Section~\ref{sec:PLbackground}), so that $f(S_{M_i})$ is the closure of the set of double points of $f|_{M_i}$. Since $f$ is 
in general position, the images $f(M_i)$ intersect transversely at $x$ and at $y$, each pairwise intersection $f(M_i)\cap f(M_j)$ 
has dimension $m_i+m_j-d$, and $f(\singset_{M_i})$ has dimension at most $2m_i-d$ and is at positive distance from $x$ and 
$y$.

For each $i$, we choose a PL path $\lambda_i \subseteq f(M_i)$ connecting $x$ and $y$. By choosing $\lambda_i$ to be in general position within $f(M_i)$, we can guarantee that $\lambda_i$ intersects the other $f(M_j)$, $j\neq i$, only in $x$ and $y$, and that $\lambda_i$ is disjoint from 
$f(\singset_{M_i})$, see Figure~\ref{fig_intersection_in_sigma_1}; here, we use that,  by \eqref{eq:codimension-3}, both $f(M_i)\cap f(M_j)$ and $f(\singset_{M_i})$ have codimension at least $3$ within $f(M_i)$ (in fact, codimension $2$ would be enough).
\begin{figure}[h]
\begin{center}
\includegraphics[scale=.8]{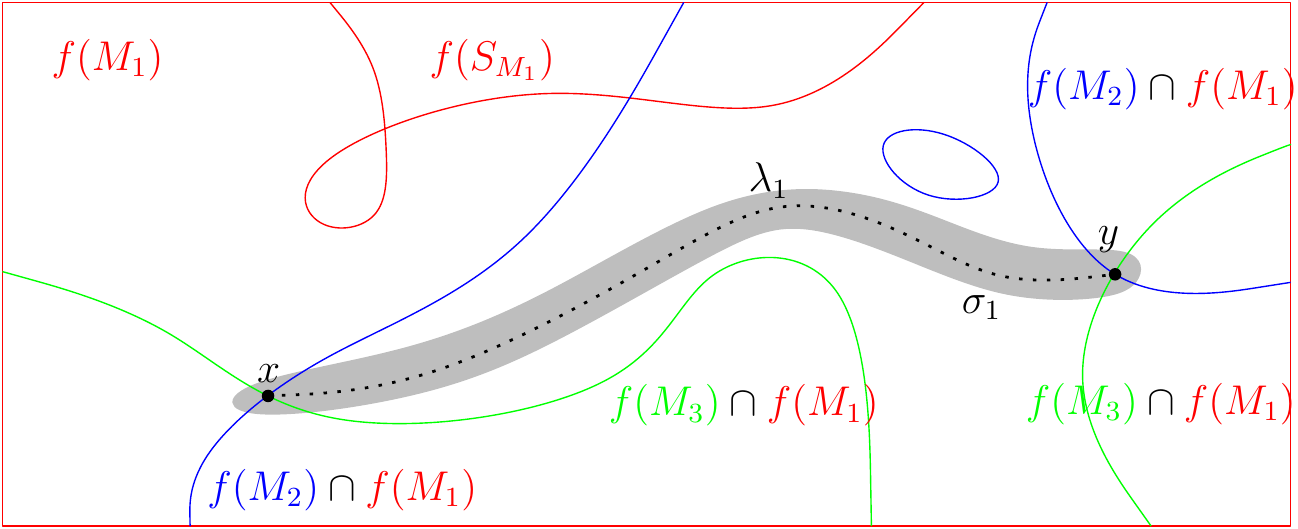}
\caption{On $f(M_1)$, the path $\lambda_1$ joins $x$ and $y$. Any sufficiently small regular neighborhood $\sigma_1$ of $\lambda_1$ in $f(M_1)$ is an $m_1$-dimensional PL ball.}
\label{fig_intersection_in_sigma_1}
\end{center}
\end{figure}

The union $\lambda_1 \cup \lambda_2$ is an embedded circle in $\R^d$, and, again using general position,\footnote{Indeed, we can take $D_{12}$ to be the cone over $\lambda_1\cup \lambda_2$ with an apex in general position.} we can fill it with an embedded $2$-dimensional PL disk $D_{12}$ that intersects $f(M_1)$ and $f(M_2)$ precisely in $\lambda_1$ and $\lambda_2$, respectively, that intersects all other $f(M_i)$, $i\neq 1,2$ precisely in $\{x,y\}$, and that is disjoint from all $f(S_i)$ (see Figure~\ref{fig_lambda_1_lambda_2}); here, we require codimension at least $3$.
\begin{figure}[h]
\begin{center}
\includegraphics[scale=1]{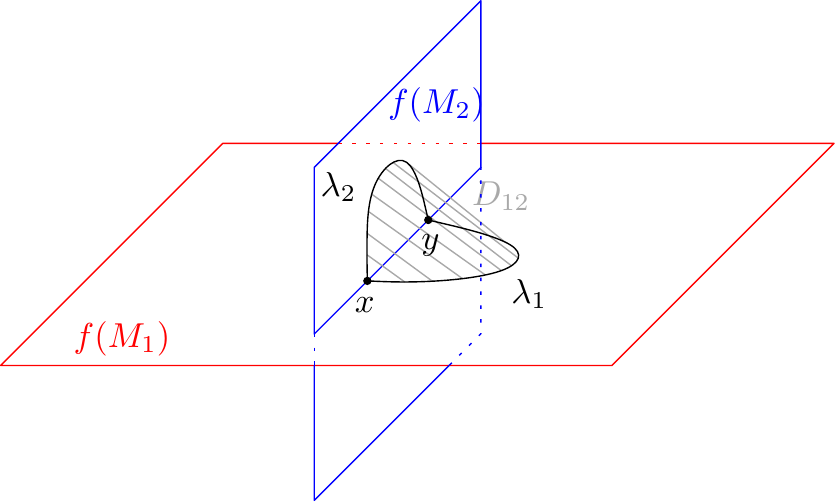}
\caption{The disk $D_{12}$ fills the circle $\lambda_1 \cup \lambda_2$.}
\label{fig_lambda_1_lambda_2}
\end{center}
\end{figure}

Repeating the same construction on each successive circle $\lambda_i \cup \lambda_{i+1}$, $1\leq i\leq r-1$, we get the sequence of filling disks
\[
D_{12}, D_{23}, \ldots, D_{(r-1)r}.
\]
By \eqref{eq:codimension-3}, we have $d\geq 3r \geq 6$, so by general position, the these filling disks are internally disjoint and their union is a disk $D$ with boundary $\lambda_1 \cup \lambda_r$.

We pick a regular neighborhood $B$ of $D$; this neighborhood is a $d$-dimensional PL ball. If we pick this neighborhood sufficiently small then $B$ intersects each image $f(M_i)$ in an $m_i$-dimensional PL ball $\sigma_i$ that is a regular neighborhood of $\lambda_i$, and we get Property~3 of the standard local situation since the images $f(M_i)$ intersect transversely at $x$ and at $y$.

Furthermore, if $L$ and $B'$ are as in the statement of the lemma, then we can choose the paths $\lambda_i$ and the disks $D_{i(i+1)}$ to be contained in $\interior B'$ and to avoid $L$, and hence the same holds for any sufficiently small regular neighborhood $B$ of $D$.
\end{proof}
\begin{remark}
If we apply the preceding lemma to a finite collection of pairwise disjoint pairs $\{x,y\}$ of $r$-fold points, then by general position, we can choose the resulting disks $D$, and hence the corresponding regular neighborhoods $B$ to be pairwise disjoint.
\end{remark}

Using Lemma~\ref{lem:localization}, Theorem~\ref{thm_whitney_trick_extended} reduces to the following:
\begin{proposition} 
\label{prop_reduced_whitney_tric}
Suppose that $B \subset \R^d$ and $\simplex_1,\ldots,\simplex_r \subset B$ form a standard local situation around two $r$-fold points 
$x,y \in \interior B$, and that the codimension condition \eqref{eq:codimension-3} is satisfied.

Suppose furthermore that $x$ and $y$ have opposite intersection sign, i.e., for some (and then every) choice of orientations of $\R^d$ and of the $\sigma_i$,
$$\sign_x(\sigma_1,\ldots,\sigma_r)=-\sign_y(\sigma_1,\ldots,\sigma_r).$$
Then there exist $r-1$ PL ambient isotopies 
$$H^2 , \ldots , H^r \colon B \times [0,1]\to B\times [0,1],$$
each  fixing $\boundary B$ pointwise, 
such that 
\[
\sigma_1 \cap H^2_1  (\sigma_2) \cap \cdots \cap H^r_1  (\sigma_r) = \emptyset.
\]
\end{proposition}
\begin{proof}[Proof of Theorem~\ref{thm_whitney_trick_extended} using Proposition~\ref{prop_reduced_whitney_tric}]

Using Lemma~\ref{lem:localization}, we show that if Proposition~\ref{prop_reduced_whitney_tric} holds for a given multiplicity $r\geq 2$, 
then so does Theorem~\ref{thm_whitney_trick_extended}.

Suppose the hypotheses of Theorem~\ref{thm_whitney_trick_extended} are satisfied. Apply Lemma~\ref{lem:localization} 
to get a PL $d$-ball $B$ disjoint from $L$ and such that $B$ and $\sigma_i:=B\cap f(M_i)$ form a standard local situation around
the pair $x,y$ of $r$-fold points in question. By assumption, these points have opposite signs (here, we use that intersection signs are determined locally, so that it does not matter whether we restrict $f(M_i)$ to its intersection with $B$). Let $H^2_t,\ldots,H^r_t\colon B \to B$ be the isotopies guaranteed by Proposition~\ref{prop_reduced_whitney_tric}. Since they are fixed pointwise on $\partial B$, we can extend each $H^i_t$ to an isotopy of $\R^d$ by letting it fix every point outside of $B$; slightly abusing notation, we denote the resulting isotopies by the same symbol. Then the intersection $f(M_1)\cap H^2_1(f(M_2))\cap \ldots \cap H^r_1(f(M_r))$ does not contain any points from $B$ (in particular, it does not contain $x$ or $y$), and it coincides with $f(M_1)\cap \ldots \cap f(M_r)$ outside of $\interior B$.
\end{proof}

\subsection{Restriction to \texorpdfstring{$\sigma_1$}{sigma1}, Piping and Unpiping} %
%\label{subsec_res_sigma_1}
\label{subsec_piping}
\noindent To prove Proposition~\ref{prop_reduced_whitney_tric}, the idea is to restrict ourselves to  $\sigma_1$, and to consider 
$x$ and $y$ as $(r-1)$-fold intersection points of the pairwise intersections $\sigma_1\cap \sigma_2, \ldots, \sigma_1\cap \sigma_r$ 
inside the $m_1$-dimensional ball $\sigma_1$). The plan is to solve the situation inductively inside $\sigma_1$, and then to extend the solution, i.e., the resulting isotopies of $\sigma_1$ fixing $\partial \sigma_1$, to isotopies of $B$, using that $\sigma_1$ is unknotted in $B$.
%, i.e., that $B\cong \sigma_1 * S^{d-m_1-1}$.

Each $\sigma_1\cap \sigma_i$ is a PL manifold with boundary properly embedded in $\sigma_1$, of codimension 
$$m_1-\dim \sigma_1\cap \sigma_i =d-m_i\geq 3,\qquad 2\leq i\leq r.$$
We now fix orientations of $\sigma_1,\ldots,\sigma_r$ and of $B$ and consider the induced orientations on $ \sigma_1\cap \sigma_i$, $2\leq r$.
By Lemma~\ref{lem:prop-inters-prod}, 
$$\sign_x^{\sigma_1}( \sigma_1\cap \sigma_2,\ldots, \sigma_1\cap \sigma_r)=\sign_x^{B}(\sigma_1,\ldots, \sigma_r),$$
and likewise for $y$. Thus, with respect to the induced orientations, 
$x$ and $y$ have opposite intersection signs as $(r-1)$-fold intersection points of $\sigma_1\cap \sigma_2,\ldots, \sigma_1\cap \sigma_r$ in $\sigma_1$.

However, there is a caveat that prevents us from directly proceeding by induction:  The pairwise intersections  are \emph{not connected}; indeed, by the hypotheses of Proposition~\ref{prop_reduced_whitney_tric}, each $ \sigma_1\cap \sigma_i$ is the disjoint union of two PL balls
$B_{i,x} \ni x$ and $B_{i,y}\ni y$ of dimension $m_1+m_i-d$, $2\leq i\leq r$, see Figure~\ref{fig_intersection_on_sigma_1_before_piping}.
\begin{figure}[h]
\begin{center}
\includegraphics[scale=0.7]{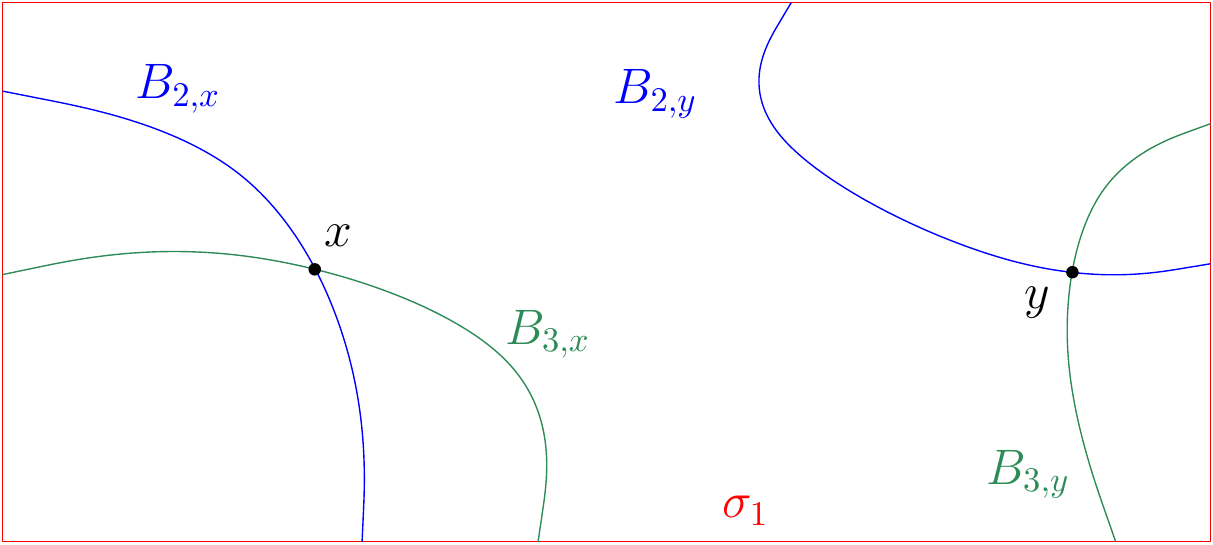}
\caption{The pairwise intersections $\sigma_1\cap \sigma_i$, $2\leq i\leq r$ are not connected.}
\label{fig_intersection_on_sigma_1_before_piping}
\end{center}
\end{figure}

Thus, the fact that $x$ and $y$ have opposite signs is no longer independent of the choice of orientations; indeed, 
if we revert the orientation on one of the components of $\sigma_1\cap \sigma_2$, say, then the signs become the same.
More importantly, in this situation there are simply no ambient isotopies $H^3_t,\ldots,H^r_t \colon \sigma_1\to \sigma$
fixing $\boundary \sigma_1$ that eliminate the intersection points. For example, in the case $r=3$ depicted in Figure~\ref{fig_intersection_on_sigma_1_before_piping}, the ball $B_{2,x}$ and the boundary $\boundary B_{3,x}$ are \emph{linked} in $\sigma_1$, i.e., for any homeomorphism fixing $\boundary \sigma_1$, we have $B_{2,x}\cap h(B_{3,x})\neq \emptyset$.

To remedy this shortcoming, we apply two operations, \emph{piping} and \emph{unpiping}, to be described presently, 
to the simplices $\sigma_2,\ldots,\sigma_r$ to force connectivity of the intersections $\sigma_1\cap \sigma_i$, $2\leq i\leq r$.
These operations correspond to a pair of \emph{complementary surgeries} (see below) performed on each $\sigma_i$, $2\leq i\leq r$. 
First, we perform a $1$-surgery on $\sigma_i$ to produce a manifold $\sigma_i^*$, and then we perform a complementary $2$-surgery on $\sigma^*$ to obtain a manifold $\sigma_i^{**}$ that is again an $m_i$-dimensional ball. Moreover, these surgeries are performed in an ambient way inside $B^d$, keeping the boundaries of the $\sigma_i$ and of $B^d$ fixed and not affecting the intersection points $x$ and $y$, such that $\sigma_1\cap \sigma_i^* = \sigma_1\cap \sigma_i^{**}$ is connected. We now describe this in more detail.

\paragraph{Surgeries and Handles.} Let $M$ be an $m$-dimensional PL manifold (possibly with boundary). Suppose that we have a
 PL embedding of $\alpha \colon S^{p-1} \hookrightarrow \interior M$, and that $\alpha$ can be extended to an embedding 
$\psi \colon  S^{p-1}\times B^{m-p+1} \hookrightarrow \interior M$, where we identify $S^{p-1}$ with $S^{p-1}\times 0 \subset S^{p-1}\times B^{m-p+1}$. Then we can use the fact that $\boundary (S^{p-1}\times B^{m-p+1}) = \boundary(B^p \times S^{m-p})=S^{p-1}\times S^{m-p}$,
remove the interior of the image $\psi(S^{p-1}\times B^{m-p+1})$ from $M$, and patch the resulting ``hole'' by attaching $B^p \times S^{m-p}$ via the attaching map $\psi|_{S^{p-1}\times S^{m-p}}\colon S^{p-1}\times S^{m-p} \to \interior M$, i.e., form the new manifold
\begin{equation*}
%\label{eq:after-p-surgery}
M':=M\setminus \textrm{int}\, \psi(S^{p-1}\times B^{m-p+1}) \cup_{\psi|_{S^{p-1}\times S^{m-p}}}  B^p \times S^{m-p}.
\end{equation*}
We refer to this operation as \define{attaching a hollow $p$-handle} $B^p \times S^{m-p}$ to $M$ or 
performing a \define{$p$-surgery} on $M$ along $\alpha$. (Note that this does not affect the boundary $\partial M$.)

If $M \subset \partial W$ is PL embedded on the boundary of an $(m+1)$-dimensional PL manifold $W$, then the operation just described corresponds to attaching a \define{solid $p$-handle}  $B^p\times B^{m-p+1}$ to $W$ to obtain a new $(m+1)$-manifold $W'$, as described in \cite[Chapter~6, p.74]{Rourke:Introduction-to-piecewise-linear-topology-1982} (where the embedded sphere $\alpha(S^{p-1})$ is called the \emph{$a$-sphere} of the solid 
$p$-handle). The $p$-surgery describes how $M$ and $\partial W$ change when attaching the $p$-handle to $W$.
We remark that our use of the adjectives \emph{hollow} and \emph{solid} is slightly nonstandard (in \cite[Chapter~6]{Rourke:Introduction-to-piecewise-linear-topology-1982}, solid handles are simply called handles).

Suppose now that after obtaining $M'$ from $M$ by a $p$-surgery along $\alpha$ as described above, we perform a $(p+1)$-surgery 
on $M'$ along an embedding $\beta\colon S^p \hookrightarrow \interior M'$ to obtain another manifold $M''$. 
We say that these two surgeries are \define{complementary} if the embedded spheres $\beta(S^p)$ and $\{0\} \times S^{m-p}$ in $M'$ are in general position and have algebraic intersection number $\pm 1$ (with respect to some arbitrarily chosen orientations); we call the sphere 
$\{0\} \times S^{m-p}$ the \define{cocore sphere} of the $p$-surgery. (This corresponds to complementarity of the solid $p$-handle attached to $W$ and the solid $(p+1)$-handle attached to $W'$, as described in \cite[Chapter~6, pp.~76--80]{Rourke:Introduction-to-piecewise-linear-topology-1982}, where the cocore sphere $\{0\} \times S^{m-p}$ is called the \emph{$b$-sphere}; it is the boundary of the cocore ball $\{0\} \times B^{m-p+1}$ of the solid $p$-handle attached to $W$.)

The main fact we will need is the following: 
\begin{lemma}
\label{lem:cancellation}
If $M''$ is obtained from $M$ by performing a $p$-surgery followed by a complementary $(p+1)$-surgery, then $M''$ and $M$ are PL homeomorphic.
\end{lemma}
This is essentially the cancellation lemma for handle theory \cite[Lemma~6.4]{Rourke:Introduction-to-piecewise-linear-topology-1982}, which states that if $W''$ is obtained from $W$ by attaching a $p$-handle and then a complementary $(p+1)$-handle, then there is a PL homeomorphism $W \cong W''$ that is the identity outside of a neighborhood of the two handles (so that it restricts to a PL homeomorphism $M\cong M''$).

\paragraph{Piping \cite[pp.~67--68]{Rourke:Introduction-to-piecewise-linear-topology-1982}.} Let $M_1$ and $M_2$ be two %connected 
disjoint $m$-dimensional submanifolds of $B^d$, with $d-m \geq 3$. The \emph{piping} technique consists of forming a new submanifold $M_3$ homeomomorphic to the connected sum $M_1 \# M_2$ as follows \cite[p.~46]{Rourke:Introduction-to-piecewise-linear-topology-1982}: Pick two points $p_i\in M_i$, $i=1,2$, and choose a path $\lambda$ in $B^d$ that connects $p_1$ and $p_2$; by general position, we can assume that $\lambda$ is disjoint from the $M_i$ except at its endpoints and that $\lambda$ avoids any given obstacle (closed polyhedron) of codimension at least $2$. Remove the interiors of two small $m$-dimensional balls $B_1$ and $B_2$ around $p_1\in M_1$ and $p_2\in M_2$ and patch the resulting holes by a an embedded cylinder $Z\cong S^{m-1} \times [-1,+1]$ along $\lambda$, the \define{piping tube}, see Figure~\ref{fig_piping_trick}. Thus, $Z$ intersects  $M_1\cup M_2$ precisely in $\partial T=\partial B_1\cup \partial B^1$, and $M_3=(M_1\cup M_2) \setminus (\interior B_1 \cup \interior B_2) \cup Z$.
The sphere $S^{m-1}\times \{0\}\subset Z$ is the \define{cocore sphere} of the piping. If both $M_1$ and $M_2$ are oriented, then the piping can be performed in such a way that $M_3$ is oriented compatibly with both given orientations.
\begin{figure}[h]
\begin{center}
\includegraphics[scale=0.9]{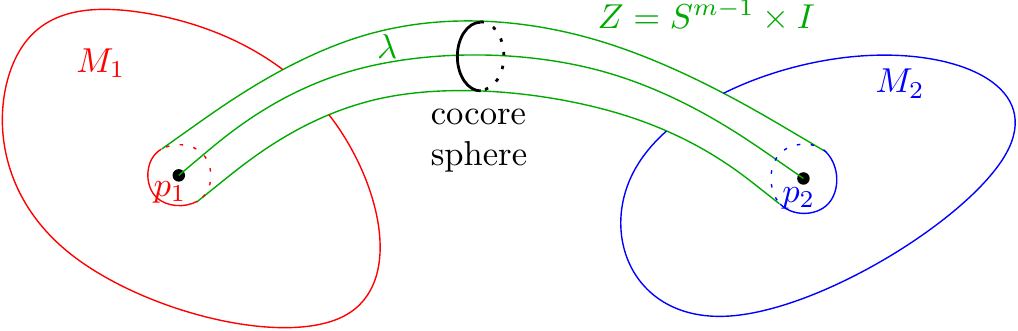}
\caption{Piping of two submanifolds.}
\label{fig_piping_trick}
\end{center}
\end{figure}

Somewhat more formally, the piping tube can be described as follows:
\begin{proposition}[{\cite[Proposition~5.10]{Rourke:Introduction-to-piecewise-linear-topology-1982}}]
\label{prop:piping}
Let $\lambda$ be as above. Let $(N,N_1,N_2)$ be a regular neighborhood of $\lambda$ in $(B^d,M_1,M_2)$.
Then there is a PL homeomorphism 
$$h\colon (N, N_1, N_2)\cong ([-1,+1]^{d-1}\times [-2,2], [-1,1]^{m} \times 0^{d-1-m} \times \{-1\} ,[-1,1]^m \times 0^{d-1-m} \times \{1\}),$$
and $h$ can be chosen to preserve any given orientations (for this, $d-m\geq 2$ would suffice). 
The piping tube can be taken to be $Z=\partial[-1,1]^m \times 0^{d-1-m} \times [-1,1]$.
\end{proposition}
If $M_1$ and $M_2$ are submanifolds of an $m$-manifold $M$, then piping corresponds to performing a $1$-surgery on $M$, in an ambient way inside $B^d$, with the hollow $1$-handle embedded as the piping tube. If $M$ is oriented, we use that the piping tube can be given an orientation compatible with that of $M$ at both ends, so that the resulting manifold $M'$ is again orientable.

Moreover, the piping tube is \emph{unique} up to ambient isotopy of $B^d$ fixed on $M \cup \partial B^d$, in the following sense \cite[Exercise, p.~68]{Rourke:Introduction-to-piecewise-linear-topology-1982}: Consider two PL paths $\lambda$ and $\lambda'$ in general position with endpoints $p_1$ and $p_2$ (and otherwise disjoint from $M$). By general position, using $d-m\geq 3$, there is an isotopy $F$ between $\lambda \cup M \subset B^d$ and $\lambda' \cup M \subset B^d$, fixed on $M$ and such that $F^{-1}(\partial Q\times [0,1])=\partial M\times [0,1]$ (so $F$ is allowable, see Section~\ref{sec:PLbackground}). By Proposition~\ref{prop:Hudson-polyhedra}, there is an ambient isotopy $H$ of $B^d$, fixed on $M\cup \partial Q$, such that $H_1(\lambda)=\lambda'$. Thus, by the uniqueness of regular neighborhoods up to ambient isotopy, any piping tube along $\lambda$ is ambient isotopic to any piping tube along $\lambda'$.

\paragraph{Piping simultaneously in $\boldsymbol{\sigma_1}$ and in $\boldsymbol{B^d}$.} We now apply this to each $\sigma_i$, $2\leq i \leq r$ to make 
the pairwise intersections
\[
\sigma_1 \cap \sigma_i \iso B_{i,x} \sqcup B_{i,y}.
\]
connected: For each $i$, $2\leq i\leq r$, we pick two points $b_{i,x} \in B_{i,x}$ and $b_{i,y} \in B_{i,y}$ and not contained in any
other $\sigma_j$, $j\not\in \{1,i\}$. We connect $b_{i,x}$ and $b_{i,y}$ by a path $\lambda_i$ in $\sigma_1$; by general position, we may assume that $\lambda_i$ avoids $\sigma_1 \cap \sigma_j$, $j \not \in\{1, i\}$. We now perform an ambient $1$-surgery on $\sigma_1$, i.e., we run a piping tube from $\sigma_i$ to itself along $\lambda_i$, in an orientation-compatible way, as described above. We denote the resulting piped $m_i$-manifold by $\sigma_i^*$, see Figure~\ref{fig_sigma_i_piped}. 
\begin{figure}[h]
\begin{center}
\includegraphics[scale=1]{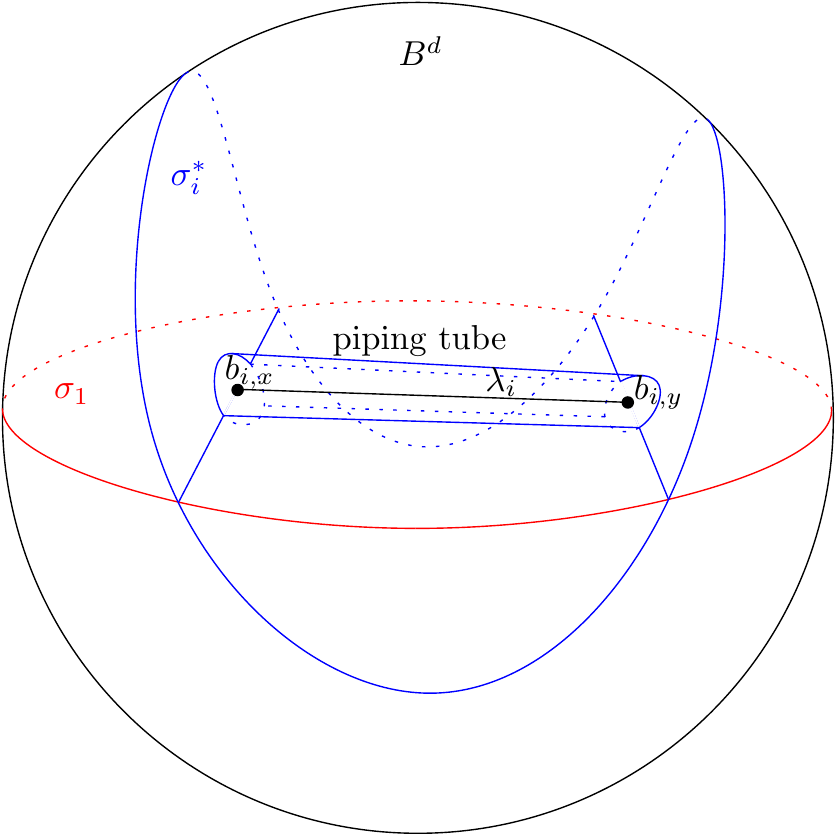}
\caption{$\sigma_i$ is piped along $\lambda_i \subset \sigma_1$, forming $\sigma_i^*$.}
\label{fig_sigma_i_piped}
\end{center}
\end{figure}

Moreover, $\sigma_1$ is unknotted in $B^d$, i.e., up to a homeomorphism of $B^d$, $\sigma_1$ is embedded as a coordinate $m_1$-ball. 
Therefore, we can take the piping tube to be transverse to $\sigma_1$. Then $\sigma_i^\ast$ is still transverse to $\sigma_1$, and 
the intersection $\sigma_1\cap \sigma_i^*$ is a piping of the two components $B_{i,x}$ and $B_{i,y}$ of $\sigma_1\cap \sigma_i$,
see Figure~\ref{fig_intersection_on_sigma_1_after_piping}). 
\begin{figure}
\begin{center}
\includegraphics[scale=0.7]{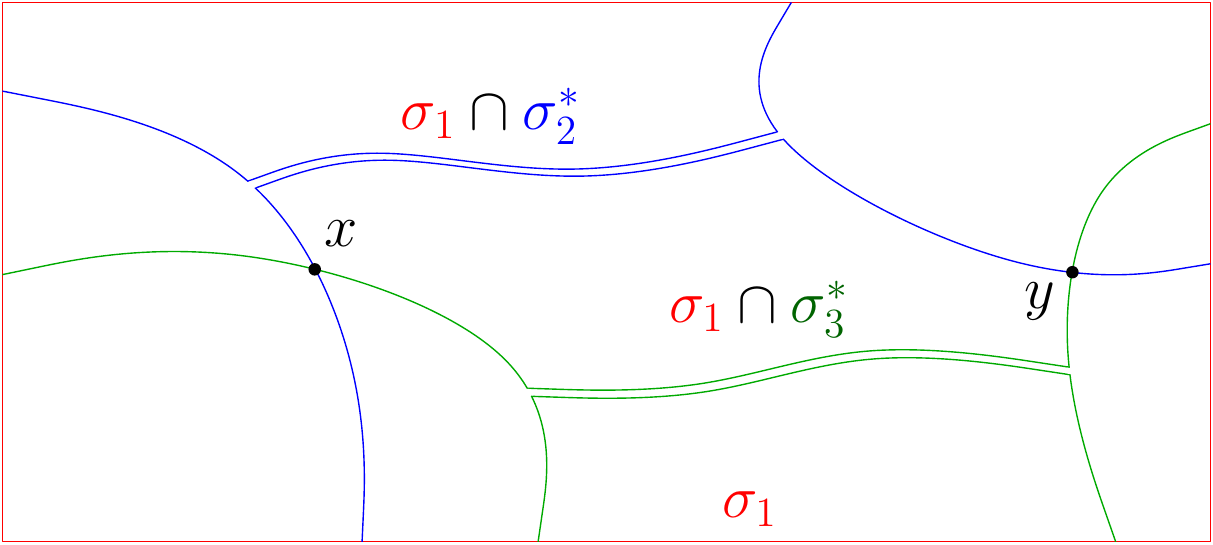}
\caption{The ``piped'' surfaces $\sigma_2^*$ and $\sigma_3^*$ intersected with  $\sigma_1$.}
\label{fig_intersection_on_sigma_1_after_piping}
\end{center}
\end{figure}
Since orientations are preserved by the pipings, $x$ and $y$ have opposite signs as $(r-1)$-fold intersections points of the connected oriented
manifolds $\sigma_1 \cap \sigma_2^*, \ldots , \sigma_1 \cap \sigma_r^*$ inside $\sigma_1$.

%\uli{TODO: Should we explain more carefully that both pipings, the one of $\sigma_i$ as well as its restriction to $\sigma_1\cap \sigma_i$, preserve orientations? Maybe it would be easier to say that we first perform that piping in an orientation-preserving way and then extend the piping to one of $\sigma_i$ inside $B^d$, using the flatness of $\sigma_1$ and the fact $\sigma_1$ and $\sigma_i$ intersect transversely?} 

\paragraph{Unpiping in $\boldsymbol{B^d}$.} As explained above, piping $\sigma_i$ corresponds to performing a $1$-surgery on $\sigma_1$,
in an ambient way inside $B^d$. In this way, we obtained a submanifold $\sigma_i^*$, with the same boundary as $\sigma_i$, such that 
$\sigma_i^* \cap \sigma_1$ is connected. However, $\sigma_i^*$ is not homeomorphic to an $m_i$-ball, so in particular, there is no isotopy of $B^d$ that transforms $\sigma_i$ into $\sigma_i^*$. 

We now describe how to amend this by performing a complementary ambient $2$-surgery on $\sigma_i^*$, which we call \define{unpiping}, 
such that the resulting manifold $\sigma_i^{**}$ is again an $m_i$-ball and such that $\sigma_1 \cap \sigma_i^{**} =\sigma_1\cap \sigma_i^*$ does not change (hence stays connected). The basic idea is shown in Figure~\ref{fig_sigma_i_unpiped}.
\begin{figure}[h]
\begin{center}
\includegraphics[scale=1]{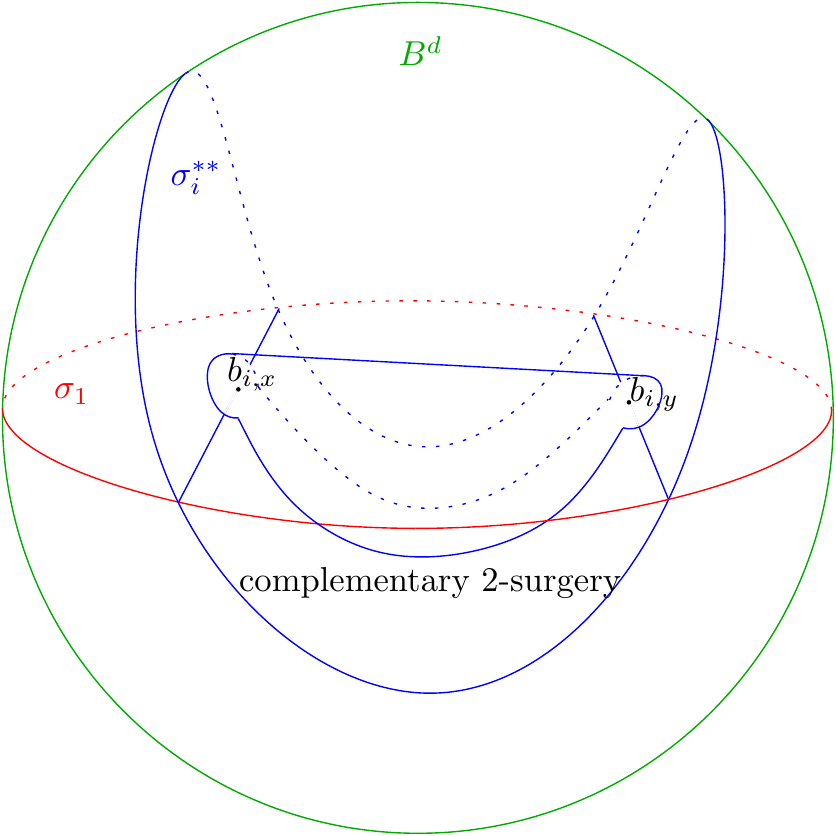}
\caption{A $1$-surgery can be cancelled by a complementary $2$-surgery, both ambient.}
\label{fig_sigma_i_unpiped}
\end{center}
\end{figure}

\begin{lemma}[\textbf{Unpiping Lemma}]
\label{lem:unpiping}
For each $i$, $2\leq i\leq r$, there is an ambient isotopy $\t{H}^i$ of $B^d$ fixed on $\partial B^d$ such that 
$\sigma_i^{**}:=\t{H}^i_1(\sigma_i)$ satisfies $\sigma_1 \cap \sigma_i^{**} =\sigma_1\cap \sigma_i^*$ and
$\sigma_i^{**} \cap \sigma_j^{**} = \sigma_i\cap \sigma_j$, $2\leq i<j\leq r$.
\end{lemma}
\begin{proof}
We need to achieve three things:
\begin{enumerate}
\item First, if we think of $\sigma_i$ and $\sigma_i^*$ as abstract (non-embedded) PL-manifolds, with $\sigma_i^*$ obtained from $\sigma_i$ by a $1$-surgery, then in order to be able to perform a complementary $2$-surgery on $\sigma_i^*$ and obtain an $m_i$-ball $\sigma_i^{**}$, we need an embedded circle $\beta_i$ in $\sigma_i^*$ that intersects the cocore circle of the $1$-surgery exactly once and such that that a small neighborhood of $\beta_i$ in $\sigma_i^*$ is PL homeomorphic to $S^1 \times B^{m_i-1}$. 
\item Moreover, in our situation, $\sigma_i^*$ is an embedded submanifold of $B^d$ and we want to perform the $2$-surgery \emph{ambiently} in $B^d$, i.e., we want to attach a hollow $2$-handle \emph{embedded} in $B^d$ and internally disjoint from $\sigma_i^*$ to get 
$\sigma_i^{**}$ embedded as well.
\item Furthermore, we want to avoid introducing new intersections, so we want the embedded hollow $2$-handle for $\sigma_i^*$ to be disjoint from $\sigma_1$ and $\sigma_j^*$, $j\neq i$, and the handles to be disjoint from each other. In order to do this, we will show that, for each $i=2,\ldots,r$, there is a $2$-dimensional disk $D_i$ in general position with boundary $\beta_i$ such that we can choose the hollow $2$-handle for $\sigma_i^*$ to lie in a small regular neighborhood of $D_i$ in $B^d$. Then, by general position, $D_i$ is disjoint from $\sigma_1$, and from $\sigma_j^*$ and $D_j$, $j\neq i$, so the same holds for any sufficiently small neighborhood of $D_i$, and hence for the hollow $2$-handles.
\end{enumerate}

We now make this more precise.

Let us first see that we can achieve the first two goals. We use the fact that $\sigma_i$ is unknotted in $B^d$, i.e., 
up to a PL self-homeomorphism of $B^d$, $\sigma_i$ is a standard coordinate $m_i$-ball embedded in $B^d$. Next, 
all possible pipings of $\sigma_i$ are ambient isotopic keeping $\sigma_i$ 
fixed. Thus, we may assume that $\sigma_i^*$ is a ``standard'' piped $m_i$-ball in $B^d$, see Figure~\ref{fig_standard_piping}.
In this ``standard'' situation, it is clear that we can find the desired $\beta$ and that the ambient  $2$-surgery can be performed
such that the hollow $2$-handle lies in a small neighborhood of a ``standard'' $2$-dimensional disk $D_i$ with $\partial D_i=\beta_i$.
\begin{figure}[h]
\begin{center}
\includegraphics[scale=1]{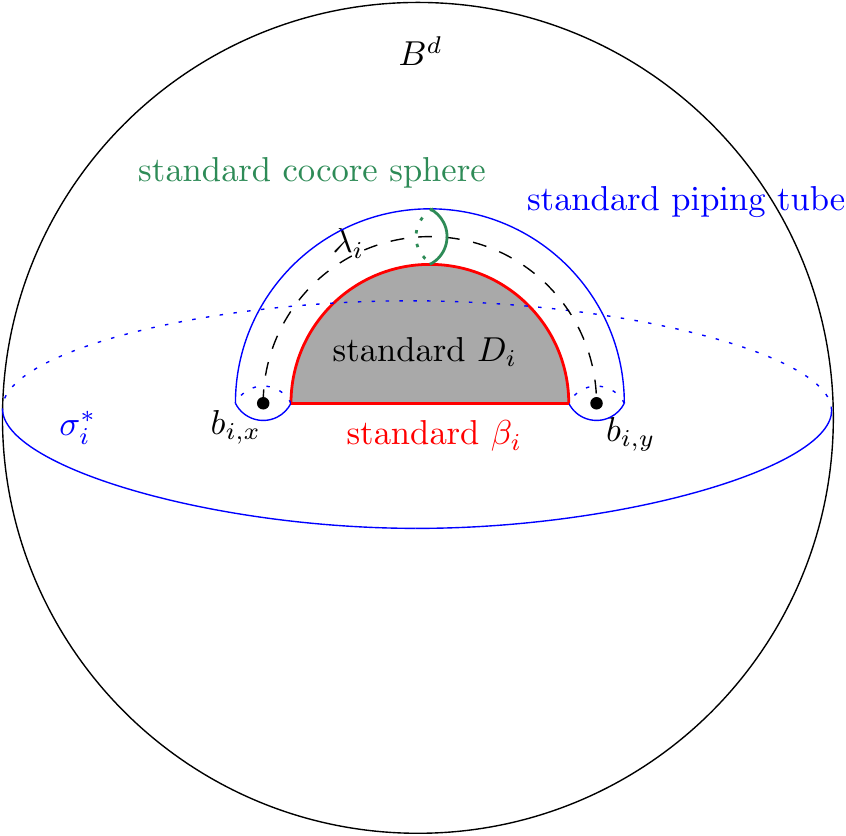}
\caption{A standard piped $\sigma_i^*$.}
\label{fig_standard_piping}
\end{center}
\end{figure}
More precisely, in this standard situation, we can find a small regular neighborhood $N$ of $D_i$ in $B^d$ and
a PL homeomorphism  
$$h\colon N \cong [-2,2]^2 \times [-1,1]^{d-2}$$ 
such that $h(D_i)=[-1,1]^2 \times 0^{d-2}$ and $h(N\cap \sigma_i^*)=\partial[-1,1]^2 \times [-1,1]^{m_i-1} \times 0^{d-m_i}$.

We do not control how the self-homeomorphism of $B^d$ and the ambient isotopy that we apply to get $\sigma_i^*$ into standard position affect $\sigma_1$ or the other $\sigma_j^*$ and $D_j$, $j\neq i$, and a priori they may intersect $N$. However, we know that each of them is of codimension at least $3$ in $B^d$ (and hence in $N$) and intersects $\sigma_i^*$ transversely in a submanifold of dimension at most $m_i-3$. Thus, up to a small ``parallel perturbation'' of $\beta_i$ in $\sigma_i^\ast$ corresponding to a parallel translation of $h(\beta_i)=\partial[-1,1]^2 \times 0^{d-2}$ by a random vector in $0^2 \times (-\delta,\delta)^{m_i-1} \times 0^{d-1-m_i}$ for some small $\delta>0$, we may assume that $\beta_i$ is disjoint from $\sigma_1\cap \sigma_i^*$ and from $\sigma_i^* \cap \sigma_j^*$, $1\leq j\leq r$, $j\neq i$. Similarly, up to a small perturbation of of $D_i$ inside $N$ and keeping $\beta_i$ fixed, we may assume that 
$D_i$ disjoint from $\sigma_1$ and $\sigma_j^*$ and $D_j$, $j\neq i$ (e.g., we can think of $D_i$ as a cone over $\beta_i$ and slightly perturb the apex of the cone, if necessary). Then we can take the hollow $2$-handle to be the preimage under $h$ of 
$$[-1,1]^2 \times \partial[-\varepsilon,\varepsilon]^{m_i-1} \times 0^{d-m_i},$$
which is disjoint from $\sigma_1$ as well as $\sigma_j^*$ and $D_j$, and hence from $\sigma_j^{**}$, $j\neq i$, for $\varepsilon>0$ sufficiently small.

Finally, $\sigma_i$ and $\sigma_i^{**}$ are $m_i$-dimensional PL balls properly embedded in $B^d$, $d-m_i\geq 3$, with $\partial \sigma_i =
\partial \sigma_i^{**}$, $2\leq i\leq r$. Thus, by the relative version of Zeeman's Unknotting Theorem (Corollary~\ref{thm_Zeeman_proper_embeddings}), for each $i$ there is an ambient isotopy $\t{H}^i$ of $B^d$ such that $\t{H}^i(\sigma_i)=\sigma_i^{**}$.
\end{proof}
\begin{remark}
Instead of using the above somewhat ad-hoc elementary argument to show that we can perform the ambient $2$-surgery, we could simply choose the disks $D_i$, $2\leq i\leq r$, in general position and then construct the required embedded hollow $2$-handles using the fact that each $D_i$ has a normal disk bundle in $B^d$ by \cite[Corollary~4.2]{Haefliger:Piecewise-linear-bundles-in-the-stable-range-1965}. However, we prefer to avoid using PL (micro)bundles in the present paper.
\end{remark}

\subsection{Proof of the Higher-Multiplicity Whitney Trick} 
\label{subsec_proof_whitney_extended}

As shown above, it suffices to prove Proposition~\ref{prop_reduced_whitney_tric}.

\begin{proof}[Proof of Proposition~\ref{prop_reduced_whitney_tric}]
As mentioned before, we proceed by induction on $r$, and the base case $r=2$ is the PL version of the Whitney Trick (see, e.g., Weber~\cite{Weber}). Thus, we may assume that $r\geq 3$ and that Proposition~\ref{prop_reduced_whitney_tric} holds for multiplicity
$r-1$. %Consequently, we may also assume that Theorem~\ref{thm_whitney_trick_extended} holds for multiplicity $r-1$.

As described in Section~\ref{subsec_piping}, we pipe and then unpipe each of $\sigma_2, \ldots , \sigma_r$ to form $\sigma_2^{**}$, $\ldots$, $\sigma_r^{**}$. Each $\sigma_i^{**}$ is a PL ball of dimension $m_i$, each pairwise intersection 
$\sigma_1\cap \sigma_i^{**}=\sigma_1\cap \sigma_i^*$ is a PL cylinder $S^{m-1}\cong [0,1]$ properly embedded into $\sigma_1$ and of codimension $m_1-\dim(\sigma_1\cap \sigma_i^{**})=d-m_i\geq 3$. Moreover, these cylinders intersect 
inside $\sigma_1$ in two $(r-1)$-fold intersection points of opposite sign,
$$\{x,y\}= (\sigma_1\cap \sigma_2^{**}) \cap \ldots \cap (\sigma_1\cap \sigma_r^{**}).$$

Since each $\sigma_1\cap \sigma_i^{**}$ is connected, by Lemma~\ref{lem:localization}, there is an $m_i$-dimensional ball
$B^{m_1}\subseteq \interior \sigma_1$ such that $B^{m_1}$ and $\sigma_1\cap \sigma_i^{**}\cap B^{m_1}=\sigma_i^{**}\cap B^{m_1}$, $2\leq i\leq r$, form a standard local situation around $x$ and $y$. By induction, there are ambient isotopies $\widehat{H}^i$ of $B^{m_1}$, fixed on $\partial  B^{m_1}$, $3\leq i\leq r$, which we can view as ambient isotopies of $\sigma_1$ fixed outside of $\interior B^{m_1}$, such that 
$$\sigma_1\cap \sigma_2^{**} \cap \widehat{H}_1^3(\sigma_1\cap \sigma_3^{**})\cap \ldots \cap \widehat{H}_1^r(\sigma_1\cap \sigma_r^{**})=\emptyset.$$
Since $\sigma_1$ is unknotted in $B^d$, i.e., $B^d\cong \sigma_1\ast S^{d-1-m_1}$, we can extend the $\widehat{H}^r$ to ambient isotopies of $B^d$, fixed on $\partial B^d$, which by some abuse of notation, we will denote by the same symbol. These ambient isotopies of $B^d$ satisfy $\widehat{H}^i_t(\sigma_1)=\sigma_1$ for all $t\in [0,1]$ and hence 
$$\sigma_1\cap \sigma_2^{**} \cap  \widehat{H}_1^3(\sigma_3^{**})\cap \ldots \cap \widehat{H}_1^r(\sigma_1\cap \sigma_r^{**})=\emptyset.$$
Let $\t{H}^i$ be the ambient isotopy of $B^d$ constructed in Lemma~\ref{lem:unpiping}, i.e., $\t{H}_1(\sigma_i)=\sigma_i^{**}$, $2\leq i\leq r$. Let $H^i$ be the composition of $\widehat{H}^i$ and $\t{H}^i$, $3\leq i\leq r$, and set $H^2:=\t{H}^2$. Then each $H^i$ is an ambient isotopy of $B^d$ fixed on $\partial B^d$, and  
$$\sigma_1\cap H^2_1(\sigma_2) \cap  H_1^3(\sigma_3)\cap \ldots \cap H_1^r(\sigma_r)=\emptyset,$$
as desired.%
%Using that the manifolds $\sigma_1\cap \sigma_i^{**}$ are connected, for each $i$ we connect $x$ and $y$ by a path $\lambda_i \in \sigma_i^{**}\cap \sigma_1$. Then we perform a similar localization trick as in Section~\ref{subsec_reduction}: for each $i$ we fill $\lambda_i \cup \lambda_{i+1}$ by $2$-disks in $\sigma_1$, and restrict to a regular neighborhood $B'$ in $B$ of the filling $2$-disk resulting from the union of those $2$-discs. Inside this neighborhood, $\sigma_1 \cap B'^d$ is an unknotted ball and all the intersection $\sigma_i^{**} \cap \sigma_1 \cap B'^d$ are $(n_1 + n_i-d)$-balls.
%
%Working inside of this neighborhood $B'^d$, we have by induction hypothesis a sequence of ambient isotopies  $H^2_t, \ldots, H^r_t$ of $\sigma_1$ (one for each $\sigma_2^{**}, \ldots , \sigma_r^{**}$), such that
%\[
%H^2_1(\sigma_1 \cap \sigma_2^{**}) \cap \cdots \cap H^r_1(\sigma_1 \cap \sigma_r^{**}) = \emptyset.
%\]
%(The ambient isotopies are the identity outside $B'^d  \cap \sigma_1$.)
%
%Since $B^d = \sigma_1 * S^{d-n_1 -1}$, we can extend these isotopies to $B^d$ to get isotopies $H^i_t: B^d \rightarrow B^d$ fixed on the boundary.
%
%
%We have
%\[
%\sigma_1 \cap H^2_1(\sigma_2^{**}) \cap \cdots \cap H^r_1(\sigma_r^{**}) = \emptyset.
%\]
%
%Then, by standard theory (see \cite[Ch IV, Cor 1, p. 16]{Zeeman}), we can find ambient isotopies of $B^d$ throwing $\sigma_i$ to (resp.) $H^i_1(\sigma_i^{**})$ fo $i=2, \ldots,r$.
\end{proof}

To complete this section, we present the missing 
\begin{proof}[Proof of Proposition~\ref{prop_r_intersection_manifold}]
We adapt the proof that a $k$-manifolds can be embeeded into $\R^{2k}$  \cite[Thm~5.5 p.~63]{Rourke:Introduction-to-piecewise-linear-topology-1982}.
Let $f : M \rightarrow \R^{rk}$ be a general position PL map.  Let $x_1, \ldots, x_r$ be $r$ distinct points of $M$ mapped by $f$ to the same point $y\in \R^{rk}$. 

In $M$, we draw a path $\lambda$ passing through $x_1, \ldots, x_r$ and avoiding all other preimages of $r$-fold points. The image $f(\lambda)$ is a $1$-dimensional polyhedron in $\R^{rk}$. Let us pick a generic point $p\in R^{rk}$, and consider the cone $C_{f( \lambda)}$ obtained by joining $p$ to $f(\lambda)$. By general position this cone is a collapsible $2$-ployhedron intersecting $f(M)$ only in $f (\lambda)$. 

Next, we take a regular neighbordhood $N$ of $C_{f( \lambda)}$ in $\R^{rk}$.  Since $C_{f(\lambda)}$ is collapsible, $N$ is a ball. Furthermore, if $N$ is constructed on a sufficiently fine triangulation of $\R^{rk}$, then $f^{-1} (N)$ is a regular neighbordhood of $\lambda$, and hence an $m$-dimensional PL ball $B$ in $M$ containing $\lambda$ in its interior. 

Note that $f : \boundary B  \rightarrow \boundary N$ is a PL map without $r$-intersection point. We redefine $f$ on the interior of $B$ by using the cone construction \cite[Ex~1.6.(3) p.~5]{Rourke:Introduction-to-piecewise-linear-topology-1982}: $B$ can be represented as a cone  $\boundary B * v $, where $v$ is  an interior point of $B$, and the same is true for $N$. Hence, we can extend $f$ linearly from $\boundary B$:  we choose an image for $v$ in the interior of $N$ and extend linearly. By construction, redefining $f$ on the interior of $B$ in this way removes the $r$-intersection point $y$ without introducting any new ones. 
\end{proof}

\section{The Deleted Product Criterion for Tverberg Points}
\label{sec:DeletedProductCriterionTverberg}
\label{sec_tverberg_points}

In this section we prove Theorem~\ref{thm:VK-Tverberg-complete}. The proof is subdivided into three steps as follows (the necessary definitions will be given in the corresponding subsections):
\begin{enumerate}
\item[\ref{sec_equivariant_obs_theory}]
If $K$ is an $m$-dimensional simplicial complex, $m=(r-1)k$ and $d=rk$, $k\geq 1$ (more generally, if $\dim \delprod{K}{r}=d(r-1)$, then there exists a primary equivariant obstruction 
$\vko{K}{r} \in Z^{d(r-1)}_{\sym_r}(\delprod{K}{r};\calZ)$, the \define{generalized Van Kampen obstruction}, such that there exists an equivariant map $F\colon \delprod{K}{r}\to_{\sym_r} S^{d(r-1)-1}$ if and only if $\vko{K}{r}=0$. Moreover, if $f\colon K\to \R^d$ is any PL map in general position, then the obstruction can be represented by an \define{intersection number cocycle} $\vko{K}{r}=[\cocyc_f]$, where
$$\cocyc_{f}(\sigma_1\times \ldots \times \sigma_r)=\pm f(\sigma_1)\scap \ldots \scap f(\sigma_r).$$
\item[\ref{sec_van_kampen_fingers_move}] Starting with an arbitrary map $f\colon K \rightarrow \R^d$ with $\vko{K}{r}=[\cocyc_f] = 0$, one can construct a new PL map $g\colon K\to \R^d$ 
by using an $r$-fold generalization of the classical \emph{Van Kampen finger moves}. From $\cocyc_g=0$, we conclude that, for each $r$-tuple of pairwise disjoint $m$-simplices of $K$,
the $r$-Tverberg points $y\in g(\sigma_1)\cap \ldots\cap g(\sigma_r)$ appear in pairs of opposite sign.
\item[\ref{sec_proof_thm_purely_comb}] Having obtained such a map $g\colon K\to \R^d$, and assuming now $k\geq 3$, we can apply the $r$-fold Whitney trick (Theorem~\ref{thm_whitney_trick_extended}) to remove $r$-Tverberg points in pairs of opposite sign, without introducing new $r$-fold points in the process. Thus, we obtain a map PL map 
$h \colon K \rightarrow \R^d$ without $r$-Tverberg point.
\end{enumerate}

\subsection{Equivariant Obstruction Theory and Intersection Number Cocycles} 
\label{sec_equivariant_obs_theory}

Here, we briefly review some basic elements of equivariant obstruction theory. For short and very accessible introductions, see 
\cite{Blagojevic:Using-equivariant-obstruction-theory-in-combinatorial-geometry-2007} or \cite[Sec.~4.1]{Zivaljevic:UserGuideEquivariantTopologyCombinatorics2-98}; for a comprehensive and detailed treatment of the theory, 
the standard source is tom Dieck's monograph \cite[Sec.~II.3]{Dieck:Transformation-Groups-1987}. 

For the present section, fix parameters $r\geq 2$ and $d\geq 1$, and set $n:=d(r-1)$. Let $Y:=(\R^d)^r \setminus \thindiag{r}{\R^d} \simeq_{\sym_r} S^{n-1}$ be the complement of the thin diagonal in $(\R^d)^r$, with the natural action of the symmetric group $\sym_r$ by permuting components.

We will need the fact that $Y$ is \define{$(n-2)$-connected} (i.e., every map $S^{\ell-1}\to Y$ is nullhomotopic, $\ell< n$) and that, by the classical theorem of Hopf, the set $[S^{n-1},Y]$ of homotopy classes of maps $f\colon S^{n-1}\to Y$ can be identified with the integers via the mapping degree,
\begin{equation}
\label{eq:Hopf}
[S^{n-1},Y]\cong \Z, \qquad [f]\mapsto \deg(f).
\end{equation}
More precisely, the definition of the degree involves  the choice of an orientation of $S^{n-1}$ and of a generator $\zeta$ of $H_{n-1}(Y;\Z)\cong \Z$, and in what follows we will always specify these choices.\footnote{Choosing an orientation of 
 $S^{n-1}$ is equivalent to choosing a generator $\iota$ of $H_{n-1}(S^{n-1};\Z)\cong \Z$, and given $\iota$ and $\zeta$, the degree  $\deg(f)$ is, by definition, the unique integer such that $f_\ast(\iota)=\deg(f)\zeta$, where $f_\ast$ is the induced map in homology.}
 
The action of $\sym_r$ on $Y$ induces a natural action on $[S^{n-1},Y]$ and hence, via the bijection \eqref{eq:Hopf}, on the integers $\Z$ (it can be checked that the action of a permutation $\pi$ is given by multiplication by $(\sign \pi)^d$); we will use the notation $\calZ$ to denote the integers with this $\sym_r$-action.
%%%%%%%%%%%%%%%%%%%%%%%%%%%%%%%%%%%%%%%%%%%%%%%%%%%%%%%%%%%%%%%%%
%%% ABSOLUTE VERSION OF THE PRIMARY OBSTRUCTION
%%%%%%%%%%%%%%%%%%%%%%%%%%%%%%%%%%%%%%%%%%%%%%%%%%%%%%%%%%%%%%%%%
\medskip

Let $X$ be an $n$-dimensional CW complex on which $\sym_r$ acts freely by cellular maps. The two cases that we will be interested in the present paper are $X=\delprod{K}{r}$, and $X=\join{\sym_r}{(n+1)}$.

%The deleted product $\delprod{K}{r}$ of a simplicial complex $K$ is a CW complex on which the symmetric group acts freely by cellular maps.

An $\ell$-dimensional cellular cochain $\varphi\in C^\ell(X;\calZ)$ 
%(i.e., an skew-symmetric labelling of the oriented $\ell$-cells in $X\setminus A$ by elements of $\calZ$) 
is \define{equivariant} if it commutes with the group action, i.e., $\varphi(\cell\cdot \pi)=\varphi(\cell)\cdot \pi$ for every oriented $\ell$-cell $\cell$ of $X$ and $\pi\in \sym_{r}$. The equivariant cochains form a subgroup $C^\ell_{\sym_r}(X;\calZ)$ of the usual (nonequivariant) cochains. Moreover, the coboundary operator sends equivariant cochains to equivariant cochains, so we get subgroups $B^\ell_{\sym_{r}}(X;\calZ)$ of \define{equivariant coboundaries} (coboundaries of equivariant $(\ell-1)$-cochains) and $Z^\ell_{\sym_{r}}(X;\calZ)$ of \define{equivariant cocycles} ($\ell$-cocycles that are equivariant), and the \define{equivariant cohomology groups} are defined by 
$$H^\ell_{\sym_{r}}(X;\calZ)=Z^\ell_{\sym_{r}}(X;\calZ)/B^\ell_{\sym_{r}}(X;\calZ).$$

The basic idea of (equivariant) obstruction theory is that we want to construct an (equivariant) map $F\colon X\to Y$  inductively over skeleta of $X$ of increasing dimension, and likewise for (equivariant) homotopies between such maps (which are maps $X\times [0,1]\to Y$). 
If $\cell$ is an $\ell$-cell of $X$ and if we inductively assume that $F$ is already defined on $\skel{\ell-1}{X}$, hence in particular on the boundary $\partial \cell \cong S^{\ell-1}$, then we can extend $F$ over $\cell$ if and only if $F|_{\partial \cell}$ is nullhomotopic.\footnote{Here, we tacitly use that $X$ is a \define{regular} CW complex, i.e., that all attaching maps are homeomorphisms, so that a closed $\ell$-cell $\cell$ of $X$ is a closed $\ell$-disk embedded in $X$; for more general CW complexes, the condition would be that $F\circ \alpha_\cell|_{S^{\ell-1}}$ needs to be nullhomotopic, where $\alpha_\cell\colon S^{\ell-1}\to X$ is the attaching map of the cell $\cell$.} If this is the case, then any choice of such an extension to $\cell$ 
yields a unique equivariant extension to all cells $\pi\cdot \cell$ in the orbit of $\cell$ (since the action of $\sym_r$ on $X$ is free).

Using the connectivity of $Y$, it is not hard to show \cite[Prop.~II.3.15]{Dieck:Transformation-Groups-1987} that there exists an equivariant map $G\colon \skel{n-1}{X}\to_{\sym_r}Y$, and that the restrictions of any two such maps to  $\skel{n-2}{X}$ are equivariantly homotopic.

In the next extension step to the $n$-skeleton of $X$ (which is the last since $\dim X=n$), however, we might get stuck, namely if there is an $n$-cell $\cell$ such that $\deg(G|_{\partial \cell}\colon \partial \cell\to Y)\neq 0$. If this is the case, we might try to modify the chosen $G$ on $\skel{n-1}{X}$ so as to make $G|_{\cell}$ nullhomotopic. Whether it is possible to achieve this for all $n$-cells $\cell$ simultaneously is governed by a single $n$-dimensional equivariant cohomology class% (this is a special case of the \define{primary equivariant obstruction}, 
; see \cite[Section~II.3, pp.~119--120]{Dieck:Transformation-Groups-1987} for a proof:

\begin{theorem}%[\textbf{Primary Equivariant Obstruction}]
\label{thm:primary-obs}
Suppose that $X$ is an $n$-dimensional CW complex with a free cellular action of $\sym_r$.
%Suppose that $K$ is a finite simplicial complex such that $\dim\delprod{K}{r}=n=d(r-1)$. 
Then there exists an equivariant cohomology class $\obs(X) \in H^n_{\sym_{r}}(X;\calZ)$, called 
the \define{primary equivariant obstruction}, such that the following properties are satisfied:
\begin{enumerate}[label=\textup{(\arabic*)}]
\item There exists an equivariant map $F\colon X \to_{\sym_r} Y=(\R^d)^r \setminus \thindiag{r}{\R^d}$ if and only if $\obs(X)=0$. 
\item Let $G \colon \skel{n-1}{X} \to_{\sym_r} Y$ be an arbitrary equivariant map, and let $\zeta_0$ be a fixed generator of $H_{n-1}(Y;\Z)\cong \calZ$. For every oriented $n$-cell $\cell$ of $X$, set
$$
\cocyc_{G}(\cell):=\deg(G|_{\partial \cell}\colon \partial \cell\to Y) \in \calZ,
$$
where the mapping degree is computed with respect to $\zeta_0$ and the orientation of $\partial \cell \cong S^{n-1}$ is induced by that of $\cell$.
This defines an equivariant \define{obstruction cocycle} 
$$\cocyc_G \in Z_{\sym_r}^n(X;\calZ)$$
which represents the primary obstruction, i.e., $\obs(X)=[\varphi_G]$.
\end{enumerate}
\end{theorem}

In the special case that $X=\delprod{K}{r}$ for a finite simplicial complex $K$, we call $\vko{K}{r}$ the \define{$r$-fold Van Kampen obstruction}

\begin{lemma} 
\label{lem:r-intersection-cycle}
 \begin{enumerate}[label=\textup{(\alph*)}]
 \item Suppose the equivariant map $G \colon \skel{n-1}{X}\to_{\sym_r} (\R^d)^r \setminus \thindiag{r}{\R^d}$ in 
Theorem~\textup{\ref{thm:primary-obs}~(2)} is the restriction of an equivariant PL map in general position%
\footnote{Note that for every $G\colon \skel{n-1}{X}\to_{\sym_r} (\R^d)^r \setminus \thindiag{r}{\R^d}$ there is such an extension, since $(\R^d)^r$ is contractible; conversely, for every PL map $G\colon X \to_{\sym_r} (\R^d)^r$ in general position, its restriction to $\skel{n-1}{X}$ avoids the thin diagonal.}
 (denoted by the same symbol, by abuse of notation)
 $$G\colon X \to_{\sym_r} (\R^d)^r.$$
 
Then the value of the obstruction cocycle $\cocyc_G$ on each oriented $n$-cell $\cell$
of $X$ is given by the (pairwise) intersection number\footnote{Calculated with respect to the orientations of $(\R^d)^r$ and of $\thindiag{r}{\R^d}$ induced by that of $\R^d$ as described in Section~\ref{sec:intersection_signs}.}
\begin{equation}
\label{eq:cocyc-intersection-diagonal}
\cocyc_{G}(\cell):= G(\cell)\scap \thindiag{r}{\R^d}.
\end{equation}
\item Furthermore, suppose that $X=\delprod{K}{r}$ for a simplicial complex $K$ and that $f\colon K\to \R^d$ is a PL map in general position. In this case, we can take 
$$G=f^r\colon \delprod{K}{r} \to_{\sym_r} (\R^d)^r$$
as in the proof of Lemma~\ref{lem:delprod-necessary}, and represent $\vko{K}{r}=[\cocyc_f]$ by the following \define{intersection number cocycle} (denoted by $\cocyc_f$ instead of $\cocyc_{f^r}$ for simplicity) given by
\begin{equation}
\label{eq:r-intersection-cocycle}
\begin{array}{rcl}
\cocyc_f(\sigma_1\times\ldots\times \sigma_r) &= & \big(f(\sigma_1)\times \ldots \times f(\sigma_r)\big) \scap \thindiag{r}{\R^d} \\[2ex] 
& = & \varepsilon_{d,m_1,...,m_r} f(\sigma_1) \scap \ldots \scap f(\sigma_r) \\[2ex]
& = & \varepsilon_{d,m_1,...,m_r} \displaystyle \sum_{y\in f(\sigma_1)\cap\ldots \cap f(\sigma_r)} \isign{y}{f(\sigma_1)}{f(\sigma_r)}
\end{array}
\end{equation}
where $\varepsilon_{d,m_1,...,m_r}$ is the sign introduced in Lemma~\ref{lem:prop-inters-prod}~(d), %Equation~\eqref{eq:r-sign-vs-product-diagonal}, 
and $m_i=\dim \sigma_i$, $1\leq i\leq r$.
\end{enumerate}
\end{lemma}
\begin{proof}
A generator $\zeta_0$ of $H_n(Y;\Z)$ can be represented geometrically as the homology class $\zeta_0=[\partial \tau_0]$ of the boundary of an oriented linear $n$-simplex $\tau_0$ in $(\R^d)^r$ that intersects $\thindiag{r}{\R^d}$ in its relative interior. For concreteness, we choose $\tau_0$ so that it intersects $\thindiag{r}{\R^d}$ positively. Then\footnote{To see this, note that  the boundaries of any two oriented linear $n$-simplices that intersect the diagonal positively correspond to the same generator of $H_{n-1}(Y, \Z)$, and if we reverse the orientation of such a simplex $\tau$, so that its intersection sign with $\thindiag{r}{\R^d}$ becomes negative, then we also reverse the sign of $[\partial \tau]$ as a generator of the homology group. Furthermore, if $\tau$ is disjoint from $\thindiag{r}{\R^d}$ then $[\partial \tau]=0$ in the homology group. The first part of the lemma now follows by choosing a sufficiently fine triangulation of the cell $\sigma%_1\times \ldots \times \sigma_r
$ on which $G$ is simplexwise linear: Then $G_\ast([\partial %(
\sigma%_1\times \ldots \times \sigma_r)
])=\sum_\tau G_\ast([\partial \tau])$, where $\tau$ ranges over all $n$-simplices in the triangulation, and $[\partial \tau]$ equals $+[\partial \tau_0]$, $-[\partial \tau_0]$, or zero depending on whether $h(\tau)$ intersects $\thindiag{r}{\R^d}$ positively, negatively, or not at all.}
$$\deg(G|_{\partial \sigma%_1\times \ldots \times \sigma_r
}\colon \partial \sigma%_1\times \ldots \times \sigma_r 
\to Y) = G(\sigma%_1\times \ldots \times \sigma_r
)\scap \thindiag{r}{\R^d},$$
which shows \eqref{eq:cocyc-intersection-diagonal}. Furthermore, \eqref{eq:r-intersection-cocycle} follows immediately, by 
Lemma~\ref{lem:prop-inters-prod}~(d).
\end{proof}

%%%%%%%%%%%%%%%%%%%%%%%%%%%%%%%%%%%%%%%%%%%%%%%%
\subsection{\texorpdfstring{$r$}{r}-Fold Van Kampen Finger Moves} 
\label{sec_van_kampen_fingers_move}
%%%%%%%%%

By Lemma~\ref{lem:r-intersection-cycle}, vanishing of the $r$-fold Van Kampen obstruction means that for every PL map $f\colon K\to \R^d$ 
in general position, the corresponding intersection number cocycle $\cocyc_f$ satisfies $\vko{K}{r}=[\cocyc_f]=0$ as a cohomology class, i.e.,
$\cocyc_f\in B^{d(r-1)}_{\sym_r}(\delprod{K}{r};\calZ)$ is an equivariant coboundary.
The goal of this section is to show that in this situation, we can find a map $g$ such that $\cocyc_g=0$ as a cocycle (see Lemma~\ref{lem_finger_moves_I} below).
To this end, we consider the following system of generators of the equivariant coboundaries.

\paragraph{Elementary coboundaries.} 
For any dimension $\ell$, we get a basis of the $\ell$-dimensional equivariant cochains $C^\ell_{\sym_{r}}(\delprod{K}{r};\calZ)$ as follows: Choose an $\ell$-dimensional oriented cell $\eta_1 \times \cdots \times \eta_r$ of $\delprod{K}{r}$ (\ie the product of pairwise disjoint simplices of $K$ with $\sum_{i=1}^r \dim \eta_i =\ell$). We define the cochain $\I_{(\eta_1 \times \cdots \times \eta_r)\cdot \sym_{r}}$ to take value $1$ on $\eta_1 \times \cdots \times \eta_r$ and then extend equivariantly over the $\sym_{r}$-orbit of the cell, i.e., $\I_{(\eta_1 \times \cdots \times \eta_r)\cdot \sym_{r}}$ takes value $(\sign \pi)^d$ on $\ell$-cells of the form $(\eta_1 \times \cdots \times \eta_r)\cdot \pi=\pm \eta_{\pi(1)}\times \ldots \times \eta_{\pi(r)}$, $\pi \in \sym_r$ (where the sign depends how the action of $\pi$ affects the orientation), and the cochain %$\I_{(\eta_1 \times \cdots \times \eta_r)\cdot \sym_{r}}$ 
evaluates to zero on all other cells.
%$$
%\I_{(\eta_1 \times \cdots \times \eta_r)\cdot \sym_{r}}(\nu_1\times \ldots \times \nu_r)=
%\left\{
%\begin{array}{cl}
%\pi\cdot 1=(\sign \pi)^d & \textrm{if } \nu_1\times \ldots \times \nu_r=\pi\cdot(\eta_1 \times \cdots \times \eta_r) \textrm{ for some } \pi \in \sym_r,\\[1ex]
%0 & \textrm{otherwise.}
%\end{array}
%\right.
%$$

Thus, the equivariant coboundaries $B^{\ell+1}_{\sym_{r}}(\delprod{K}{r};\calZ)$ are generated by \define{elementary equivariant coboundaries} of the form $\delta \I_{(\eta_1 \times \cdots \times \eta_r)\cdot \sym_{r}}$, where $\eta_1 \times \cdots \times \eta_r$ is an oriented $\ell$-cells of $\delprod{K}{r}$. In particular, we have the following:
\begin{lemma}
If $f\colon K\to \R^d$ is a PL map in general position then $\vko{K}{r}=[\cocyc_{f}]=0$ if and only if $\cocyc_{f}$ can be written as a finite sum of elementary coboundaries,
\begin{equation} 
\label{eq_phi_f_coboundary}
\cocyc_f = \sum \pm \delta \I_{(\eta_1 \times \cdots \times \eta_r)\cdot \sym_{r}},
\end{equation}
where the sum is over a finite multiset of $(d(r-1)-1)-dimensional$-cells of $\delprod{K}{r}$.
\end{lemma}

Suppose now that $\dim K=m=(r-1)k$, and $d=rk$. If $\eta_1 \times \cdots \times \eta_r$ is a cell of $\delprod{K}{r}$ of dimension $d(r-1)-1=rm-1$ then (up to a permutation $\pi \in \sym_r$ of the $\eta_i$), we may assume that
\begin{equation}
\label{eq:codim-1-cell}
\eta_1 \times \cdots \times \eta_r = \mu_1\times \sigma_2 \times \ldots \times \sigma_r,
\end{equation}
where $\mu_1$ is an $(m-1)$-simplex of $K$ and $\sigma_i$ is an $m$-simplex of $K$, $2\leq i\leq r$. Consequently,
\begin{equation}
\label{eq_coboundary_mu_sigma_i}
\delta  \I_{(\mu_1 \times \sigma_2 \times \cdots \times \sigma_r)\cdot \sym_{r}} = \sum_{\sigma_1} \I_{(\sigma_1 \times \sigma_2\times  \cdots \times  \sigma_r)\cdot \sym_{r}},
\end{equation}
where the sum is over all the $m$-simplices $\sigma_1$ of $K$ that contain $\mu_1$ in their boundary and that are disjoint from $\sigma_i$, $2\leq r\leq r$
(where the orientation of $\sigma_1$ is chosen such that $\mu_1$ appears positively oriented in $\partial\sigma_1$).

On the one hand, this immediately yields a proof that the condition $\vko{K}{r}=0$ is efficiently testable (see the end of this subsection).
More importantly, by the following lemma, addition of single elementary coboundary to $\cocyc_f$ can be emulated geometrically by a 
simple modification of the map $f$ (the case $r=2$ corresponds to the classical Van Kampen finger moves). 

%%%%%%%%%%%%%%%%%%%%%%%%%%%%%
\begin{lemma}[\textbf{$r$-Fold Finger Move}]
\label{lem:single-finger-move}
If $f\colon K\to \R^d$ is a PL map in general position and if $\delta \I_{(\mu_1 \times \sigma_2 \times \cdots \times  \sigma_r)\cdot\sym_{r}}$ is an elementary equivariant $mr$-dimensional coboundary then for any choice of a sign $\varepsilon\in \{-1,+1\}$, there exists a PL map $g\colon K\to \R^d$ such that $\cocyc_{g} = \cocyc_f +\varepsilon \cdot \delta \I_{(\mu_1 \times \sigma_2 \times \cdots \times \sigma_r)\cdot\sym_{r}}$.
\end{lemma}
%%%%%%%%%%%%%%%%%%%%%%%%%%%%%%

\begin{corollary} 
\label{lem_finger_moves_I}
Suppose $K$ is a finite simplicial complex, $\dim K=m\leq d-1$, $\dim\delprod{K}{r} =d(r-1)$ and 
$$\vko{K}{r}=0.$$
Then there exists a PL map $g \colon K\to \R^d$ such that 
$$\varphi_g=0$$ 
as a cocycle, i.e., $ g(\sigma_1) \scap \dots \scap g(\sigma_r)=0$ for every $d(r-1)$-cell $\sigma_1\times \dots \times \sigma_r$ 
of $\delprod{K}{r}$.
\end{corollary}

\begin{remark}
\label{rem:equivariance-sym_r-finger-moves}
Lemma~\ref{lem:single-finger-move} and Corollary~\ref{lem_finger_moves_I} are where we need equivariance with respect to the full symmetric group $\sym_r$ and not just with respect to some subgroup $H\leq \sym_r$. If $H$ is some proper subgroup then we get a larger set of $H$-equivariant coboundaries $\delta \I_{(\mu_1 \times \sigma_2 \times \cdots \times \cdots \sigma_r)\cdot H}$ (hence the condition that $\varphi_f$ is a sum of $H$-equivariant coboundaries is more easily fulfilled), but we do not have a analogous geometric modification of a given map $f$ that would allow us to emulate the addition of $\delta \I_{(\mu_1 \times \sigma_2 \times \cdots \times \cdots \sigma_r)\cdot H}$ to $\varphi_f$.
\end{remark}

\begin{proof}[Proof of Corollary~\ref{lem_finger_moves_I}]
Let $f\colon K\to \R^d$ be an arbitrary PL map in general position. Then $\vko{K}{r}=[\cocyc_{f}]=0$, so, by \eqref{eq_phi_f_coboundary}, we get the desired map $g$ by a finite number of applications of 
Lemma~\ref{lem:single-finger-move}.\end{proof}

\begin{proof}[Proof of Lemma~\ref{lem:single-finger-move}] 
Fix $f\colon K\to \R^d$ and an oriented $(mr-1)$-cell $\mu_1 \times \sigma_2\times  \cdots \times  \sigma_r$ of $\delprod{K}{r}$. 

By \eqref{eq_coboundary_mu_sigma_i} and \eqref{eq:r-intersection-cocycle}, we need to construct $g\colon K\to \R^d$ that satisfies two conditions: First,
$$
g(\sigma_1)\scap g(\sigma_2) \scap \ldots \scap g(\sigma_r)=f(\sigma_1)\scap f(\sigma_2) \scap \ldots \scap f(\sigma_r) + \varepsilon
$$
whenever  $\sigma_1$ is an $m$-simplex of $K$ that is disjoint from  $\sigma_i$, $2\leq i\leq r$, and that contains $\mu_1$ in its boundary with positive orientation.
Second, $g(\tau_1)\scap \ldots \scap g(\tau_r)=f(\tau_1) \scap \ldots \scap f(\tau_r)$ for every $mr$-cell of $\delprod{K}{r}$ that is not incident to
any $(mr-1)$-cell of the form $\pi\cdot(\mu_1 \times \sigma_2\times  \cdots \times  \sigma_r)$, $\pi \in \sym_r$.

Consider a point $x$ in the relative interior of $f(\mu_1)$. Since $f$ is PL, in a sufficiently small neighborhood of $x$, $f$ looks like a simplexwise linear map, 
see Figure~\ref{fig_finger_moves}.

Choose $(r-1)$ PL spheres $S_2, \ldots , S_r$ in $\R^d$ in general position, each of dimension $m$, such that
\[
S_2 \cap \cdots \cap S_r = S^{d-m},
\]
is a PL sphere of dimension $(d-m)$ that bounds a flat PL ball $B^{d-m+1}$ (a convex polytope, say) such that $f(\mu_1)\cap B^{d-m+1}=\{x\}$ (i.e., $S^{d-m}$ is locally linked with $f(\mu_1)$).

\begin{figure}[h!]
\begin{center}
\includegraphics[scale=1]{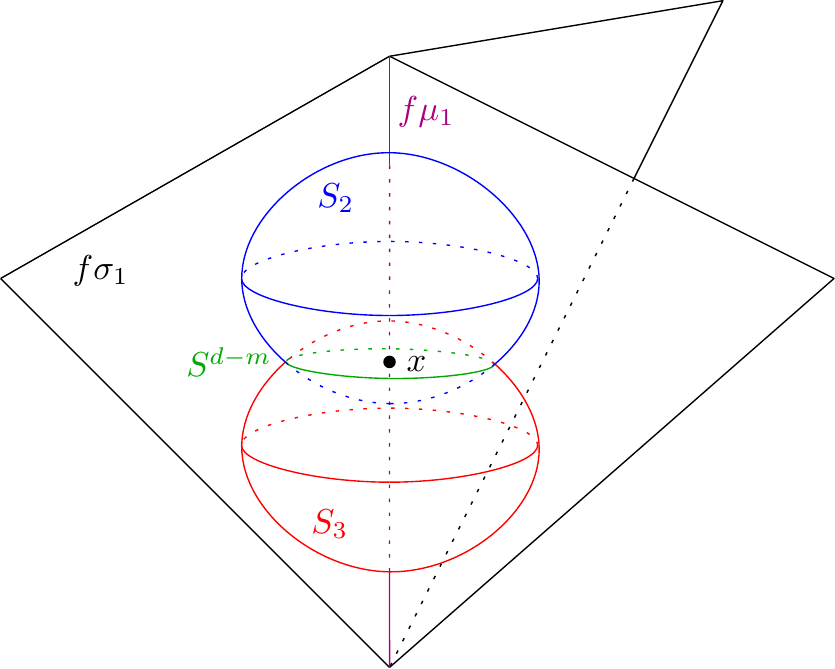}
\caption{For $r= 3$, $S_2$ and $S_3$ intersects in a sphere $S^{d-m}$.}
\label{fig_finger_moves}
\end{center}
\end{figure}

By choosing the spheres $S_i$ sufficiently small, we can make sure that $S^{d-m}$ is disjoint from the image $f(\tau)$ of any simplex $\tau$ of $K$ that does not contain $\mu_1$,
and that $S^{d-m}$ intersects the image $f(\sigma_1)$ of each $m$-simplex $\sigma_1$ containing $\mu_1$ in a single point.

Choose the orientation of $B^{d-m+1}$ such that $f(\mu_1)\scap B^{d-m+1}=(-1)^m \varepsilon $, and let $S^{d-m}=\partial B^{d-m+1}$ have the induced orientation.
Then, by Lemma~\ref{lem:intersection-bd}, we have
$$
f(\sigma_1) \scap S^{d-m} = (-1)^m \partial f(\sigma_1)\scap B^{d-m+1}=\varepsilon,
$$
if $\sigma_1$ contains $\mu_1$ on its boundary with positive orientation, and $f(\tau)\scap S^{d-m}=0$ if $\tau$ does not contain $\mu_1$.

By Lemma~\ref{lem:prop-inters-prod}~(a), we can choose orientations for the spheres $S_2, \ldots , S_r$ such that the induced orientation on their intersection $S^{d-m}$ agrees with
the orientation of $S^{d-m}$ as the boundary of $B^{d-m+1}$. Thus, we have
$$
f(\sigma_1)\scap S_2\scap \ldots \scap S_r=\varepsilon.
$$
and $f(\tau)\scap S_2\scap \ldots \scap S_r=0$ whenever $\tau$ does not contain $\mu_1$.

To conclude, we connect each $m$-sphere $S_i$, $2\leq i\leq r$ to $f(\sigma_i)$ by a pipe  that avoids $f(K)$ except at its boundary and that preserves orientations at both ends (see Section~\ref{subsec_piping}).
Piping with $S^m$ does not change the topology, so we can view the piped $f(\sigma_i)$ as the image $g(\sigma_i)$ of $\sigma_i$ under a PL map. We get the desired map $g\colon K\to \R^d$ by
setting $g=f$ outside of the interiors $\interior \sigma_i$, $2\leq i\leq r$.
\end{proof}

\begin{proof}[Proof of Corollary~\ref{cor:VKO-computable}]
Let $R$ be the number of $\sym_r$-orbits $(\sigma_1\times \ldots \times \sigma_r)\cdot \sym_r$ of $d(r-1)$-cells of $\delprod{K}{r}$, and let $S$ be the number of $\sym_r$-orbits $(\mu_1\times \ldots \times \sigma_r)\cdot \sym_r$ of cells of $\delprod{K}{r}$ of dimension $d(r-1)-1$. Then we can identify $C^{d(r-1)}_{\sym_r}(\delprod{K}{r};\calZ)$ 
with the free abelian group $\Z^R$, and we can identify the equivariant coboundaries $B^{\ell+1}_{\sym_{r}}(\delprod{K}{r};\calZ)$ with the subgroup $B \leq Z^R$ generated by (vectors corresponding 
to) the elementary coboundaries $\I_{(\mu_1\times \ldots \times \sigma_r)\cdot \sym_r}$. Let $A\in \{-1,0,1\}^{R\times S}$ be the matrix whose columns are these generators of $B$.

Choose an arbitrary simplexwise linear map $f\colon K\to \R^d$ in general position. Then the intersection number cocycle $\cocyc_f$ takes only values in $\{-1,0,+1\}$, so we can view $\cocyc_f$ as a vector $v \in \{-1,0,1\}^R$. Then $[\cocyc_f]=0$, or equivalently $\cocyc_f \in B^{d(r-1)}_{\sym_{r}}(\delprod{K}{r};\calZ)$ if and only if the inhomogeneous system of integer linear equations
$$Ax=v$$
has a solution $x\in \Z^S$. For fixed $r$, this system has size polynomial in the size (number of simplices) of $K$, and solvability of $Ax=v$ can be tested by bringing the matrix $A$ into Smith normal form.
For this, several polynomial-time algorithms are available in the literature, both deterministic (see e.g., \cite{Storjohann:NearOptimalAlgorithmsSmithNormalForm-1996}) and randomized ones (see, e.g., 
\cite{Giesbrecht:FastComputationSmithFormSparseIntegerMatrix-2001,DumasSaundersVillard:EfficientSparseIntegerMatrixSmithNormalForm-2001}).
\end{proof}

%%%%%%%%%%%%%%%%%%%%%%%%%%%%%%%%%%%%%%%%%%
\subsection{Proof of Sufficiency of the Deleted Product Criterion}
\label{sec_proof_thm_purely_comb}
%%%%%%%%%%%%%%%%%%%%%%%%%%%%%%%%%%%%%%%%%%

\begin{proof}[Proof of Theorem~\ref{thm:VK-Tverberg-complete}]
Suppose that there is a $\sym_r$-equivariant map $\delprod Kr \rightarrow_{\sym_r}  S^{mr-1}$, or equivalently, that $\vko{K}{r}=0$. By Corollary~\ref{lem_finger_moves_I},
there exists a PL map $f \colon K \to \R^d$ in general position such that $\cocyc_f=0$, or equivalently
$$0=f(\sigma_1)\scap \ldots \scap f(\sigma_r)$$
whenever $\sigma_1,\ldots,\sigma_r$ are pairwise-disjoint $m$-simplices of $K$. Thus, the $r$-Tverberg points  
$y\in f(\sigma_1)\cap \ldots \cap f(\sigma_r)$ occur in pairs of opposite signs (we match the pairs up arbitrarily). By the generalized Whitney Trick (Theorem~\ref{thm_whitney_trick_extended}), 
we can remove these pairs of $r$-intersection points, one pair at a time, by local ambient isotopies. Since we can choose the isotopies for each pair to have support in a PL ball that avoids any given obstacle $L$ of codimension at least $3$, we do not introduce any new $r$-intersection points in the process.
\end{proof}

\section{Counterexamples to the Topological Tverberg Conjecture in Dimension \texorpdfstring{$\boldsymbol{3r}$}{3r}}
\label{sec:counterexamples}

In this section, we prove Theorem~\ref{thm:counterexamples}, i.e., we show that for $r$ not a prime power and $N=(3r+1)(r-1)$ there exists a PL map 
$f\colon \simplex^N\to \R^{3r}$ without $r$-Tverberg points.

The idea of the proof is to consider a restricted family of maps, called \emph{prismatic maps}, whose special structure guarantees that in order to study the $r$-Tverberg 
points of a prismatic map, it suffices to consider the restriction of the map to a certain ``colorful'' $m$-dimensional subcomplex $C$ of $\simplex^N$, where $m=3(r-1)$.

Since the codimension $3r-m=3$ is large enough, the $r$-fold Whitney trick is applicable to maps $C\to \R^{3r}$. 

The main technical part of the proof consists in showing that there are variants of the $r$-fold finger moves and of the $r$-fold Van Kampen obstruction for the restricted, prismatic setting.

\subsection{Prismatic Maps}
\label{subsec:prismatic}

Fix parameters $r\geq 2$ and $k\geq 1$ and set
\begin{equation}
\label{eq:def-d-m}
%d=rk, \qquad 
N=(rk+1)(r-1),\qquad \textrm{and} \qquad m=(r-1)k. 
\end{equation}
We note that $N+1=r(m+1)$, and we fix a partition of the vertices of $\simplex^N$ into $m+1$ subsets
\begin{equation}
\label{eq:labeled-color-class}
C_j=\{v_{1,j},\dots, v_{r,j}\}, \qquad 0\leq j \leq m,
\end{equation}
consisting of $r$ vertices each; we choose and fix labeling of the vertices in each $C_j$ as indicated. We think of the vertex subsets $C_0,\ldots,C_m$ as \define{color classes}, and we call a simplex $\tau$ of $\simplex^N$ \define{colorful} if it contains at most one vertex from each color class $C_j$, $0\leq j\leq m$. The colorful simplices form a subcomplex
\begin{equation}
C=C_0\ast \dots\ast C_m \subset \simplex^N.
\end{equation}
Let us fix a labeling $u_0,\ldots,u_m$ of the vertices of $\simplex^m$. This yields a \define{projection map}
\begin{equation}
\label{eq:projection}
p\colon \sigma^N\to \sigma^m 
\end{equation}
by setting $p(v_{i,j})=u_j$ for $1\leq i\leq r$ and $0\leq j\leq m$ and extending linearly. We note that the colorful simplices 
are precisely those simplices $\tau$ of $\sigma^N$ such that $p|_\tau$ is injective.

We consider a particular kind of maps whose image is contained in the ``prism'' $\sigma^m \times \sigma^k \subset \R^d$, and which we call \emph{prismatic}; to motivate the general definition, we first consider the special case of affine maps; see Figure~\ref{fig_prismatic_map} for an illustration in the case $k=1,r=3$.

\begin{example}%[\textbf{Affine Prismatic Maps}] 
\label{ex_prismatic_map} 
For the vertices $v_{i,j}$ of $\simplex^N$, we choose \emph{generic} image points\footnote{The notion of genericity used here is a bit different from the notion of general position as discussed in Section~\ref{sec:GP} and will be discussed in more detail in the proof of Lemma~\ref{lem:defining-properties-prismatic}~\ref{prismatic-GP} below.}
\begin{equation}
\label{eq:vertex-images}
f(v_{i,j}) \in \{u_{j}\} \times \interior \simplex^k ,\qquad 1\leq i\leq r,\quad 0\leq j\leq m,
\end{equation}
and then extend linearly on each face of $\sigma^N$ to obtain an affine map (called \define{affine prismatic map})
$$f\colon \simplex^N \to \simplex^m \times \interior \simplex^k \subset \R^{rk}.$$
\begin{figure}[!ht]
\begin{center}
\includegraphics[scale=1]{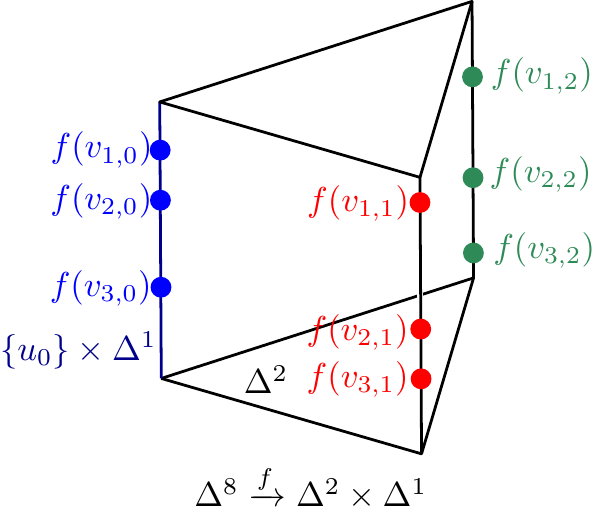}
\caption{For $k=1$ and $r=3$ (hence $m=2$), an affine prismatic map $f\colon \simplex^8\to \simplex^2 \times \interior \simplex^1\subset \R^3$ (with images of vertices in $C_0$, $C_1$, and $C_2$ colored blue, red, and green, respectively). The map is extended linearly on each face of $\simplex^8$. 
}
\label{fig_prismatic_map}
\end{center}
\end{figure}
\end{example}

The following lemma summarizes the basic properties of affine prismatic maps that we will use to define prismatic maps in general:

\begin{lemma}\label{lem:defining-properties-prismatic}
Let $f\colon \simplex^N \to \simplex^m \times \simplex^k \subset \R^{rk}$ be an affine prismatic map as defined in Example~\ref{ex_prismatic_map}. 
\begin{enumerate}[label=\textup{(\alph*)}]
\item \label{height-function}
There exists a map $h\colon \simplex^N\to \simplex^k$ such that
\begin{enumerate}[label=\textup{(REG)}]
\item \label{REG} $\hfill f(x)=(p(x),h(x)) \hfill$
\end{enumerate}
for $x\in \simplex^N$. We view $h(x)$ as the ``\emph{height}'' of $f(x)$ in the prism $\simplex^m\times \simplex^k$ with ``\emph{base}'' $\simplex^m$ and ``\emph{vertical component}'' $\simplex^k$.
\item \label{consequence-PR1-PR2}
As an immediate consequence of \ref{height-function}, $f$ has the following properties:
\begin{enumerate}[label=\textup{(PR\arabic*)}]
\item \label{PR1} For every simplex $\tau$ of $\simplex^N$, 
$$f(\interior \tau) \subseteq p(\interior \tau) \times \interior \simplex^k,$$
where $p$ is the projection map \eqref{eq:projection}, and $\interior \tau$ denotes the relative interior of $\tau$.
\item \label{PR2} If $\tau$ is colorful (i.e., if $p|_\tau $ is injective) then $f |_\tau$ is also injective.
\end{enumerate}
\item \label{prismatic-GP} Furthermore, apart from non-generic behavior forced by the property \ref{PR1},\footnote{For instance, the affine map in Figure~\ref{fig_prismatic_map} is not, strictly speaking, in general position
as a map into $\R^3$, since the three vertices in each color class have collinear images.} the restriction of the map $f$ to colorful simplices is in general position, in the following sense: 
\begin{enumerate}[label=\textup{(PR\arabic*)}]
\setcounter{enumii}{2}
\item \label{PR3}
Let $\omega$ be a $q$-dimensional face of $\simplex^m$, $0\leq q\leq m$. Then the restriction
$$f|_{p^{-1}(\interior{\omega}) \cap C} \colon p^{-1}(\interior{\omega}) \cap C \to \interior{\omega}\times \interior \simplex^k \cong \R^{q+k}$$
is in general position. In particular, if $\tau_1,\ldots,\tau_s$, $1\leq s\leq r$, are pairwise disjoint colorful simplices in $C\subset \sigma^N$ with $p(\tau_i)=\omega$ then 
\begin{equation}
\label{eq:intersections-prismatic-colorful}
\dim\big(f(\interior \tau_1)\cap\dots\cap f(\interior \tau_s)\big) =\max\{-1, \underbrace{sq-(s-1)(q+k)}_{=q-(s-1)k}\}.
\end{equation}
\end{enumerate}
\end{enumerate}
\end{lemma}

\begin{proof}
Part~\ref{height-function} (and therefore also \ref{consequence-PR1-PR2}) follows immediately from the definition of an affine prismatic map. 
The proof of \ref{prismatic-GP} is by induction on the dimension $q$. For $q=0$, the requirement is simply that we choose the image points $f(v_{i,j}$ to be pairwise distinct. More generally, given $q$-simplices $\tau_i$, $1\leq i\leq s$ as in \ref{prismatic-GP}, we observe that for each $i$ and each vertex $u_j$ of $\omega$, the affine subspaces $A_i:=\aff(f(\tau_i)$ and $\{u_j\} \times \R^k$ of $\aff(\omega)\times \R^k \subset \R^m \times \R^k =\R^d$ intersect transversely, at an angle bounded away from zero. Moreover, it is clear that we could 
achieve general position if we could perturb each image $f(v_{i,j})$ inside a small $(q+k)$-dimensional open set $U_{i,j}$ in $\aff(\omega) \times \R^k$ containing $f(v_{i,j})$. Since we want to keep the map $f$ prismatic, we are only allowed to perturb each $f(v_{i,j})$ inside a small $k$-dimensional open set $O_{i,j}$ in $\{u_j\}\times \R^k$. However, in order to analyze the intersections of the $f(\tau_i)$, we can imagine that we first perform this perturbation within $O_{i,j}$ and then further perturb each $f(v_{i,j})$ inside a small $q$-dimensional open set $Q_{i,j}$ inside $A_i$. Together these two perturbations would amount to perturbing $f(v_{i,j})$ in a $(q+k)$-dimensional open set, as desired. However, since the second perturbation does not affect $A_i$, the first one alone is sufficient to bring the subspaces $A_i$ into general position.
\end{proof}

\begin{definition}[\textbf{Prismatic Maps}]
\label{def:prismatic} Let $K$ denote either $\simplex^N$ or the colorful subcomplex $C$.

A PL map $f \colon K \rightarrow \simplex^{m} \times \interior \simplex^k$ is \define{prismatic} if it satisfies 
Conditions~\ref{PR1} (for all simplices $\tau$ in $K$), \ref{PR2}, and \ref{PR3}. 

A prismatic map is called \define{regular} if, in addition, it is of the special form \ref{REG}.
\end{definition}

Thus, a non-regular prismatic map does not need to respect the projection onto the base $\simplex^m$ (see Figure~\ref{fig_prismatic_non-regular} for an example), and this additional flexibility will be convenient for some techncial arguments in what follows.

\begin{figure}[!ht]
\begin{center}
\includegraphics[scale=1]{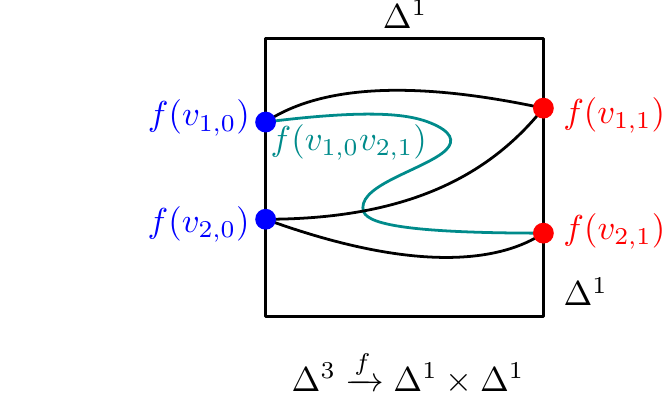}
\caption{For $k=1$, $r=2$, a prismatic map $C \to \simplex^1 \times \simplex^1$ that is non-regular; regularity is violated for the image of the edge $v_{1,0}v_{2,1}$.}
\label{fig_prismatic_non-regular}
\end{center}
\end{figure}

The following lemmas capture two key properties of prismatic maps.

\begin{lemma}
\label{lem_prismatic_map} 
Let $f\colon \sigma^N\to \sigma^m\times \interior \sigma^k \subset \R^{rk}$ be a prismatic map.
If $y\in f(\tau_1)\cap \dots \cap f(\tau_r)$ is an $r$-Tverberg point of $f$ then each simplex $\tau_i$ is colorful and of dimension $m$.
\end{lemma}

\begin{proof}
Let $\omega$ be the unique face of $\simplex^m$ such that $y\in \interior\omega\times \interior \simplex^k$, and let $q=\dim\omega$. Without loss of generality (up to relabeling), we may assume that
the vertex set of $\omega$ is $\{u_0,\ldots,u_q\}$. 

By \ref{PR1}, all simplices $\tau_1,\dots,\tau_r$ must be contained in $p^{-1}(\omega)$, so their vertices are contained in $C_0\cup \dots \cup C_q$, which is a set of size $(q+1)r$. Moreover, every simplex $\tau_i$ must contain at least one vertex from each of $C_j$, $0\leq j\leq q$, otherwise (by \ref{PR1} again), the image $f(\tau_i)$ and hence $y$ would be contained in $\boundary \omega \times \interior \simplex^k$, contradicting the choice  of $\omega$. By straightforward counting, it follows that every $\tau_i$ contains exactly one vertex from each $C_j$, $0\leq j\leq q$, i.e., every $\tau_i$ is colorful. 

Therefore, by Condition~\ref{PR3}, we have $q=m$, since for $q<m$, \eqref{eq:intersections-prismatic-colorful} and induction on $q$ would imply that $f(\tau_1)\cap \dots \cap f(\tau_r)=\emptyset$.
\end{proof}

\begin{lemma}
\label{lem:extend-prismatic}
Every prismatic map $g\colon C\to \simplex^m\times \interior{\simplex}^k$ can be extended to a prismatic map $f\colon \simplex^N \to \simplex^m\times \interior{\simplex}^k$.
\end{lemma}

\begin{proof}
We can construct the extension by induction on the dimension of the faces $\tau$ of $\simplex^N\setminus C$: Suppose $f$ is already been defined on $\boundary \tau$.
Let $\omega=p(\tau)$. We can extend $f$ to $\interior \tau$ by coning, using that $\omega\times \simplex^k$ is convex. More precisely, fix a point $b\in \interior \tau$,
choose an arbitrary image $f(b)\in \interior \omega \times \interior \simplex^k$ and extend $f$ linearly.
\end{proof}

Using these two lemmas, the proof of Theorem~\ref{thm:counterexamples} reduces to the following:

\begin{proposition}
\label{prop:prismatic_map_C_no_Tverberg}
Suppose $r\geq 6$ is not a prime power and $k\geq 3$.
Then there exists a prismatic map $g\colon C\to \simplex^m \times \interior \simplex^k$ without $r$-Tverberg points.
\end{proposition}

\begin{proof}[Proof of  Theorem~\ref{thm:counterexamples} using Proposition~\ref{prop:prismatic_map_C_no_Tverberg}]
Let $r\geq 6$ is not a prime power, $k=3$, and let $g$ be the prismatic map whose existence is guaranteed by the proposition. 

By Lemma~\ref{lem:extend-prismatic}, we can extend $g$ to a prismatic map $f\colon \simplex^N\to \simplex^m \times \interior \simplex^3$, and by Lemma~\ref{lem_prismatic_map}, the map $f$ has no $r$-Tverberg points since  $g=f|_C$ does not have any, which proves the theorem.
\end{proof}

\begin{proof}[Proof of Corollary~\ref{cor:type-m}]
The corollary follows directly from Lemma~\ref{lem_prismatic_map} and the affine prismatic maps constructed in Example~\ref{ex_prismatic_map}.
\end{proof}

%%%%%%%%%%%%%%%%%%%%%%%%%%%%%%
%%%%%%%%%%%%%%%%%%%%%%%%%%%%%%
\subsection{A Deleted Product Criterion For Prismatic Maps}
\label{sec:primatic-obstruction}
%%%%%%%%%%%%%%%%%%%%%%%%%%%%%%
%%%%%%%%%%%%%%%%%%%%%%%%%%%%%%

Thus, it remains to prove Proposition~\ref{prop:prismatic_map_C_no_Tverberg}. 
For this purpose, we will need analogues, for the restricted class of prismatic maps, of the Deleted Product Criterion, of the $r$-fold Van Kampen obstruction, and of $r$-fold finger moves. We begin by defining a suitable configuration space. 

\paragraph{The prismatic configuration space \boldmath{$X\cong_{\sym_r} (\sym_r)^{\ast (m+1)}$}.} 

By Lemma~\ref{lem_prismatic_map}, the preimages of $r$-Tverberg points of a prismatic map 
are supported on $r$ pairwise disjoint colorful $m$-simplices $\tau_1,\dots,\tau_r$ in 
$C \subset \simplex^N$. Using the fixed labeling $C_j=\{v_{1,j},\dots,v_{r,j}\}$ of the $r$ vertices in each color class, we can encode such an $r$-tuple of simplices using an $(m+1)$-tuple of permutations $\pi_j \in \sym_r$.
Slightly more generally, we have the following:

\begin{observation}
\label{obs:faces-permutations}
Suppose that $J=\{j_0,\ldots, j_q\}$ is a $(q+1)$-element subset of $\{0,\dots,m\}$, $0\leq q\leq m$, and that 
$$
(\tau_1,\ldots,\tau_r)
$$
is an $r$-tuple of pairwise disjoint $q$-simplices in $C_{j_0}\ast \dots \ast C_{j_q}$. Such an $r$-tuple of simplices corresponds bijectively to a $(q+1)$-tuple
\begin{equation}
\label{eq:q-tuple-perm}
\boldpi=(\pi_{j_0},\ldots, \pi_{j_q})
\end{equation}
of permutations $\pi_j \in \sym_r$ given by
\begin{equation}
\label{eq:permutations-colorful_simplices}
\tau_i \cap C_j = v_{\pi_j(i),j}
%\tau_i \cap C_j = v_{\pi^{-1}_j(i),j}
\end{equation}
for $1\leq i\leq r$ and $j\in J$.
\end{observation}

\begin{observation}
\label{obs:faces_join}
Consider the $(m+1)$-fold join
\begin{equation*}
\label{eq:conf-space_prismatic}
(\sym_r)^{\ast (m+1)}
\end{equation*} 
(where we view the symmetric group $\sym_r$ as a zero-dimensional complex). 
Every point in $(\sym_r)^{\ast (m+1)}$ can be written as a formal convex combination
\begin{equation}
\label{eq:point-in-join}
\lambda_0 \pi_0 + \dots + \lambda_m \pi_m,
\end{equation}
with $\pi_j\in \sym_r$ and $\lambda_j \in [0,1]$, $\sum_{j=1}^m \lambda_j =1$.

For $0\leq q\leq m$, a $q$-dimensional face of $(\sym_r)^{\ast (m+1)}$ is uniquely described by a 
pair
\begin{equation}
\label{eq:face_join}
(J,\boldpi)
\end{equation} 
where $J=\{j_0,\dots,j_q\}\subseteq \{0,\dots,m\}$ and $\boldpi=(\pi_{j_0},\ldots, \pi_{j_q})$ as in \eqref{eq:q-tuple-perm}; the corresponding face consists of all formal convex combinations of the form $\sum_{j\in J} \lambda_j \pi_j$, $0\leq \lambda_j \leq 1$.
\end{observation}

\paragraph{The group action.} As remarked in Section~\ref{sec:introduction} (see the discussion preceding Theorem~\ref{thm:ozaydin}), the join $(\sym_r)^{\ast (m+1)}$ is an $E_{\sym_r}^m$-space, i.e., it is an $m$-dimensional and $(m-1)$-connected space on which the group $\sym_r$ acts freely, by multiplication on the right,
\begin{equation}
\label{eq:right-action-sym_r-join}
(\lambda_0 \pi_0 + \dots + \lambda_m \pi_m)\cdot \pi= \lambda_0 (\pi_0 \pi) + \dots + \lambda_m (\pi_m \pi),
\end{equation}
for $\pi,\pi_0,\dots,\pi_m\in \sym_r$ and $\lambda_0,\dots,\lambda_m\in [0,1]$.
\medskip

There is an alternative way of looking at this space: Consider the space
\begin{equation}
\label{eq:fiber-prod}
X:=\{\boldx=(x_1,\ldots,x_r)\in \delprod{C}{r} \mid p(x_1)=\dots =p(x_r)\},
\end{equation}
on which $\sym_r$ acts by permuting components.\footnote{The definition of $X$ is closely related to the standard \emph{pullback} or \emph{fiber product}  of $r$ copies of $C$ over the common base space $\simplex_m$, except for the additional condition that we only take $r$-tuples of points supported in pairwise disjoint simplices; one might call $X$ the \emph{deleted $r$-fold fiber product} of $C$.}  The space $X$ is a simplicial complex, whose faces can be described as follows: For $0\leq q\leq m$, a $q$-dimensional simplex of $X$ is of the form
\begin{equation}
\label{eq:face-X}
\tau=\tau_1\times_p \ldots \times_p \tau_r :=\{\boldx=(x_1,\ldots,x_r) \in \tau_1\times \dots \times \tau_r \mid p(x_1)=\dots =p(x_r)\},
\end{equation}
where 
$
(\tau_1,\ldots,\tau_r)
$
is an $r$-tuple of pairwise disjoint $q$-simplices of $C$, each of which projects via $p$ onto the same $q$-dimensional face $\omega$ of the base space $\simplex^m$. 

\paragraph{Orientations.} In what follows, unless indicated otherwise, we consider the simplices $\tau_i$ and $\tau=\tau_1\times_p \ldots \times_p \tau_r$ to be \emph{oriented compatibly}, via the projection $p$ (which restricts to an isomorphism on each of these simplices) with a given orientation of the corresponding face $\omega$ of the base $\simplex^m$; such an orientation can be described in terms of an ordering of the set $J$ indexing the vertices of $\omega$ and the corresponding color classes $C_j$, $j\in J$. %Notice that the action of the group $\sym_r$ permutes the elements of each $C_j$ ``vertically'', but leaves the orientation invariant.

\begin{lemma} 
\label{lem:fiber-prod}
There is a canonical equivariant simplicial homeomorphism
\[
\Phi\colon (\sym_r)^{\ast (m+1)} \cong_{\sym_r} X
\]
which sends 
$\lambda_0 \pi_0 + \dots + \lambda_m \pi_m$ to $\boldx=(x_1,\ldots,x_r)$ given by $x_i=\sum_{j=0}^m \lambda_j v_{\pi_j(i),j}$.
\end{lemma}
\begin{proof}
For $\boldx=(x_1,\ldots,x_r) \in X$, consider the face $\omega$ of $\simplex^m$ that supports the projections $p(x_i)$,
and let $\{u_{j}\mid j\in J\}$ be the vertex set of $\omega$. We can write $p(x_1)=\dots=p(x_r)=\sum_{j\in J}\lambda_j u_j$, where 
$\lambda_j  \in (0,1)$ for $j\in J$ and $\sum_j\lambda_j$. Then each $x_i$ is supported on a $(|J|-1)$-dimensional colorful simplex $\tau_i$ with $\tau_i \cap C_j=1$ for $j\in J$; since the $x_i$ have disjoint supports, there are permutations $\pi_j \in \sym_r$, $j\in J$, defined by Equation~\eqref{eq:permutations-colorful_simplices}, such that $x_i=\sum_{j\in J} \lambda_j v_{\pi_j(i),j}$. This defines
$\Phi^{-1}(\boldx)=(J,\boldpi)$, where $\boldpi=(\pi_j\mid j\in J)$. 

It is straightforward to verify that $\Phi^{-1}$ is continuous (the $\lambda_j$ are the barycentric coordinates of each $x_i$), and 
$\Phi$ is equivariant since $x_{\pi(i)}=\sum_{j=0}^m \lambda_j v_{\pi_j(\pi(i)),j}$.
\end{proof}
\medskip
Using this configuration space, we obtain, as an analogue of Lemma~\ref{lem:delprod-necessary}, 
the following necessary condition for the existence of \emph{regular} prismatic maps without Tverberg points:

\begin{lemma}
\label{lem:nec-regular-prismatic-DPC}
Suppose $f\colon C \to \simplex^m\times \interior \simplex^k \subset \R^{rk}$ is a regular prismatic map and $h\colon C\to \interior \simplex^k\subset \R^k$ is the corresponding height function, i.e., $f(x)=(p(x),h(x))$ for $x\in C$. Consider the induced map
$$\gauss{h} \colon X\to (\R^k)^r,\qquad \gauss{h}(x_1,\ldots,x_r)=\big(h(x_1),\dots,h(x_r)\big).$$
\begin{enumerate}[label=\textup{(\alph*)}]
\item Suppose that $y \in f(\tau_1)\cap \dots \cap f(\tau_r) \subset \simplex^m \times \interior \simplex^k$ is an $r$-Tverberg
point of $f$, and that $z$ is the projection of $y$ onto $\simplex^k$ (i.e., $y=(w,z)$ for some $w\in \simplex^m$). Then
the $r$-fold intersection point $y$ corresponds to the pairwise intersection point $(z,\dots,z)$ of $\gauss{h}(\tau)$ 
with the thin diagonal $\thindiag{r}{\R^k}$, where $\tau=\tau_1\times_p \ldots \times_p \tau_r$ is the $m$-simplex of $X$ corresponding to the $\tau_i$. 

\item Moreover, up to a universal sign $\primaticsign$ depending only on $r$ and $k$, the intersection signs at these points agree,
 i.e.,
\begin{equation}
\label{eq:r-sign=pair-sign-prismatic}
\sign_{(z,\dots,z)}\big(\gauss{h}(\tau),\thindiag{r}{\R^k}\big) =\primaticsign \cdot \isign{y}{f(\tau_1)}{f(\tau_r)}.
\end{equation}

\item In particular, if $f$ has no $r$-Tverberg point, then there is an equivariant map
\begin{equation}
\gauss{h} \colon X \to_{\sym_r} (\R^k)^r \setminus \thindiag{r}{\R^k} \simeq_{\sym_r} S^{m-1}.%,
\end{equation}
\end{enumerate}
\end{lemma}

\begin{proof}
This is analogous to the proof of Lemma~\ref{lem:delprod-necessary}. It is clear that the map $\gauss{h}$ is equivariant.
Since $h$ is a prismatic map, any $r$-Tverberg point of $f$ occurs as an $r$-intersection point of pairwise disjoint $m$-simplices 
$\tau_1,\dots,\tau_r$. Moreover, since $f=(p,h)$ is regular, we have $y=f(x_1)=\dots=f(x_r)$ for $x_i \in \tau_i$, $1\leq i\leq r$,
if and only if $p(x_1)=\dots=p(x_r)$ and $z=h(x_1)=\dots=h(x_r)$, or equivalently  $(x_1,\dots,x_r)\in \tau=\tau_1\times_p \dots \times_p \tau_r \subset X$ and $(z,\dots,z)\in \gauss{h}(\tau)\cap \thindiag{r}{\R^k}$. This proves (a) and hence (c), since, as before, we have an equivariant homotopy equivalence $\rho\colon (\R^k)^r \setminus \thindiag{r}{\R^k} \simeq_{\sym_r} S^{(r-1)k-1}=S^{m-1}$.

It remains to prove (b). Since intersections signs are completed locally, it suffices to consider the case that the height function $h$ and hence $f=(p,h)$ are simplexwise affine maps, and that the intersection $f(\tau_1)\cap \dots \cap f(\tau_r)$ consists of a single point $y=(w,z)$. 
We may assume that the base $\simplex^m$ has the standard orientation given by the identity matrix $I_m $, and that the orientation 
of each affine simplex $f(\tau_i)$ is given by $\smash{\big[{A_i \atop I_m}\big]}$, where $A_i\in \R^{k\times m}$ is the matrix describing the 
linear part of the affine function $h|_{\tau_i}$. Thus, the orientation of $\tilde{h}(\tau)$ is given by the matrix $[A_1|\dots | A_r]^{\top} \in \R^{rk\times m}$, and the pairwise intersection sign of $\tilde{h}(\tau)$ and  $\thindiag{r}{\R^k}$ equals the determinant of the matrix
\begin{equation*}
\label{eq:pairwise-prismatic-sign}
B:= \left[
\begin{array}{c}
A_1 \quad I_k\\
A_2 \quad I_k\\
\vdots \\
A_r \quad I_k
\end{array}
\right] \in  \R^{rk\times rk}
\end{equation*}

Moreover, by Lemma~\ref{lem:prop-inters-prod}~(d), we have the identity
\begin{equation}
\label{eq:primatic-sign-intermediate-1}
\isign{y}{f(\tau_1)}{f(\tau_r)}=\epsilon_{r,k}\cdot \sign_{(y,\dots,y)}(f(\tau_1)\times\dots \times f(\tau_r),
\thindiag{r}{\R^d})
\end{equation}
between the $r$-fold intersection sign in $\R^d$ and the pairwise intersection sign with the thin diagonal in $(\R^d)^r$, where
$\epsilon_{r,k}$ is the universal sign introduced in \eqref{eq_epsilon}. Furthermore, the pairwise intersection sign on the right-hand side of 
\eqref{eq:primatic-sign-intermediate-1} is equal to the determinant of the matrix
\begin{equation*}
\label{eq:primatic-sign-intermediate-2}
A:=\left[
\begin{array}{cccccc}
A_1 &&&&I_k& \\
I_m &&&&& I_m \\
& A_2 &&&I_k & \\
& I_m &&&&I_m \\
&& \ddots &&&\\
&&& A_r & I_k & \\
&&& I_m & & I_m
\end{array}
\right] \in \R^{rd\times rd}
\end{equation*}
We can modify this matrix $A$, without changing its determinant, to obtain the matrices $A'$ and $A''$ described below, as follows:
First we get $A'$ by successively subtracting the columns of $A$ corresponding to each submatrix $A_i$ from the last $m$ columns.
Next, we eliminate the copies of the $A_i$ appearing in the left part of $A'$ by subtracting suitable linear combinations of the rows corresponding to the remaining copies of $I_m$. In this way, we obtain $A''$, where
\begin{equation*}
\label{eq:primatic-sign-intermediate-3}
A'=\left[
\begin{array}{cccccc}
A_1 &&&&I_k& -A_1\\
I_m &&&&& 0 \\
& A_2 &&&I_k & -A_2\\
& I_m &&&& 0  \\
&& \ddots &&&\\
&&& A_r & I_k & -A_r\\
&&& I_m & & 0
\end{array}
\right], \quad \textrm{and} \quad 
A''=\left[
\begin{array}{cccccc}
0 &&&&I_k& -A_1\\
I_m &&&&& 0 \\
& 0 &&&I_k & -A_2\\
& I_m &&&& 0  \\
&& \ddots &&&\\
&&& 0 & I_k & -A_r\\
&&& I_m & & 0
\end{array}
\right]  
\end{equation*}
Finally, by multiplying the last $m=k(r-1)$ columns of $A''$ by $-1$ and by a total of $km\binom{r+1}{2}$ row transpositions, we can transform $A''$ into
\begin{equation*}
\label{eq:primatic-sign-intermediate-4}
A'''=\left[
\begin{array}{cccccc}
I_m &&&&& \\
 &I_m &&&& \\
&& \ddots &&&\\
&&& I_m && \\
&&&& I_k & A_1\\
&&&& I_k & A_2\\
&&&& \vdots & \\
&&&& I_k & A_r\\
\end{array}
\right]  
= 
\left[
\begin{array}{cc}
I_{rm} & \\
& B
\end{array}
\right]  
\end{equation*}
Thus, 
$$\isign{y}{f(\tau_1)}{f(\tau_r)}=\epsilon_{r,k} \det A = \primaticsign \det A'''= \primaticsign \det B= \primaticsign \sign_{(z,\dots,z)}\big(\gauss{h}(\tau),\thindiag{r}{\R^k}\big),$$ 
as we set out to show, where
\begin{equation}
\label{eq:prismatic-sign}
\primaticsign := \epsilon_{r,k}\cdot (-1)^{k^2(r-1)\binom{r+1}{2}+k(r-1)}.
%\qedhere
\end{equation} 
\end{proof}

Moreover, for codimension $k\geq 3$, we will prove the following partial converse of Lemma~\ref{lem:nec-regular-prismatic-DPC}: 

\begin{theorem}[\textbf{Sufficiency of the Prismatic Deleted Product Criterion}]
\label{thm_prismatic}
Let $r\geq2$, $N=(rk+1)(r-1)$ and $m=(r-1)k$. 

If $k\geq 3$ and if there exists a $\sym_r$-equivariant map
\begin{equation}
\label{eq:equivariant_map_prismatic}
X %\join{\sym_r}{(m+1)} 
\rightarrow_{\sym_r} (\R^k)^r \setminus \thindiag{r}{\R^k} \simeq_{\sym_r} S^{m-1} 
\end{equation}
then there exists a prismatic map 
\[
C\to \simplex^m \times \interior \simplex^k
\]
without $r$-Tverberg point.
\end{theorem}

We believe that it should be possible to strengthen the conclusion of the theorem and obtain a regular prismatic map. 
However, the current form of the theorem serves our purposes and, together with \"Ozaydin's Theorem~\ref{thm:ozaydin},
implies Proposition~\ref{prop:prismatic_map_C_no_Tverberg},
and hence the existence of counterexamples to the topological Tverberg conjecture in dimension $3r$ (Theorem~\ref{thm:counterexamples}):

\begin{proof}[Proof of Proposition~\ref{prop:prismatic_map_C_no_Tverberg} using Theorem~\ref{thm_prismatic}]
Suppose $r\geq 6$ is not a prime power and $k \geq 3$. Then Theorem~\ref{thm:ozaydin} implies that there exists an
an equivariant map $X%\join{\sym_r}{(m+1)} 
\rightarrow_{\sym_r} S^{m-1}$. Consequently, by  Theorem~\ref{thm_prismatic},
there exists a prismatic map  $C\to \simplex^m \times \interior \simplex^k$ without $r$-Tverberg point.
\end{proof}

The proof of Theorem~\ref{thm_prismatic} is structured along similar lines as the proof of Theorem~\ref{thm:VK-Tverberg-complete}.

In a first step, by Theorem~\ref{thm:primary-obs}, there is a primary obstruction $\obs(X) \in H^m_{\sym_{r}}(X;\calZ)$
such that there exists an equivariant map $X \rightarrow_{\sym_r} (\R^k)^r \setminus \thindiag{r}{\R^k}$ if and only if
$\obs(X)=0$. Moreover, by Lemma~\ref{lem:nec-regular-prismatic-DPC}, any regular prismatic map $f=(p,h) \colon C\to  \simplex^m\times \interior \simplex^k$ induces an equivariant map $\gauss{h}\colon X\to (\R^k)^r $ in general position, and thus, by Lemma~\ref{lem:r-intersection-cycle}, the obstruction $\obs(X)=[\cocyc_f]$ is represented by the \emph{prismatic intersection number cocycle} $\cocyc_f$ defined 
on $m$-cells $\tau=\tau_1\times_p\dots \times_p \tau_r$ of $X$ by 
\begin{equation}
\label{eq:prismatic-intersection-cocycle}
\cocyc_f(\tau)= \tilde{h}(\tau)\scap \thindiag{r}{\R^k} = \primaticsign f(\tau_1)\scap \dots \scap f(\tau_r),
\end{equation}
where the last equality follows from \eqref{eq:r-sign=pair-sign-prismatic}. Note that, while the middle term of this equation makes sense only for \emph{regular} prismatic maps, the right-hand side is defined for arbitrary prismatic maps, and we will use this as the definition of the intersection cocycle for arbitrary prismatic maps $f$.

The main technical lemma to prove Theorem~\ref{thm_prismatic} is the following:
\begin{lemma}[\textbf{Prismatic Finger Moves}]
\label{lem:prismatic_finger_moves}
Suppose $r\geq 2$, $k\geq 1$, $m=(r-1)k$ and $N=(kr+1)(r-1)$. Suppose furthermore that $f\colon C\to  \simplex^m \times \interior \simplex^k$ is a prismatic map, that $\eta$ is an oriented $(m-1)$-simplex of $X$, and that 
$\delta \I_{\eta\cdot \sym_{r}}$ is the corresponding equivariant $m$-dimensional coboundary (see Section~\textup{\ref{sec_van_kampen_fingers_move}}). 

Then there exists a prismatic map $f'\colon C\to  \simplex^m \times \interior \simplex^k$ such that 
$$
\cocyc_{f'}= \cocyc{f}- \delta \I_{\eta\cdot \sym_{r}}.
$$
\end{lemma}

\begin{proof}[Proof of Theorem~\ref{thm_prismatic} using Lemma~\ref{lem:prismatic_finger_moves}]
We start by choosing and fixing an arbitrary \emph{regular} prismatic map $f=(p,h)\colon C\to  \simplex^m \times \interior \simplex^k$ (e.g., an affine prismatic map as described in Example~\ref{ex_prismatic_map}). By assumption, there exists an 
an equivariant map $X\rightarrow_{\sym_r} S^{m-1}$. This is equivalent to the vanishing of the primary obstruction, $\obs{X}=[\cocyc_f]=0$, which means that the prismatic intersection number cocycle $\cocyc_f$ can be written as a finite sum of elementary equivariant coboundaries. 
By repeatedly applying Lemma~\ref{lem:prismatic_finger_moves}, once for each elementary coboundary in the sum, we thus arrive at a prismatic map $f'$ such that $\cocyc_{f'}=0$ as a cocycle, i.e.,
$$f'(\tau_1)\scap \dots \scap f'(\tau_r)=0$$
for every $r$-tuple of pairwise disjoint $m$-simplices of $C$. Thus, we can arbitrarily pair up the $r$-Tverberg points in
$f'(\tau_1)\cap \dots \cap f'(\tau_r)$ into pairs of opposite sign. To conclude, we eliminate each pair by applying the $r$-fold Whitney trick,
without introducing new $r$-Tverberg points; this is possible since the codimension $d-\dim C=k$ is at least $3$.

More precisely, suppose $x,y\in f'(\tau_1)\cap \dots \cap f'(\tau_r)$ is a pair of $r$-Tverberg points of $f'$ of opposite sign. 
By the $r$-fold Whitney trick  there are are ambient isotopies $H^2,\ldots,H^r$ of $\R^d$ such that 
$$f'(\tau_1)\cap H_1^2(f'(\tau_2)\cap \dots \cap  H_1^r(f'(\tau_r) = f'(\tau_1)\cap f'(\tau_2)\cap\dots\cap f'(\tau_r)\setminus \{x,y\}.$$ 
Moreover, we can choose these isotopies to be fixed outside an open $d$-ball $B$ that avoids all other faces of $C$ and is contained in $\interior \sigma^m\times \interior \sigma^k$;
in particular, each isotopy fixes the boundary of the prism $\sigma^m\times \sigma^k$. Thus, if we define a new PL map $f''\colon C\to \sigma^m\times \interior \sigma^k$
by setting $f''(x)=H^i(f'(x))$ for $x\in \interior \tau_i$, $2\leq i\leq r$, and $f''(x)=f'(x)$ otherwise, then $f''$ is again a prismatic map and has the same $r$-Tverberg points as $f'$, except for $\{x,y\}$.
By applying this procedure a finite number of times, we arrive at a prismatic map $g\colon C\to \sigma^m\times \interior \sigma^k$ that has no $r$-Tverberg points at all. 
\end{proof}

It remains to prove Lemma~\ref{lem:prismatic_finger_moves}. This is done in the following subsection.

%%%%%%%%%%%%%%%%%%%%%%%%%%%%%%%%%%%%%%%
%%%%%%%%%%%%%%%%%%%%%%%%%%%%%%%%%%%%%%%
\subsection{\texorpdfstring{$\boldsymbol{r}$}{r}-Fold Linking Numbers and Prismatic Finger Moves}
\label{subsec:prismatic-finger-moves}
%%%%%%%%%%%%%%%%%%%%%%%%%%%%%%%%%%%%%%%
%%%%%%%%%%%%%%%%%%%%%%%%%%%%%%%%%%%%%%%

Throughout this subsection, let $r\geq 2$, $k\geq 1$, and let $m=(r-1)k$.

Suppose that $\Sigma_1,\ldots,\Sigma_r$ are $r$ PL spheres of dimension $m-1$ contained in a PL sphere $S^{rk-1}$ and in general position with respect to one another. Suppose furthermore that we have chosen orientations for each of the $\Sigma_i$ and for $S^{rk-1}$.

By Alexander duality (see, e.g., \cite[Theorem~3.44]{Hatcher:Algebraic-topology-2002}), 
$$H_{k-1}(S^{rk-1}\setminus \Sigma_r)\cong H^{m-1}(\Sigma_r)\cong \Z.$$ 
In order to fix a specific isomorphism with the integers, we fix a generator $\zeta$ of $H_{k-1}(S^{rk-1} \setminus \Sigma_r)$ as follows: 
Choose a small $k$-dimensional PL disk $D$ in $S^{rk-1}$ that intersects $\Sigma_r$ transversely in a single point, and orient $D$ 
such that this pairwise intersection point has positive sign; then $\zeta$ is represented by $\boundary D$.

By the general position assumption, $\Sigma_1\cap\dots \cap \Sigma_r=\emptyset$. The orientations of the $\Sigma_i$ induce an orientation 
of the intersection $\Sigma_1 \cap \dots \cap \Sigma_{r-1}$, as described in Section~\ref{sec:intersection_signs}. Moreover, this oriented intersection 
is a $(k-1)$-cycle (in fact, a closed $(k-1)$-dimensional PL manifold) and thus defines a homology class
\[
[\Sigma_1 \cap \dots \cap \Sigma_{r-1}]  
\in H_{k-1} (S^{rk-1} \setminus \Sigma_r) \iso \Z.
\]
\begin{definition}
Via the choice of the generator $\zeta$, we can write $[\Sigma_1 \cap \dots \cap \Sigma_{r-1}] =\linking \cdot \zeta$ for a uniquely defined integer $\linking=\linking(\Sigma_1,\dots,\Sigma_r) \in \Z$, which we call the \define{$r$-fold linking number} of 
$\Sigma_1, \dots , \Sigma_r$ in $S^{rk-1}$.
\end{definition}
We remark that the $r$-fold linking number depends on the order of the $\Sigma_i$ and on the choice of the orientations.

Next, suppose that $\sigma_1, \dots , \sigma_r$ are $r$ PL-balls of dimension $m=(r-1)k$ properly embedded in a PL ball $B^{rk}$.
Then we can apply the previous definition to the $(m-1)$-dimensional PL spheres $\Sigma_i=\boundary \sigma_i$ in $S^{rk-1}=\boundary B^{rk}$
(with the induced orientations on the boundaries).

\begin{lemma}\label{lem_hom_linking_number}
In the setting described above, the $r$-fold linking number $\linking(\boundary \sigma_1, \dots , \boundary \sigma_r)$ of the $\boundary \sigma_i$ in $S^{rk-1}=\boundary B^{rk}$ 
is equal to the algebraic $r$-fold intersection number $\sigma_1\scap \dots \scap \sigma_r$ of the $\sigma_i$ in $B^{rk}$.
\end{lemma}

\begin{proof}
The argument is similar to the one for the standard $2$-fold intersection and linking numbers (see, e.g.,  \cite[Lemma 5.15]{Rourke:Introduction-to-piecewise-linear-topology-1982}).

First, we note that the inclusion map $\iota \colon \boundary B^{rk}\setminus \boundary \sigma_r\hookrightarrow B^{rk}\setminus \sigma_r$ 
induces an isomorphism $\iota_\ast \colon H_{k-1}(\boundary B^{rk}\setminus \boundary \sigma_r) \iso H_{k-1}(B^{rk}\setminus \sigma_r)$; in particular,
$\iota_\ast(\zeta)$ is a generator of $H_{k-1}(B^{rk}\setminus \sigma_r)$. Thus, $r$-fold linking number $\linking=\linking(\boundary \sigma_1, \dots , \boundary \sigma_r)$
can be equivalently defined as the unique integer such that $[\boundary \sigma_1\cap \dots \cap \boundary \sigma_r]=\linking \cdot \iota_\ast(\zeta) \in H_{k-1}(B^{rk}\setminus \sigma_r)$.

The generator $\iota_\ast(\zeta)$ is represented by the boundary $\boundary D$ of the $k$-dimensional disk $D \subset S^{rk-1}$ used above. Alternatively, we can slightly translate this disk into the interior to obtain a small $k$-dimensional PL disk $D'$ in $\interior B^{rk}$ that intersects $\sigma_r$ transversely in a single point and that is oriented so that this pairwise intersection has positive sign; then $\iota_\ast (\zeta)=[\boundary D'] \in H_{k-1}(B^{rk}\setminus \sigma_r)$.

By Lemma~\ref{lem:associativity}, the $r$-fold intersection number $\sigma_1 \scap \dots \scap \sigma_r$ equals the $2$-fold intersection number
$\omega \scap \sigma_r$, where
$$
\omega :=\sigma_1 \cap \dots \cap \sigma_{r-1}
$$
is the oriented intersection of the first $(r-1)$ terms, which is an oriented $k$-dimensional PL manifold with boundary $\boundary \omega = \boundary \sigma_1 \cap \dots \cap \boundary \sigma_{r-1}$, properly embedded in $B^{rk}$. 

Consider an intersection point
$$y \in \omega \cap \sigma_r=\sigma_1\cap\dots \cap \sigma_r$$
with $2$-fold intersection sign $\sign_y(\omega,\sigma_r) \in \{-1,+1\}$. Choose a small $k$-dimensional disk $D_y \subset \omega$ containing $y$ in its interior,
with the orientation induced from $\omega$. Then $\sign_y(\omega,\sigma_r)=\sign_y(D_y,\sigma_r)$, and the sphere $\boundary D_y$ (with the induced orientation) represents the element  
$$\sign_y(\omega,\sigma_r) \cdot \iota_\ast(\zeta) \in H_{k-1} (B^{rk} \setminus  \sigma_r).$$

Choosing such a $k$-ball $D_y$ for each $y \in \omega \cap \sigma_r$, we can consider
\[
\omega \setminus \Big(\bigcup_{y \in \omega \cap \sigma_r} \interior D_y\Big).
\]
This is an oriented $k$-dimensional manifold with boundary and hence a $k$-dimensional chain in $B^{rk} \setminus  \sigma_r$ witnessing that the two $(k-1)$-cycles
$$
\boundary \omega= \boundary \sigma_1 \cap \cdots \cap \boundary \sigma_{r-1}
$$
and 
$$
\bigcup_{y \in \sigma_1 \cap \dots \cap \sigma_r } \boundary D_y.
$$
are homologous in $B^{rk} \setminus  \sigma_r$. Thus, they define the same homology class
$$
[\boundary \sigma_1 \cap \dots \cap \boundary \sigma_{r-1}]=\sum_{y \in \sigma_1 \cap \dots \cap \sigma_r} \sign_y(\omega,\sigma_r) \cdot \iota_\ast(\zeta) \in H_{k-1} (B^{rk} \setminus  \sigma_r).
$$
Therefore, the linking number $\linking(\boundary \sigma_1,\dots,\boundary \sigma_r)$ is equal to the intersection number $\sigma_1\scap \dots \scap \sigma_r=
\sum_{y \in \sigma_1 \cap \dots \cap \sigma_r} \sign_y(\omega,\sigma_r)$, as we set out to show.
\end{proof}

\paragraph{Modifying the $r$-fold linking number}
As before, let $\Sigma_1,\ldots,\Sigma_r$ be $r$ PL spheres of dimension $m-1$ in general position in a PL sphere $S^{rk-1}$.
We describe a down-to-earth way of changing their $r$-fold linking number by $\pm 1$.

Let $\varepsilon \in \{-1,+1\}$. Choose $(r-1)$ small PL spheres $S_1, .., S_{r-1}$ of dimension $m-1$ in embedded in general position in $S^{rk-1}$. We 
arrange the spheres and orient them in such a way that their oriented intersection
\[
S_1 \cap  \dots \cap S_{r-1} 
\]
is an oriented $(k-1)$-sphere $S$ that links precisely once with $\Sigma_r$, with the chosen sign $\varepsilon$, i.e., 
\[
[S_1 \cap \dots \cap S_{r-1}] =\epsilon \zeta \in H_{k-1} ( S^{rk-1} \setminus \Sigma_r).
\]

This embedding can be performed in a small neighbourhood of an affine piece of $\Sigma_r$ in $S^{rk-1}$. 
In particular, we chose the spheres $S_i$ so that they are disjoint from all $\Sigma_j$, $1\leq i,j\leq r-1$.

Finally, for $1\leq i\leq r-1$, we connect $\Sigma_i$ to $S_i$ by an orientation-preserving pipe (see Section~\ref{subsec_piping}), 
as in the proof of Lemma~\ref{lem:single-finger-move} to obtain a new $(m-1)$-dimensional PL sphere $\Sigma_i'=\Sigma_i\# S_i$.
By construction, this has the effect of modifying the $r$-fold linking number by $\varepsilon$, i.e.,
$$\linking(\Sigma_1',\dots,\Sigma_r')= \linking(\Sigma_1,\dots,\Sigma_r)+\varepsilon.$$

In particular, suppose that $\sigma_1,\dots,\sigma_r$ are $m$-dimensional PL balls properly containd in $B^{rk}$, and that we modify the spheres 
$\Sigma_i=\partial \sigma_i$ in $\boundary B^{rk}$ as just described. Suppose furthermore that we arbitrarily choose $m$-dimensional PL balls $\sigma_i'$ 
in $B^{rk}$ with $\boundary \sigma_i'=\Sigma_i'$ (this is always possible, e.g., by coning over $\Sigma_i'$ from the center of $B^{rk}$). Then, by Lemma~\ref{lem_hom_linking_number} the $r$-fold intersection number of the balls in $B^{rk}$ also changes by $\varepsilon$, i.e.,
\begin{equation}
\label{eq:change-intersection-number}
\sigma_1'\scap \dots \scap \sigma_r'= \sigma_1\scap \dots \scap \sigma_r + \varepsilon.
\end{equation}

We are now ready to prove the last remaining lemma.

\begin{proof}[Proof of Lemma~\ref{lem:prismatic_finger_moves}]
Let $f\colon C\to  \simplex^m \times \interior \simplex^k$ be a prismatic map, and let $\eta$ be an oriented $(m-1)$-simplex of $X$.
We know that $\eta=\eta_1\times_p \dots \times_p \eta_r$ for $r$ pairwise disjoint $(m-1)$-simplices of $C$ that project onto the same
$(m-1)$-simplex $\omega=p(\eta_1)=\dots=p(\eta_r)$ of the base $\simplex^m$ of the prism.

In analogy with the previously described way of changing linking numbers, we modify $f$ to obtain a new 
new prismatic map $f'\colon \colon C\to  \simplex^m \times \interior \simplex^k$ as follows:
\begin{itemize}
\item We select $r-1$ small oriented PL spheres $S_1, ..., S_{r-1}$ of dimension $m-1$ in general position in $\omega \times \interior \simplex^k$;
we choose these sphere so that their intersection $S_1 \cap \dots \cap S_{r-1}$ is a flat $(k-1)$-dimensional PL sphere $S$ ``linking'' with
$f(\eta_r)$ exactly once and with negative sign, i.e., if we fill this sphere with $k$-dimensional PL ball then this ball intersects $f(\eta_r)$ exactly once, 
with negative intersection sign.
\item For $1\leq i\leq r-1$, we connect $f(\eta_i)$ to $S_i$ by an orientation-preserving pipe to create a new $(m-1)$-dimensional ball in $\omega \times \interior \simplex^k$ with the same boundary as $f(\eta_i)$. 
\item We define $f'$ to agree with $f$ on all faces of $C$ of dimension less than $m-1$ and on all $(m-1)$-simplices of $C$ except for $\eta_1,\dots,\eta_{r-1}$. 
On $\eta_i$, $1\leq i\leq r-1$, we define $f'$ so that $f'(\eta_i)$ equals the result of piping $f(\eta_i)$ with $S_i$ (this possible, since $f(\eta_i)$ and the result of the piping are two PL balls in $\omega \times \simplex^k$ with the same boundary).\footnote{For $k\geq 3$, there even exists an ambient homotopy $H^i$ of $\omega \times \simplex^k$, fixed on the boundary, such that we can take $f'|_{\eta_i}=H^i_1\circ f|_{\eta_i}$, but we will not need this.} 
\item Finally, let $\tau$ be an $m$-dimensional simplex of $C$. If $\tau$ does not contain any one of the simplices $\eta_1,\ldots,\eta_{r-1}$ on its boundary,
then we define $f'|_{\tau}=f|_\tau$. Otherwise, we redefine $f'$ on $\tau$ so that $f'(\tau)$ is an $m$-dimensional ball properly contained in $\sigma^m\times \sigma_k$;
this is always possible, for instance by coning over $f'(\boundary \tau)$ from a point in general position in the interior of $\simplex^m \times \simplex^k$.
\end{itemize}

It is clear that the resulting map $f'$ is prismatic. We claim that its prismatic intersection number cocycle satisfies
$$
\cocyc_{f'}= \cocyc_{f}- \delta \I_{\eta\cdot \sym_{r}}.
$$
To see this, consider an $m$-simplex $\tau_1\times_p\dots \times_p\tau_r$ of $X$ corresponding to an
$r$-tuple of pairwise disjoint $m$-simplices $\tau_1,\ldots,\tau_r$ of $C$. 
Up to the universal sign $\primaticsign$, the value of $\cocyc_{f'}(\tau_1\times_p\dots \times_p\tau_r)$ euqls the 
intersection number $f'(\tau_1)\scap \dots \scap f'(\tau_r)$ in the $rk$-ball $\simplex^m\times \sigma^k$, or equivalently, the linking
number $\linking(f'(\boundary \tau_1),\dots, f'(\boundary\tau_r))$ in $\boundary (\simplex^m\times \sigma^k)$. 

If there is one $\tau_j$ that contains none of the $\eta_i$ in its boundary, then
$$\linking(f'(\boundary \tau_1),\dots, f'(\boundary\tau_r))=\linking(f(\boundary \tau_1),\dots, f(\boundary\tau_r))$$
is unchanged.

Otherwise, up to a permutation of the indices, we may assume that $\eta_i$ is contained in the boundary of $\tau_i$, $1\leq i\leq r$. 
In this case, as discussed above, the piping of $\eta_i$, $1\leq i\leq r-1$ has the effect that
$$\linking(f'(\boundary \tau_1),\dots, f'(\boundary\tau_r))=\linking(f(\boundary \tau_1),\dots, f(\boundary\tau_r))-1,$$
i.e., $\cocyc_{f'}(\tau_1\times_p\dots \times_p\tau_r) - \delta \I_{\eta\cdot \sym_{r}}(\tau_1\times_p\dots \times_p\tau_r)$.
By equivariance, the same is true if $\eta_i$ is contained in the boundary of $\tau_{\pi(i)}$, $1\leq i\leq r$. This proves the claim and hence the lemma.
\end{proof}
This also completes the proofs of Theorems~\ref{thm_prismatic} and \ref{thm:counterexamples}.

\bibliographystyle{abbrv}
\bibliography{EliminatingMultiplePoints}

\end{document}